\documentclass[10pt,a4paper,reqno,oneside]{amsart}

\usepackage{supertabular}
\usepackage{array}
\usepackage{url}
\usepackage{multicol}
\usepackage[linktocpage = true]{hyperref}
\usepackage{multirow}

\hypersetup{
 colorlinks = true,
 citecolor = blue,
}
\usepackage{color}
\definecolor{red}{rgb}{1,0,0}
\definecolor{green}{rgb}{0,1,0}
\definecolor{blue}{rgb}{0,0,1}
\definecolor{refkey}{gray}{.625}
\definecolor{labelkey}{gray}{.625}
\usepackage{extarrows}
\usepackage[lite]{amsrefs}
\usepackage[a4paper]{geometry}
\usepackage{amsfonts}
\usepackage{amsmath,amssymb}
\usepackage{tablists}
\restorelistitem    
\usepackage{amssymb}
\usepackage{mathrsfs}
\usepackage{amsmath}
\usepackage{amsthm}
\usepackage{amscd}
\usepackage{enumerate}
\usepackage{paralist}
\usepackage{graphicx}
\usepackage{mathtools}
\usepackage[all,cmtip]{xy}
\usepackage{kantlipsum}
\usepackage{tikz}
\usetikzlibrary{trees}

\usepackage{tikz-cd}
\allowdisplaybreaks

\def\F{\mathbb F}

\def\Aut{\mathrm{Aut}}

\def\GL{\mathrm{GL}}
\def\Gal{\mathrm{Gal}}

\def\dim{\mathrm{dim}}
\def\tr{\mathrm{Tr}}

\newcommand{\Nqfrac}[1]{\mathrm{Q}{(#1)}}

\def\Ker{\mathrm{ker}}

\makeatletter

\theoremstyle{plain}
\newtheorem{thm}{\protect\theoremname}[section]
\theoremstyle{plain}
\newtheorem{prop}[thm]{\protect\propositionname}
\theoremstyle{plain}
\newtheorem{cor}[thm]{\protect\corollaryname}
\theoremstyle{plain}
\newtheorem{lem}[thm]{\protect\lemmaname}
\theoremstyle{defn}
\newtheorem{example}[thm]{\protect\examplename}
\theoremstyle{defn}
\newtheorem{defn}[thm]{\protect\definitionname}
\theoremstyle{plain}
\newtheorem{remark}[thm]{\protect\remarkname}
\theoremstyle{defn}
\newtheorem{notation}[thm]{\protect\notationname}
\AtBeginDocument{
	
}

\makeatother

  \providecommand{\corollaryname}{Corollary}
  \providecommand{\examplename}{Example}
  \providecommand{\lemmaname}{Lemma}
  \providecommand{\propositionname}{Proposition}
  \providecommand{\theoremname}{Theorem}
  \providecommand{\definitionname}{Definition}
  \providecommand{\remarkname}{Remark}
  \providecommand{\notationname}{Notation}

  \makeatletter
  \newcommand{\be}{%
  \begingroup
  \eqnarray%
   \@ifstar{\nonumber}{}%
  }

\newcommand{\flag}[1]{( #1 )}
\newcommand{\coveringdeg}{\mathrm{cdeg}}
\newcommand{\matrixdeg}{\mathrm{mdeg}}

\newcommand{\Ttorsiontree}{\mathcal{T}}

\newcommand{\matrixtree}{\mathcal{M}}

\newcommand{\set}[1]{\{#1\}}
\newcommand{\Lbar}{\overline{L}}
\newcommand{\Fq}{\F_q}
\newcommand{\phiTu}[1]{\phi_T^{(#1)}}
\newcommand{\phinoTu}[1]{\phi^{(#1)}}
\newcommand{\psinoTu}[1]{\psi^{(#1)}}
\newcommand{\thetau}[1]{\theta^{(#1)}}
\newcommand{\lambdau}[1]{\lambda^{(#1)}}
\newcommand{\betau}[2]{\beta^{(#1)}{(#2)}}
\newcommand{\gammau}[2]{\gamma^{(#1)}{(#2)}}
\newcommand{\kappau}[2]{\kappa^{(#1)}{(#2)}}

\newcommand{\betausingle}[1]{\beta^{(#1)} }
\newcommand{\gammausingle}[1]{\gamma^{(#1)} }
\newcommand{\kappausingle}[1]{\kappa^{(#1)} }

 \newcommand{ \hatbetau}[1]{ \hat{\beta}^{(#1)} }
\newcommand{ \hatgammau}[1]{ \hat{\gamma}^{(#1)} }
\newcommand{ \hatkappau}[2]{ \hat{\kappa}^{(#1)}(#2) }
\newcommand{ \hatxiu}[1]{ \hat{\xi}^{(#1)} }
\newcommand{ \hatetau}[1]{ \hat{\eta}^{(#1)} }
\newcommand{\Rconjugacyclass}[1]{[R_{#1}]}

\newcommand{\xiu}[2]{\xi^{(#1)}{(#2)}}
\newcommand{\etau}[2]{\eta^{(#1)}{(#2)}}
\newcommand{\xiusingle}[1]{\xi^{(#1)} }
\newcommand{\etausingle}[1]{\eta^{(#1)} }

\newcommand{\equalbyreason}[1]{\xlongequal[]{\mbox{#1}}}
\newcommand{\NODE}{\mathfrak{N}}
\newcommand{\Trank}{\mathrm{Trd}}
\usepackage{float}
\newcommand{\tohandle}[1]{\textcolor{red}{#1}}
\newcommand{\corepoly}[1]{\psi_T^{(#1)}}
\newcommand{\Knaught}{K_0}
\newcommand{\Knaughtbar }{\overline{K_0}}

\newcommand{\phignaught}{\Phi_T^{(g)}}
\newcommand{\phignaughtnoT}{\Phi^{(g)}}
 \newcommand{\BasicT}{\Phi^{(G)}_T}
 
 \newcommand{\Basic}{\Phi^{(G)}}
\newcommand{\OV}[1]{\mathcal{O}_{\mathcal{V}_{#1}}}
\newcommand{\iotag}{{\iota_g}}
\newcommand{\minimalpoly}[1]{\mathfrak{m}_{#1}}
\newcommand{\Exp}{\mathrm{Exp} }
\makeatother

\title[Drinfeld Modular Curves and Nilpotent Upper-Triangular Matrices]{Drinfeld Modular Curves Subordinate to Conjugacy Classes of Nilpotent Upper-Triangular Matrices}
\thanks{ Research partially supported by NSFC grants 12071241(Chen),
	12071247/12101616(Hu), and 11961049(Zhang).}
\thanks{$^*$ The corresponding author.}

\author{Zhuo Chen$^*$}
\address{Department of Mathematical Sciences, Tsinghua University, Beijing, China}
\email{\href{chenzhuo@tsinghua.edu.cn}{chenzhuo@tsinghua.edu.cn}}

\author{Chuangqiang Hu}
\address{Yanqi Lake Beijing Institute of Mathematical Sciences and Applications, Beijing, China}
\email{\href{huchq@bimsa.cn}{huchq@bimsa.cn}}

\author{Tao Zhang}
\address{College of Mathmatics and Information Science, Henan Normal University, Xinxiang, China}
\email{\href{zhangtao@htu.edu.cn}{zhangtao@htu.edu.cn}}

\author{Xiaopeng Zheng}
\address{School of Mathematical Sciences, University of Chinese Academy of Sciences, Beijing, China}
\email{\href{zhengxiaopeng@amss.ac.cn}{zhengxiaopeng@amss.ac.cn}}

\allowdisplaybreaks
\begin{document}
\begin{abstract} 
We introduce normalized Drinfeld modular curves that parameterize rank $m$ Drinfeld modules compatible with a  $T$-torsion structure arising from a given conjugacy class of nilpotent upper-triangular $n\times n$ matrices with  rank $\geqslant n-m$ over a finite field $\Fq$. This creates a deep link connecting the classification of nilpotent upper-triangular matrices and the decomposition of Drinfeld modular curves. The conjugacy classes of nilpotent upper-triangular matrices  one-to-one corresponds to certain $T$-torsion flags, and form a tree structure. As a result, the associated Drinfeld modular curves are organized in the same tree. This generalizes the tower structure introduced by Bassa, Beelen, Garcia, Stichtenoth, and others. Additionally,
we prove the geometric irreducibility of   $(3,2)$-type normalized Drinfeld modular curves, and characterize their associated function fields.	
	
\end{abstract}

\maketitle{}



\tableofcontents
\section{Introduction}
\bigskip \subsection{Background} 
  In the 1970's,      Drinfeld introduced the notion of elliptic modules which are now known as Drinfeld modules \cite{Drinfeld1974}, and during the last fifty years, its   theory   is  one of the most important developments in function field arithmetic. Roughly speaking, Drinfeld modules   are the positive function field 	analogues of elliptic curves, and   the modular
theory of elliptic curves is transported to the function field case. 
Let us provide a brief non-technical explanation of what are  Drinfeld modular curves in
 function field theory,  and in particular, towers of such curves  that we are talking about.



Let $\Fq$ denote the finite field of order $q$.  A tower   of curves over $\Fq $  consists of a  family of curves $C_n$ together with a sequence of successive surjective maps:
\[
\begin{tikzcd}
C_1&C_2\arrow[l,"p_1"']&\arrow[l,"p_2"']\cdots&C_{n-1}\arrow[l,]&C_n\arrow[l,"p_{n-1}"']&
\cdots\arrow[l,"p_n"']
\end{tikzcd}
\]
such that all $C_n$ and $p_n$ are defined over $\Fq $ and
the genus   $g(C_n)\to \infty$  as $n\to \infty$. The tower is called \textbf{recursive} by an absolutely irreducible polynomial $f(x,y)\in\Fq (x)[y]$ (see \cite{H.Stichtenoth2009}*{Sections 3.6 and 7.2}) if
\begin{enumerate}
	\item  The initial curve $C_1$ is the projective line with coordinate ${x}_1$; and 	
	\item For every $n\geqslant 2$,
	$C_n$ is the nonsingular projective model of an affine   curve defined by
	\[f({x}_1,{x}_2)=f({x}_2,{x}_3)=\cdots=f({x}_{n-1},{x}_n)=0. \]
\end{enumerate}




Suppose that $m=j+k\geqslant 2$ is a positive integer,  where $ j$ and $ k $ are coprime positive integers. Let $a$ and $b$ be   non-negative integers such that $ak-bj=1$.  Consider the  tower  $\mathcal{F}$ (over $\F_{q^m}$, for $m\geqslant 2$), respectively  $\mathcal{H}$,  arising from the recursive polynomial
\begin{equation}\label{Eq:TBBGSF}
{\mathcal{F}}(u,v)=\tr_j\left(\frac{v}{u^{q^k}}\right)+\tr_k\left(\frac{v^{q^j}}{u}\right)-1,\end{equation}
respectively
\begin{equation}\label{Eq:TBBGSH}
{\mathcal{H}}(u,v)=\frac{\tr_j(v)-a}{\tr_j(u)^{q^k}-a}-\frac{\tr_k(v)^{q^j}-b}{\tr_k(u)-b},
\end{equation}
where $\tr_l(x):=\sum_{i=0}^{l-1}x^{q^i}$.

 Bassa, Beelen, Garcia, and Stichtenoth \cites{Bassa2015,A.Bassa2014}  presented a  framework of  recursive tower (BBGS tower for short) of curves over non-prime fields      and outlined a modular interpretation of the defining equations. Later, by Anbar, Bassa, and Beelen \cite{Nurdagul2017},   details of the particular tower $\mathcal{H}$ in \eqref{Eq:TBBGSH}  is elaborated and proved to be modular (with $(m,j,k)=(3,2,1)$), following the  techniques of   \cite{Elkies2001}   developed by Elkies.
For  general $(m,j,k)$, our recent work \cite{MR4203564}
  pointed out explicit modular explanations  of the BBGS towers $\mathcal{F}$ and $\mathcal{H}$ in, respectively,  \eqref{Eq:TBBGSF} and \eqref{Eq:TBBGSH}. It is   shown \emph{op.cit.} that these include both the
  Garcia-Stichtenoth and BBGS towers   as special cases.

 Here is a quick review of the mechanism that produces  BBGS towers.  Roughly speaking, a Drinfeld modular curve   (with $ T^n $-level structure) parameterizes   Drinfeld modules associated with  $T^n$-torsion points which form an     $\mathbb{F}_q[T]$-module isomorphic to $ \mathbb{F}_q[T]/T^n $.   In practice,  we  are   interested in     sparse  Drinfeld modules of rank $m(\geqslant 2)$, i.e., those of the form\footnote{Modular curves of this sort were
	studied systematically by E.-U.Gekeler in   \cite{Gekeler2019}.}
\begin{equation*}
\phi_{T}=g_m \tau^m +  g_j \tau^j + 1\,.
\end{equation*}     The  rank $1$ torsion modules $ \mathbb{F}_q[T]/T^n $ are connected  naturally:
\begin{equation}
\label{Eqt:rank1torsion}
0\leftarrow \Fq[T]/T\leftarrow \Fq[T]/T^2\leftarrow \Fq[T]/T^3\leftarrow \cdots \cdots.
\end{equation}
This sequence interprets the point-wise covering maps of the corresponding BBGS tower as   a chain of Drinfeld modular curves.

\bigskip \subsection{Motivation}

A significant difference between  Drinfeld modules and elliptic curves is that  \textit{the rank of a torsion structure} of an elliptic curve must be $2$, whereas that of a Drinfeld module could be greater than or equal to $ 2$.
Thereby, it is natural to ask whether or not one can define modular curves to parameterize
Drinfeld modules together with higher rank torsions, and if so, it is also tempting to find other ways organizing these modular curves   to generalize the tower  structure as  reviewed above.
  Besides the rank $1$ torsions as in \eqref{Eqt:rank1torsion}, we further consider the relations between torsion modules of different ranks. 
  Our new point of view is that  there are torsion structures with higher rank $m(\geqslant 2)$ for Drinfeld modules which have not been particularly noted in    the literature.

  Along the lines of towers of modular curves and the tips from our last work \cite{MR4203564}, the overall aim of this paper is to develop a more general framework   of Drinfeld modular curves in connection with certain level structure, and to   explain
  the immense relevance of such structure from both   conceptual and
  practical points of view.
  In fact,  we reveal a family of Drinfeld modular curves which are organized in an elegant manner ---   a  hierarchical topology tree which we call the  \textbf{$T$-torsion tree} (of certain rank); see Section \ref{Sec:basic}. It turns out that,
    being the key ingredient,     the  $T$-torsion tree is closely related to the classification problem of nilpotent upper-triangular matrices  over  $\Fq$   (hence we will   use directly many  results    from matrix classification studies   \cites{MR577913, MR123605,MR113948,MR470098,MR1190401,MR1260823,MR1448353}).

As demonstrated in \cite{Bassa2015}, the defining equation $\mathcal{F}$ in \eqref{Eq:TBBGSF} of the BBGS tower comes from a factor of \begin{equation} \label{BBGSXYcurve} \mathcal{M}(u,v):=\frac{v^{q^{k+j}}-v}{v^{q^j}}-\frac{u^{q^{k+j}}-u}{u^{q^{k+j}-q^j+1}}. \end{equation} Note that the equation $\mathcal{M}(u,v)=0$ is not irreducible (when regarded as a curve).  So it is natural to explore other irreducible factors of $\mathcal{M}$. Our work essentially focuses on the  $(m,j,k)=(3,2,1)$-case. In fact, we will transform $\mathcal{M}$ to a polynomial $\kappausingle{u}$ which we call the modular polynomial\footnote{Our modular polynomial is the reduced version of the one in  \cite{MR1331765}.}, and $\mathcal{F}$ an irreducible factor of it (see Section \ref{Sec:32flagclassfunctionfield}). It turns out that all irreducible factors of $\kappausingle{u}$ correspond to Drinfeld modular curves that are compatible with specific $T$-torsion structures  (our notion is that they are subordinate to nodes of the $T$-torsion tree).    In this sense, the tree-based approach is a generalized form of the BBGS tower, revealing the geometric structure of other towers as well  (including the tower constructed by Anbar, Bassa, and Beelen \cite{Nurdagul2017}).





\bigskip \subsection{Contents} 
Now let us describe in more detail the structure and results of the current paper. 

In Section \ref{Sec:Pre-NUTMtree}, we
  recall some basic definitions and results about the space  $\mathbf{U}^o_n$ of nilpotent upper-triangular matrices
      (with size $(n\times n)$ and over $\Fq$),
   and the space $\mathbf{U}^o_{n }/_\sim $ of their conjugacy classes.  In particular, the tree   of   nilpotent upper-triangular matrices (denoted by $\matrixtree $)  is discussed in detail.  In short, the root node of $\matrixtree$ is the trivial class   $\Rconjugacyclass{1}= [(0)]  $ of level $1$;  level-$n$ nodes are    conjugacy classes  in $\mathbf{U}^o_{n }/_\sim $  for    $n=2,3,\cdots$; and each child node $[N^\prime]$ of level $(n+1)$ connects to its parent node $[N]$ of level $n$ where $N\in \mathbf{U}^o_{n }$ is the left-top-truncation of $N^\prime\in  \mathbf{U}^o_{n+1 }$. A level $n$ node of $\matrixtree$ is denoted by $[R_{i_1,\cdots,i_n}]$ where  the index   $(i_1,\cdots, i_n) $ is particularly designed to trace its position in $\matrixtree$.

    In Figure \ref{fig:NUTMtree},  we have drawn the 24 nodes of $\matrixtree$ with levels $n=1,2,3,4$ according to the result of \cite{MR1448353}.  As the level $n$ gets higher,  the number of nodes increases rapidly. For example, on level $n=5$, $\matrixtree$ has 61 nodes.

     In Section \ref{Sec:basic}, we introduce $T$-torsion flags, $T$-torsion flag classes, and the tree    formed by $T$-torsion flag classes. Let $A$ be the polynomial ring $\Fq [T]$. Then $S^1:=\Fq [\frac{1}{T}]/ \Fq$ and $      {S^m}:=(S^1)^{\oplus m}$ (for some fixed integer $m\geqslant 2$ which eventually coincides with the rank of Drinfeld modules we need) are  $A$-modules in an obvious sense.  	A level $n$  {$T$-torsion  flag}   in $S^m$ is a flag  of subspaces $\mathcal{V}=\flag{  V_1\subset V_2\subset \cdots\subset V_n}$ in ${S^m}$  	such that each $V_i$ is of dimension $i$  and satisfies   $TV_i\subseteq V_{i-1}$  ($TV_1=0$ in particular). In other words, each $V_i$ is also an $A$-module. We say that two $T$-torsion $n$-flags      in $S^m$ are    {isomorphic} if there is an $A$-module isomorphism between them mapping the relevant subspaces bijectively.   The set of isomorphic $T$-torsion $n$-flag classes  gives rise to a tree $\Ttorsiontree^m $ that we call the \textbf{ $T$-torsion  tree} --- 
     The (level $1$) root node    of $\Ttorsiontree^m$ is of the form $[\Fq v]$ where $v $ is a nonzero vector in  $S^m$ with $Tv=0$; each (level $(n+1)$) node $[\mathcal{V}']$ where $\mathcal{V}'=\flag{  V_1\subset V_2\subset \cdots\subset V_n\subset V_{n+1}}$   connects to its parent (level $n$) node  $[\mathcal{V}]$ with  $\mathcal{V}=\flag{  V_1\subset V_2\subset \cdots\subset V_n}$. Similar to that of $\matrixtree$, a level $n$ node of $\Ttorsiontree^m$ is denoted by $\NODE_{i_1,\cdots,i_n}$ with the same pattern of  index.

   We find that the   tree  $\matrixtree$ of nilpotent upper-triangular matrices and the tree of $T$-torsion flags $\Ttorsiontree^m$  are related in a simple way ---     A node $[R_{i_1,\cdots,i_n}]$  of $\matrixtree$ induces a node $\NODE_{i_1,\cdots,i_n}$ of $\Ttorsiontree^m$ if and only if the   matrix rank of $R_{i_1,\cdots,i_n}$ is equal to or greater than $   (n-m) $ (Proposition \ref{Prop:Nnm-Sn}).
      In other words, $\Ttorsiontree^m $ can be obtained, topologically, by removing those level $n$ nodes of   $\matrixtree$ whose rank is less than $(n-m)$. In general, both $\matrixtree$ and $\Ttorsiontree^m$ have infinite levels and nodes. But up to    level $4$, they are relatively clear.  We have drawn   $\Ttorsiontree^m$ up to level $4$ in Figure   \ref{fig:Ttree} which has   24 nodes when $m\geqslant 4$. If $m=3$, it has 23 nodes (by deleting the last one). If $m=2$, there are only 16 nodes left.

In Section \ref{Sec:Standardcorrespondence}, we establish a key tool to study Drinfeld modular curves, namely, the standard correspondence between  $n$-arrays in $\Lbar^*$ and $n$-flags in $\Lbar$.  Here we let $L$ be a field over $\Fq$ together with an homomorphism of algebras $\iota : A \to L$, and $\Lbar$ the algebraic closure of $L$. We shall refer to such $L$ as an $A$-field. By saying the standard correspondence, we mean a pair of   bijections	$$\Lambda: ~\set{n\mbox{-flags in }\Lbar}\rightleftarrows (\Lbar^*)^{\times n}:~\Theta $$ (one being the inverse to the other). Thus, we are able to transfer a geometric object, namely an $n$-flag 	$\mathcal{V}  =\flag{ V_1\subset V_2\subset \cdots\subset V_n}$ in $\Lbar$, to a sequence of numbers $\Theta(\mathcal{V})=(u_1,\cdots, u_n)$ (in $\Lbar^*$). This correspondence is indeed  very useful  as we will use  the data   $(u_1,\cdots, u_n) $   to  parameterize Drinfeld modules.

 After these purely mathematical preliminaries, we focus in Section \ref{Sec:Drinfeldmodulescurves}   on  Drinfeld modules and Drinfeld modular curves,   the core of our paper.   This will be
particularly useful to fix our notations and also to make the article sufficiently self-contained.  Roughly speaking, a Drinfeld module $ \phi $ over $L$ (an $A$-field) is identically a twisted polynomial $\phi_T$ over $L$, i.e.,
\begin{equation*}
\phi_T =   g_m \tau^m +\cdots + g_2 \tau^2 + g_1 \tau + g_0\,,
\end{equation*}
where $g_i\in L$, $g_0=\iota(T)$, and $ g_ m$  is not zero for some integer $m>0$ (called the rank of $\phi$). So $\phi_T$ acts on $L$ canonically, and on $\Lbar$ as well. Then, we can talk about    \textbf{$\phi_T$-torsion $n$-flags} in $\Lbar$, which are level $n$ flags of the form $\mathcal{V}=(V_1\subset V_2\subset \cdots\subset V_n)$    such that each $V_i\subset \Lbar$ is of dimension $i$  and satisfies   $\phi_T (V_i)\subseteq V_{i-1}$  ($\phi_T (V_1)=0$ in particular).

 If $T$ is coprime to the kernel of $\iota$, then a well-known theorem guarantees that the $T$-divisible group
 $$
 S^{\phi}:=  \varinjlim_n \Ker(\phi_{T^n})\subset \Lbar$$ is isomorphic to  $ {S^m}$ (as an $A$-module; see Lemma \ref{Lem:moduleiso}) which we considered earlier.  Thereby,  any $\phi_T$-torsion $n$-flag $\mathcal{V}$ in $\Lbar$   indeed lives in $S^{\phi}$, and via the said isomorphism, $\mathcal{V}$ is   mapped to a $T$-torsion $n$-flag in  $S^m$. Now, we consider  the central   object of study in this paper --- \textbf{pairs  $(\phi, \mathcal{V})$} where
 \begin{itemize}
 	\item[$\bullet$] $\phi$ is a  $(m,j)$-type normalized
 	Drinfeld module:   
 	$
 		\phi_{T} =  -\tau^m + g \tau ^j + 1
 	$, with $ m$ and $ j$ being coprime;
 \item[$\bullet$] $\mathcal{V}$ is a  $\phi_T$-torsion $n$-flag  $\mathcal{V} $ in $\Lbar$ which is  defined  over $L$ (Definition \ref{Def:definedoverL})
 \end{itemize}
 such that the image of $\mathcal{V}$ in $S^m$ represents a given node $\NODE_{i_1,\cdots, i_n }$ of the $T$-torsion  tree $\Ttorsiontree^m$, namely a $T$-torsion $n$-flag class. The role of $\mathcal{V}$ is to encode some ``level structure'' as in the classical studies of elliptic curves.

  Sections \ref{Sec:Drfeldparameters} and \ref{Sec:mjtypeDrinfeldmodularcurves}    provide  descriptions of two types of Drinfeld modular curves,  as new concepts introduced in this paper. (1) An algebraic curve whose $L$-points   parameterize the above 
pairs $(\phi, \mathcal{V})$    is called  the \textbf{normalized Drinfeld modular curve subordinate to} $\NODE_{i_1,\cdots, i_n }\in \Ttorsiontree^m$. Let us use $\dot{X}^{(m,j)}_{ {i_1,\cdots, i_n }  }$  to denote it. (2) Two   pairs  $(\phi,\mathcal{V} )$ and $(\tilde{\phi},\tilde{\mathcal{V}}  )$ are said to be equivalent if they are isomorphic by the multiplication of   some  $\chi\in \Lbar ^*$;  an algebraic curve
whose $L$-points  parameterize      equivalent classes of pairs
$ (\phi, \mathcal{V} )$  is called the    \textbf{Drinfeld modular curve  subordinate to} $\NODE_{i_1,\cdots, i_n }$  (with notation $ {X}^{(m,j)}_{ {i_1,\cdots, i_n }  }$).
 For more details, see Definitions \ref{Def:normalizedDrinfeldmodularcurve} and \ref{Def:Drinfeldmodularcurve} respectively.

Indeed, the two definitions as described above  are  inspired by the towers  of curves studied in \cites{Nurdagul2017,Bassa2015,A.Bassa2014,MR4203564} that we   recalled. In Examples \ref{Ex:1} and \ref{Ex:2}, we consider some known Drinfeld modular curves and explain what special nodes they are subordinate to from our new point of view. 
 If for all nodes $\NODE $ of the tree $\Ttorsiontree^m$, we find the associated modular curves ($ \dot{X}^{(m,j)}_{ {i_1,\cdots, i_n }   }$ or $ {X}^{(m,j)}_{ {i_1,\cdots, i_n }  }$), then they are lined up according to the topology of   $\Ttorsiontree^m$.
 In general, along any sequence of parent-child-linked nodes of $\Ttorsiontree^m$, the resulting (normalized or reduced) Drinfeld modular curves   comprise a tower, and each curve  subordinate to a child node $\NODE_{i_1,\cdots, i_n,i_{n+1} }$ projects to the curve subordinate to its parent node $\NODE_{i_1,\cdots, i_n }$. Our Theorem \ref{prop:coveringdegree1} tells the extension degree of function fields of the Drinfeld modular curves subordinate to   $\NODE_{i_1,\cdots, i_n,i_{n+1} } $ over that of   $\NODE_{i_1,\cdots, i_n  } $:
 	$$[{ \dot{\mathcal{F}} }^{(m,j)}_{ {i_1,\cdots, i_n,i_{n+1} } }:  { \dot{\mathcal{F}} }^{(m,j)}_{ {i_1,\cdots, i_n  } }]=[{ \mathcal{F} }^{(m,j)}_{ {i_1,\cdots, i_n,i_{n+1} }  }:  { \mathcal{F} }^{(m,j)}_{ {i_1,\cdots, i_n  } }]=\coveringdeg_{\NODE_{i_1,\cdots, i_n,i_{n+1} }   }.$$
   Here $\coveringdeg_{\NODE_{i_1,\cdots, i_n,i_{n+1} }   }$ stands for the covering degree of the node $\NODE_{i_1,\cdots, i_n,i_{n+1} } $, an integer only depending on the  linear structure of $T$-torsion flags --- see Definition \ref{Defn:coveringdegree}.
  More properties of this kind of towers will be investigated in the future.

Section \ref{Sec:32flagclassfunctionfield}    is devoted to study $(3,2)$-type Drinfeld modular curves. To this end, we consider the special $A$-field $\Knaughtbar  $, the algebraic closure of $\Knaught:={\Fq(G)}$ where $G$ is a formal variable.   Moreover, we fix a special $(3,2)$-type normalized Drinfeld module  
 $$\BasicT =  -\tau^3 + G \tau ^2 + 1   .$$ 	Let    $\NODE_{i_1,\cdots, i_n}$ be a $T$-torsion flag class (in $ S^3$), namely a node of $\Ttorsiontree^3$. Take  
 any $\BasicT$-torsion flag   $\mathcal{V} = (V_1\subset\ldots \subset V_n)$ (in $\Knaughtbar $) which is    {subordinate to}   $\NODE_{i_1,\cdots, i_n}$.  We shall define   two important objects:
 \begin{itemize}
 	\item[$\bullet$] A function field $\mathcal{F}_{{\mathcal{V}}} :=  \Fq(u_1,\ldots,u_n)$, where  $(u_1,\ldots,u_n) = \Lambda({\mathcal{V}})\in (\Knaughtbar ^*)^{\times n}$ is given by the standard correspondence mentioned above.  
  We  call       $\mathcal{F}_\mathcal{V}$ the ($(3,2)$-type) \textbf{function field}   of   $\mathcal{V}$     or the associated flag class  $\NODE_{i_1,\cdots, i_n}$.  In fact, we can prove that $\mathcal{F}_\mathcal{V}$  is, up to isomorphisms, solely determined by   the   $T$-torsion $n$-flag class $\NODE_{i_1,\cdots, i_n}$ to which $\mathcal{V}$ is subordinate  (see Proposition \ref{prop:isomorphismFVbetaV});
 	\item[$\bullet$] A polynomial $\minimalpoly{\mathcal{V}}\in \Knaughtbar  [X]$ which we call the ($(3,2)$-type) \textbf{minimal polynomial}  of   $\mathcal{V}$     or the associated flag class  $\NODE_{i_1,\cdots, i_n}$. Although its definition is kind of tricky, see Equation \eqref{beta}, $\minimalpoly{\mathcal{V}}$ turns out to be   in $\mathcal{F}_{\mathcal{V}_{n-1}}[X]$ (where $\mathcal{V}_{n-1}$ is the parent flag of $\mathcal{V}$) and the minimal polynomial that generates the function field $\mathcal{F}_{\mathcal{V}}$ 	
 	(see Proposition  \ref{Prop:main}) over $\mathcal{F}_{\mathcal{V}_{n-1}}$.

 \end{itemize}
  
The two Theorems \ref{Thm:thehardisomorphism} and \ref{Thm:NodeToFactor} in  Section \ref{Sec:32flagclassfunctionfield}, also the main results of this paper, characterize $(3,2)$-type normalized Drinfeld modular curves from two perspectives:
\begin{itemize}
	\item[(1)] (Theorems \ref{Thm:thehardisomorphism} ---) \textit{Given a $T$-torsion flag class   $\NODE_{i_1,\cdots, i_n}$ (in $ S^3$), the function field    of the $(3,2)$-type normalized Drinfeld modular curve $\dot{X}^{(3,2)}_{i_1,\cdots,i_{n}  }$ over $\Fq$ 		is isomorphic to {the $(3,2)$-type function field} of   $\NODE_{i_1,\cdots, i_n}$.} In other words,  one can choose any $\BasicT$-torsion flag   $\mathcal{V} = (V_1\subset\ldots \subset V_n)$    {subordinate to}   $\NODE_{i_1,\cdots, i_n}$, and then we have 	
	$$\dot{\mathcal{F}}^{(3,2)} _{i_1,\cdots,i_{n}  }\cong \mathcal{F}_{\mathcal{V}}.$$
	Alternatively, the affine curve $\dot{X}^{(3,2)}_{i_1,\cdots,i_{n}  }$    
	over $\Fq$  can be expressed
	in variables $(G,X_1,\ldots,X_n)$ with defining equations
	$$
	\minimalpoly{\mathcal{V}_1} (X_1)=0,\quad \minimalpoly{\mathcal{V}_2} (X_2)=0, \quad \cdots,\minimalpoly{\mathcal{V}_{n-1}} (X_{n-1})=0,\quad \minimalpoly{\mathcal{V}_n} (X_n)=0.
	$$  
Here $\mathcal{V}_j:= (V_1\subset\ldots \subset V_j)$ ($j=1,\cdots,n-1$) are ancestors of $\mathcal{V}=\mathcal{V}_n$, and 	 
  $\minimalpoly{\mathcal{V}_j}\in \mathcal{F}_{{\mathcal{V}_{j-1}}}[X]$ are the associated minimal polynomials.

	\item[(2)] (Theorem \ref{Thm:NodeToFactor} ---) Suppose that   $\NODE_{i_1,\cdots, i_n}$ has totally $k$ child nodes $\NODE_{i_1,\cdots, i_n,1}$, $\NODE_{i_1,\cdots, i_n,2}$, $\cdots$, and $\NODE_{i_1,\cdots, i_n,k}$. Choose arbitrarily child $\BasicT$-torsion flags of $ \mathcal{V}$, say $ {\mathcal{V}_{n+1}^{(1)}}$, $\cdots$, and $ {\mathcal{V}_{n+1}^{(k)}}$, which are  subordinate to $\NODE_{i_1,\cdots, i_n,1}$, $\NODE_{i_1,\cdots, i_n,2}$, $\cdots$, and $\NODE_{i_1,\cdots, i_n,k}$, respectively. Then the   minimal polynomials	$\minimalpoly{\mathcal{V}_{n+1}^{(i)}} $ ($i=1,\cdots, k$) are distinct, and irreducible factors of	\begin{eqnarray*} 
	\kappau{u_n}{X}&:=&X^{q^2+q+1}+ \frac{1-u_n^{q^2+q+1}}{u_n^{q^2+q}}    X^{q+1}-1  \ \\&&(\in  \mathcal{F}_{\mathcal{V}}[X], \mbox{ to be called the \textbf{modular polynomial}} )
	\end{eqnarray*} i.e.,	
	\begin{equation}\label{Eqt:kappadecompose0}
	\kappa^{(u_n)} =\minimalpoly{\mathcal{V}_{n+1}^{(1)}}  ~\minimalpoly{\mathcal{V}_{n+1}^{(2)}} ~ \cdots ~\minimalpoly{\mathcal{V}_{n+1}^{(k)}} .
	\end{equation}
	As a direct consequence,  \textit{monic irreducible factors of $\kappa^{(u_n)}$ and child nodes of $\NODE_{i_1,\cdots, i_n}$ are in   one-to-one correspondence.}
		
\end{itemize}

  In conclusion, the process of finding the modular curves involves determining the minimal polynomials node by node within the $T$-torsion tree. However, the computation of these minimal polynomials is a complex task and heavily relies on the nodes to which they are subordinate. Therefore, in Section \ref{Sec:Normalized32things}, we provide explicit formulas of the minimal polynomials subordinate to the first 23 nodes of $\Ttorsiontree^3$ (up to level $n=4$), thus presenting the corresponding function fields of these normalized Drinfeld modular curves in an explicit manner. To save pages,     details of the verification process are not included in this paper. More technical details and further discussions will be presented in our future work. Additionally, we strive to generalize the main results in Section \ref{Sec:32flagclassfunctionfield} to general $(m,j)$-type Drinfeld modules. However, it should be noted that there are significant difficulties that need to be resolved, as mentioned in Remark \ref{Rmk:32tomj}.

    In the classical elliptic curve theory and Drinfeld modular curves, tower structures are present. However, the tree structure proposed in this paper is unique to Drinfeld modular curves and is not found in existing elliptic curve theory. Our novel approach utilizing the $T$-torsion tree not only enhances the classic tower of torsion sequence structure but also integrates different torsion structures internally. This theory of Drinfeld modular curves has potential applications in AG codes. Construction and computation of AG codes typically rely on explicit expressions of curves. Particularly, the asymptotic property requires a sequence of curves, and our method offers multiple choices to fulfill this requirement.


\bigskip \subsection*{Notations}   Throughout the paper, $\Fq $ stands for the finite field of cardinality $q$.
We also use the short hand symbol
$$\Nqfrac{l}:=\frac{q^l-1}{q-1}=\sum_{i=0}^{l-1}{q^i}.$$

\begin{compactenum}
	\item   $\#S  $ --- the cardinality of a set $S$.
	
	\item $\mathbf{U}_n$, $\mathbf{U}^*_n$, $\mathbf{U}^o_n$ --- the set of upper-triangular $(n\times n)$ matrices (over $\Fq$),   nonsingular upper-triangular $(n\times n)$ matrices, and    nilpotent upper-triangular $(n\times n)$ matrices, respectively; see Section \ref{subSec:UTM}.
	
	\item  $\mathbf{U}^o_{n }/_\sim$ ---  the set of conjugacy $n$-classes; see Section \ref{subSec:UTM}.
	
	\item $ \Rconjugacyclass{i_1,\cdots, i_n}$ --- the conjugacy class with standard matrix $R_{i_1,\cdots, i_n}\in \mathbf{U}^o_n$. The parent matrix of $R_{i_1,\cdots, i_n}$ is $R_{i_1,\cdots, i_{n-1}}$; see Section \ref{subSec:UTM}.
	
	\item $\Exp (N)$ --- the exponent of $N\in \mathbf{U}^o_n$; see Section \ref{subSec:UTM}.
	
	\item $p_{n+1}$ --- the restriction map $\mathbf{U}^o_{n+1 } \to   \mathbf{U}^o_{n }$, or $\mathbf{U}^o_{n+1 }/_\sim
	\to \mathbf{U}^o_{n }/_\sim$; see Section \ref{subSec:UTM}.
	
	\item $\matrixtree$ ---	the tree of nilpotent upper-triangular matrices   (over $\Fq $); see Section \ref{Sec:matrixtreeupto4}.
	
	\item $\mathbf{U}^!_n$ --- the subgroup of $\mathbf{U}^*_n$ which consists of matrices of the form
	$\begin{pmatrix}
	{I}_{ n-1 }  & b   \\
	0  & k
	\end{pmatrix}$, where $I_{n-1}$ stands for the $(n-1)\times (n-1)$-unit matrix; see Section \ref{Sec:matrixtreeupto4}.
	
	
	\item $A$ ---  the polynomial ring $\Fq [T]$ over $\Fq $; see Section \ref{subSec:Ttorsionflags}.
	\item ${S^1}$, $S^m$ --- the $A$-modules ${S^1}:= \varinjlim_n (\frac{1}{T^n}A)/A  = \Fq [\frac{1}{T}]/ \Fq$ and $S^m: = (S^1)^{\oplus m}$; 
	see Section  \ref{subSec:Ttorsionflags}.
	
	\item $\Ttorsiontree^m $ --- the $T$-torsion tree of ${S^m}$; see Section \ref{subSec:Ttorsionflags}.
	
	\item $\mathcal{V}$ --- a flag in $S^m$ (or $\bar{L}$); see Section \ref{subSec:Ttorsionflags}  or \ref{Sec:flagsnotations}.
	
	\item $ \Trank (\mathcal{V} )$ --- the $T$-order  of $\mathcal{V}$; see Section \ref{subSec:Ttorsionflags}.
	
	\item $\Exp (\mathcal{V} )$ --- the exponent of $\mathcal{V} $; see Section \ref{subSec:Ttorsionflags}.
	
	\item  $(\nu_1,\ldots,\nu_n)$ ---   an ordered-basis of $\mathcal{V} $; see Section  \ref{subSec:Ttorsionflags}.
	
	\item 	$\mathbf{U}^o_{n,m}$ --- the set of  upper-triangular nilpotent $(n\times n)$ matrices with rank
	$ \geqslant n-m$; see Section \ref{subSec:Ttorsionflags}.
	
	\item $\NODE_{i_1,\cdots, i_n}$ --- a level $n$ node of the $T$-torsion tree $\Ttorsiontree^m$; see Section \ref{subSec:treeflag}.
	
	
	
	\item   $ L $, $ \Lbar  $ --- a field containing $ \Fq  $, and its  algebraic closure, respectively; see Section \ref{subSec:31}.
	
	\item  $ L\set{\tau} $  --- the twisted polynomial ring; see Section \ref{subSec:31}.
	
	\item $\Lambda$ and $\Theta$ --- the standard correspondence; see Section \ref{Sec:flagsnotations}. 
	
	\item $ \mathcal{G}(L,\phi;\NODE_{i_1,\cdots, i_n} )$ --- the set of all $\phi_T$-torsion $n$-flags   which are  defined  over $L$ and subordinate to    $ \NODE_{i_1,\cdots, i_n}$; see Section \ref{subSec:torsionflagdm}.
	
	\item $  \mathcal{D}_L^{(m,j)} $ --- the set of   $(m,j)$-type normalized  Drinfeld modules; see Section \ref{Sec:Drfeldparameters}.
	
	
	\item $\dot{X}^{(m,j)}_{i_1,\cdots,i_{n}  }$, $X^{(m,j)}_{i_1,\cdots,i_{n}  }$ --- the normalized Drinfeld modular curve  and, respectively, the Drinfeld modular curve   (subordinate to $\NODE_{i_1,\cdots, i_n}$); see Section \ref{Sec:Drfeldparameters}.
	
	\item
	$\dot{\mathcal{F}}^{(m,j)}_{i_1,\cdots,i_{n}}$, $ {\mathcal{F}}^{(m,j)}_{i_1,\cdots, i_n}$ --- the function fields of $\dot{X}^{(m,j)}_{i_1,\cdots,i_{n}  }$ and ${X}^{(m,j)}_{i_1,\cdots,i_{n}  }$, respectively; see Section \ref{Sec:Drfeldparameters}.
	
	\item $\Phi_{T}^{(G)}$ --- a special $(3,2)$-type normalized Drinfeld module; see Section \ref{Sec:32flagclassfunctionfield}. 
	
	\item $\kappau{u}{X}$ --- the modular polynomial; see Section \ref{Section6.1}.
	
	\item $\mathcal{F}_{\mathcal{V}}$ --- the function field of $\mathcal{V}$; see Section \ref{subSec:L0characteristics}.
	
	\item  $\minimalpoly{\mathcal{V}}$ --- the minimal polynomial of $\mathcal{V}$; see Section \ref{subSec:L0characteristics}.
	
	\item $\mathcal{O}_{\mathcal{V}}$ --- the integral closure of $\mathbb{F}_q[G]$ in $\mathcal{F}_{\mathcal{V}}$; see Section \ref{Section6.5}.
	
	\item $\mathcal{C}_{\mathcal{V}}$ --- the affine curve over $\mathbb{F}_q$ associated to $\mathcal{F}_{\mathcal{V}}$; see Section \ref{Section6.5}.
\end{compactenum}



\section{Preliminaries -- the tree of  nilpotent upper-triangular matrices}\label{Sec:Pre-NUTMtree}
\bigskip\subsection{Upper-triangular matrices  and their conjugacy classes}\label{subSec:UTM}
Let $\mathbf{U}_n$, $\mathbf{U}^*_n$,
and $\mathbf{U}^o_n$ be, respectively, the set of upper-triangular, nonsingular upper-triangular,   
and  nilpotent upper-triangular $(n\times n)$ matrices (over $\Fq$). The group $\mathbf{U}^*_n$ acts on $\mathbf{U}^o_n$ by conjugation:
$$
B.N:=B^{-1}NB,\qquad B\in \mathbf{U}^*_n, N\in \mathbf{U}^o_n.
$$

An orbit of the $\mathbf{U}^*_n$-conjugation action on  $\mathbf{U}^o_n$ is called a \textbf{conjugacy ($n$-)class}.
In the sequel, we use $ \mathbf{U}^o_{n }/_\sim$ to  denote  the set of conjugacy $n$-classes. An element in $\mathbf{U}^o_{n }/_\sim $ will be denoted by   $[N  ]$ if it is represented by a matrix $N \in  \mathbf{U}^o_{n }$.

The \textbf{rank} of a conjugacy $n$-class $ [N  ]$ refers to the rank of $N $ as a matrix.
The \textbf{exponent} of $ [N  ]$, denoted by   $\Exp (N)$, is the number $e\geqslant 1$ uniquely determined by the conditions $ N  ^{e-1}\neq 0$ and $ N ^{e} = 0$.

By taking the  left-top $(n\times n)$-block, there is a restriction map
$p_{n+1}: \mathbf{U}^o_{n+1 } \to   \mathbf{U}^o_{n }$. For $N' \in \mathbf{U}^o_{n+1 }$, we will call $N=p_{n+1}N' \in \mathbf{U}^o_{n  }$ \textbf{the parent matrix} of $ N' $, and call $N' $ \textbf{a child matrix} of $N$. Clearly,  the restriction $p_{n+1}$ also maps conjugacy $(n+1)$-classes   to $n$-classes --- For each   $ [N' ]\in  \mathbf{U}^o_{n+1 }/_\sim $, $p_{n+1}[N' ]$ is the $n$-class $[p_{n+1}N' ]\in \mathbf{U}^o_{n }/_\sim $.  We shall call  $[p_{n+1}N']$    \textbf{the parent class} of $[N']$, and refer to $[N'] $ as   \textbf{a child class}   of $[p_{n+1}N']$. The map $p_{n+1}: [N'] \to  [ p_{n+1}N']$ is indeed surjective.

\begin{defn}\label{Defn:conjugacyclassdirectson} If a  conjugacy $(n+1)$-class  $[N']$ can be represented by a matrix $N' \in \mathbf{U}^o_{n+1}$ whose elements in the last column are all zeros, then    we say that $[N']$ is  \textbf{trivially extended} (from its parent). 
	
\end{defn}

The    classification problem of upper-triangular matrices up to conjugation has been studied by many works. We refer to the papers \cites{MR577913, MR123605,MR113948,MR470098,MR1190401,MR1260823}  for the general theories. Here we directly use the result of \cite{MR1448353} and list a total of 24    conjugacy $n$-classes for $n=1$, $2$, $3$, and $4$. These conjugacy classes are denoted by $ \Rconjugacyclass{i_1,\cdots, i_n}$ with standard matrices $R_{i_1,\cdots, i_n}\in \mathbf{U}^o_n$, where  $R_{i_1,\cdots, i_n}$ are indexed so that the parent matrix of $R_{i_1,\cdots, i_n}$ is $R_{i_1,\cdots, i_{n-1}}$.
In  Section \ref{Sec:listofmatrices} of the appendix, one can find all such standard matrices $R_{i_1,\cdots, i_4}\in \mathbf{U}^o_4$. 
For example, we have the matrix
\begin{equation}
	\label{Eqt:R1121example}
	R_{1,1,2,1}=\left( \begin{array}{cccc}0 & 1 &0 &0 \\0 & 0 & 0&1 \\0 & 0 & 0&1\\0 &0&0&0 \end{array}\right), \end{equation}
and from  $R_{1,1,2,1}$ one can find its parent $$R_{1,1,2}=\left( \begin{array}{ccc}0 & 1 &0  \\0 & 0 & 0 \\0 & 0 & 0  \end{array}\right), $$ grandparent $R_{1,1}=\left( \begin{array}{cc}0 & 1    \\0 & 0    \end{array}\right)$, and so on.

\bigskip\subsection{The   tree of nilpotent upper-triangular matrices}\label{Sec:matrixtreeupto4}
\begin{defn}
	The \textbf{tree of nilpotent upper-triangular matrices}  over $\Fq $, denoted by $\matrixtree $,  is the rooted  tree   that consists of the following data:
	\begin{itemize}
		\item the level $1$ node   $\Rconjugacyclass{1}= [(0)]\in \mathbf{U}^o_{1 }/_\sim  $, also called the root of $\matrixtree$;
		\item all  conjugacy classes    in $\mathbf{U}^o_{n }/_\sim $ as level $n$ nodes, for    $n=2,3,\cdots$;
		\item all edges connecting parent conjugacy classes to their children.
	\end{itemize}
\end{defn}

This tree has infinitely many levels and nodes; see   Figure \ref{fig:NUTMtree}    (up to  its fourth level).

\tikzstyle{level 1}=[level distance=2cm, sibling distance=9.5cm]
\tikzstyle{level 2}=[level distance=4cm, sibling distance=3.5cm]
\tikzstyle{level 3}=[level distance=4cm, sibling distance=1.1cm]
\tikzstyle{bag} = [text width=4em, text centered]
\tikzstyle{end} = [circle, minimum width=3pt,fill, inner sep=0pt]

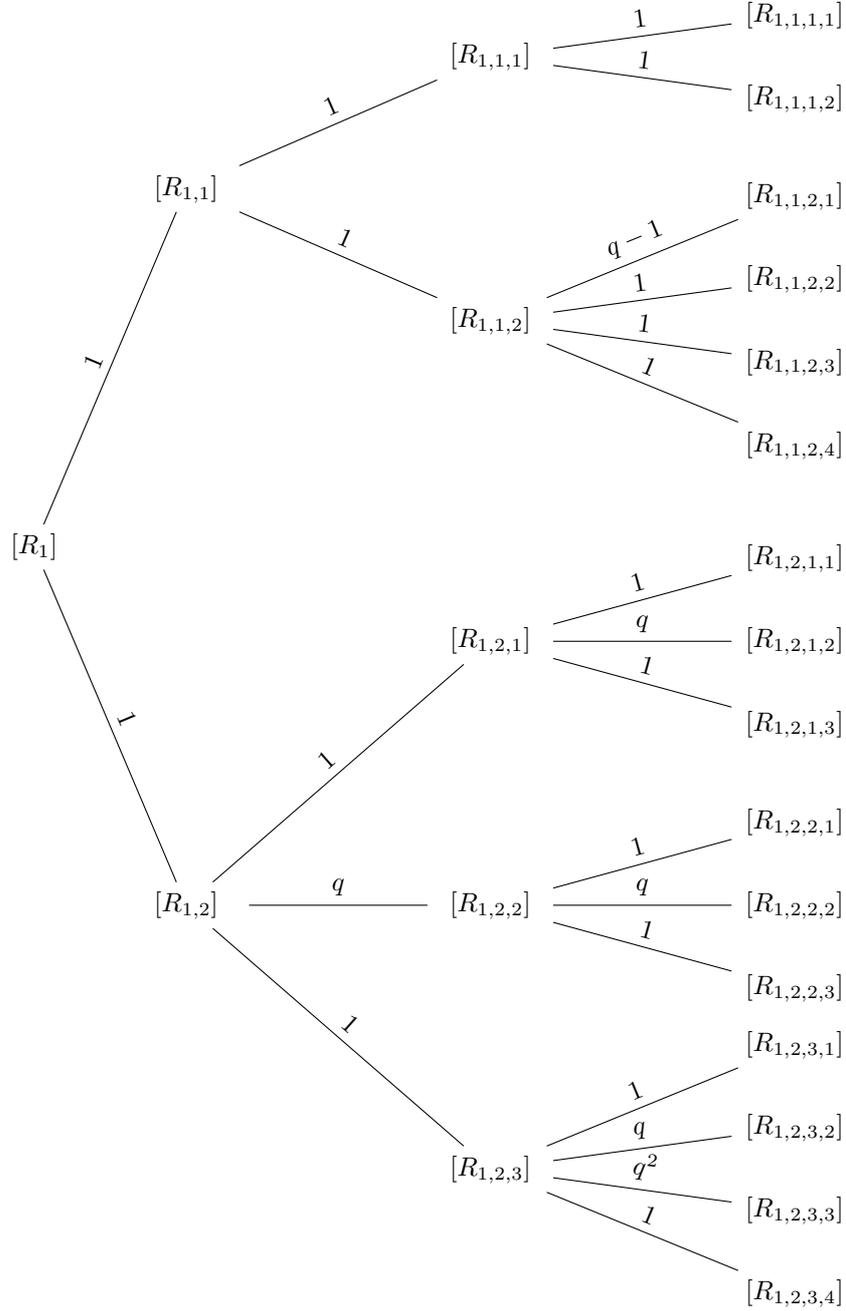
\begin{figure}[htbp]\caption{\textbf{The  tree $\matrixtree$  of nilpotent upper-triangular matrices up to its fourth level}. (The number above each edge indicates the matrix degree of the corresponding child class,  see Definition \ref{Defn:matrixdegree}. Among all  child classes of a parent class, the  one with the largest   last index is the  trivially extended conjugacy class; see Definition \ref{Defn:conjugacyclassdirectson}.)}\label{fig:NUTMtree} 
	\begin{tikzpicture}[grow=right, sloped]
		\node[bag]{$\Rconjugacyclass{1}$ }
		child{
			node[bag]{$\Rconjugacyclass{1,2}$ }
			child{
				node[bag]{$\Rconjugacyclass{1,2,3}$ }
				child{
					node[bag]{$\Rconjugacyclass{1,2,3,4}$ }
					edge from parent
					node[above] {$ 1$}
					node[below]{}
				}
				child{
					node[bag]{$\Rconjugacyclass{1,2,3,3}$ }
					edge from parent
					node[above] {$ q^2$}
					node[below]{}
				}
				child{
					node[bag]{$\Rconjugacyclass{1,2,3,2}$ }
					edge from parent
					node[above] {$ q$}
					node[below]{}
				}
				child{
					node[bag]{$\Rconjugacyclass{1,2,3,1}$ }
					edge from parent
					node[above] {$ 1$}
					node[below]{}
				}
				edge from parent
				node[above] {$ 1$}
				node[below]{}
			}
			child{
				node[bag]{$\Rconjugacyclass{1,2,2}$ }
				child{
					node[bag]{$\Rconjugacyclass{1,2,2,3}$ }
					edge from parent
					node[above] {$ 1$}
					node[below]{}
				}
				child{
					node[bag]{$\Rconjugacyclass{1,2,2,2}$ }
					edge from parent
					node[above] {$ q$}
					node[below]{}
				}
				child{
					node[bag]{$\Rconjugacyclass{1,2,2,1}$ }
					edge from parent
					node[above] {$ 1$}
					node[below]{}
				}
				edge from parent
				node[above] {$q$}
				node[below]{}
			}
			child{
				node[bag]{$\Rconjugacyclass{1,2,1}$ }
				child{
					node[bag]{$\Rconjugacyclass{1,2,1,3}$ }
					edge from parent
					node[above] {$ 1$}
					node[below]{}
				}
				child{
					node[bag]{$\Rconjugacyclass{1,2,1,2}$ }
					edge from parent
					node[above] {$ q$}
					node[below]{}
				}
				child{
					node[bag]{$\Rconjugacyclass{1,2,1,1}$ }
					edge from parent
					node[above] {$ 1$}
					node[below]{}
				}
				edge from parent
				node[above] {$ 1$}
				node[below]{}
			}
			edge from parent
			node[above] {$ 1$}
			node[below]{}
		}
		child{
			node[bag]{$\Rconjugacyclass{1,1}$ }
			child{
				node[bag]{$\Rconjugacyclass{1,1,2}$ }
				child{
					node[bag]{$\Rconjugacyclass{1,1,2,4}$ }
					edge from parent
					node[above] {$ 1$}
					node[below]{}
				}
				child{
					node[bag]{$\Rconjugacyclass{1,1,2,3}$ }
					edge from parent
					node[above] {$ 1$}
					node[below]{}
				}
				child{
					node[bag]{$\Rconjugacyclass{1,1,2,2}$ }
					edge from parent
					node[above] {$ 1$}
					node[below]{}
				}
				child{
					node[bag]{$\Rconjugacyclass{1,1,2,1}$ }
					edge from parent
					node[above] {$ q-1$}
					node[below]{}
				}
				edge from parent
				node[above] {$ 1$}
				node[below]{}
			}
			child{
				node[bag]{$\Rconjugacyclass{1,1,1}$ }
				child{
					node[bag]{$\Rconjugacyclass{1,1,1,2}$ }
					edge from parent
					node[above] {$ 1$}
					node[below]{}
				}
				child{
					node[bag]{$\Rconjugacyclass{1,1,1,1}$ }
					edge from parent
					node[above] {$ 1$}
					node[below]{}
				}
				edge from parent
				node[above] {$ 1$}
				node[below]{}
			}
			edge from parent
			node[above] {$ 1$}
			node[below]{}
		};
	\end{tikzpicture}
\end{figure}  


\bigskip\subsection{Matrix degrees of conjugacy classes}\label{subSec:matrixdegrees}

Denote by  $\mathbf{U}^!_n$ the subgroup of $\mathbf{U}^*_n$ which consists of matrices of the form
$$\begin{pmatrix}
	{I}_{ n-1 }  & b   \\
	0  & k
\end{pmatrix}
$$  where $I_{n-1}$ stands for the $(n-1)\times (n-1)$-unit matrix,  $k\in \Fq $ is nonzero, and $b\in \Fq ^{n\times 1} $ is arbitrary.

\begin{defn}\label{Defn:matrixdegree} Let  $[N']$ be a  conjugacy $(n+1)$-class and $   [N]       $  its parent. Consider the set
	$$
	B({N} ;[N']):=\left\{  \begin{pmatrix}
	{N}   & b   \\
	0 & 0
	\end{pmatrix}
	, b\in \Fq ^{n\times 1}  \right\} \cap [N'].
	$$
	Denote by $ {B}({N};[N'])/_\sim $ the quotient space of $B({N} ;[N'])$ modulo  the conjugation action by $\mathbf{U}^!_{n+1}$. Then
	the  \textbf{matrix degree} of $[N']$ is defined by
	$$\matrixdeg_{[N']}:=\# {B}({N} ;[N'])/_\sim.
	$$

\end{defn}

 By some standard linear algebra arguments, one can   prove that the number $\matrixdeg_{   [N']}$ does not depend on the choice of $N$ in $[N]=p_{n+1}[N']$. Hence, the matrix degree $\matrixdeg_{ [N']}$ is well-defined.


The following fact is also easy to verify.
\begin{prop}\label{Prop:conjugacyclassdirectchild}For any   trivially extended   conjugacy $(n+1)$-class $[N']$,   we have   $\matrixdeg_{[N'] } =1$.
\end{prop}

\begin{proof}
	It suffices to prove that for all $\tilde{N}' \in B(N;[N'])$, $\tilde{N}'$ and $N'=\begin{pmatrix}
		{N}   & 0   \\
		0 & 0
	\end{pmatrix}$ are conjugations by $\mathbf{U}_{n+1}^!$. By definition, there exists $A \in \mathbf{U}_{n+1}^*$, such that $\tilde{N}'=AN'A^{-1}$ which is of the form $\tilde{N}'=\begin{pmatrix}
		{N}   & b   \\
		0 & 0
	\end{pmatrix}$. Let us assume that $A=\begin{pmatrix}
		C   & a   \\
		0 & k
	\end{pmatrix}$, where $C \in \mathbf{U}_{n}^*$, $a\in \F_{q}^{n \times 1}$, $k \in \F_{q}^*$. A simple computation gives $b=-k^{-1}Na$. Therefore,  we have the desired relation
	$$\tilde{N}'=\begin{pmatrix}
		{N}   & -k^{-1}Na   \\
		0 & 0
	\end{pmatrix}=\begin{pmatrix}
		I   & a   \\
		0 & k
	\end{pmatrix}N'\begin{pmatrix}
		I   & a   \\
		0 & k
	\end{pmatrix}^{-1}.$$
\end{proof}
\begin{prop}Assume that a   conjugacy $(n+1)$-class $[N]$ is of rank $r$ and has totally $i$ child classes $[N^{(j)}]$, $j=1,\cdots, i$.  Then we have
	\begin{equation}\nonumber
		\sum_{j=1}^{i }\matrixdeg_{    [N^{(j)} ]}=\Nqfrac{n-r} +1=\frac{q^{n-r}-1}{q-1}+1.	
	\end{equation}
	
\end{prop}
\begin{proof}
	By definition, we have
	$$\bigcup_{j=1}^{i}B(N,[N^{(j)}])=\left\{  \begin{pmatrix}
		{N}   & b   \\
		0 & 0
	\end{pmatrix}
	, b\in \Fq ^{n\times 1}  \right\}.$$
	It is easy to verify that two matrices
	$\begin{pmatrix}
		{N}   & b   \\
		0 & 0
	\end{pmatrix}$ and  $\begin{pmatrix}
		{N}   & \tilde{b}   \\
		0 & 0
	\end{pmatrix}$ are conjugations by $\mathbf{U}^{!}_{n+1}$  if and only if there exist   $a\in \F_q^{n \times 1}$ and $k \in \F_{q}^*$   such that $b=k\tilde{b}+Na$. The number $\sum_{j=1}^{i }\matrixdeg_{    [N^{(j)} ]}$ is exactly the number of orbits by the obvious   $\F_{q}^*\times \langle N\rangle$-action on $\F_q^{n\times 1}$, where $\langle N\rangle$ is the linear space spanned by  column vectors of $N$, and considered as an abelian group.  A small amount of linear algebra analysis gives the desired number of orbits:
	\begin{eqnarray*}
		\# (\F_q^{n\times 1}/\F_{q}^*\times \langle N\rangle\mbox{-action})=
		\frac{\#(\F_{q}^{n\times 1 }/\langle N \rangle)-1}{\#\F_q^*}+1
		= \frac{q^{n-r}-1}{q-1}+1.
	\end{eqnarray*}

\end{proof}

\section{The tree of $T$-torsion flags}\label{Sec:basic}

\bigskip\subsection{$T$-torsion flags and    $T$-torsion flag classes}\label{subSec:Ttorsionflags}

Throughout this paper,  the notation $A$ stands for the polynomial ring $\Fq [T]$.
Define an $A$-module  $${S^1}:= \varinjlim_n (\frac{1}{T^n}A)/A  \cong \Fq [\frac{1}{T}]/ \Fq \cong \varinjlim_n (A/T^nA).$$ 

Let $m\geqslant 2$ be a positive integer.
In  this section, we focus on  the $A$-module \begin{equation}
\label{Eqt:Sm}
{S^m}:=(S^1)^{\oplus m} \cong
(\Fq [\frac{1}{T}])^{\oplus m}/ \Fq^{\oplus m}
 \end{equation} which is also viewed as an (infinite dimensional) $\Fq $-vector space.  {The $T$-action on $S^m$ is     a surjective $\Fq $-endomorphism, and $\Ker T   $ is spanned by $m$ vectors, namely   $(\frac{1}{T},0,\ldots, 0)$, $(0,\frac{1}{T}, \cdots, 0)$, $\cdots$, and $(0,\ldots, 0,\frac{1}{T})$ in $S^m$.


Let $n$ be a positive integer. 
By an \textbf{$n$-flag} in ${S^m}$   we mean  a sequence of subspaces ordered by inclusions: \begin{equation*}
\mathcal{V}  =\flag{ V_1\subset V_2\subset \cdots\subset V_n}
\end{equation*}
where each $V_i\subset {S^m}$   
is of dimension $i$ (over $\Fq$). Now regarding the $A$-module structure, or $T$-action, on ${S^m}$, we define a special type of $n$-flags:

\begin{defn}\label{defn:flagandcodimension}
	A \textbf{$T$-torsion $n$-flag}   in $S^m$  
	is a flag   $\mathcal{V}=\flag{  V_1\subset V_2\subset \cdots\subset V_n}$ in ${S^m}$ 	such that     $TV_i\subseteq V_{i-1}$ holds for all $i=2,\cdots, n$ and   $TV_1=\set{0}$.
	
\end{defn}So, each $V_i$
is an $A$-submodule in $S^m$ and the $T$-action on $V_i$ is nilpotent.
 The number $n$ will be referred to as the \textbf{level} of the flag $\mathcal{V}$.
 The $T$-\textbf{order} of $\mathcal{V}$, denoted by $ \Trank (\mathcal{V} )$, refers to the rank of $V_n$ as an $A$-module, i.e., the minimum number of generators. Alternatively, we can define
$$  \Trank (\mathcal{V} ):=n-\dim_{\Fq } (  {TV_n}) .$$

  The \textbf{exponent} of $\mathcal{V} $, denoted by $\Exp (\mathcal{V} )$, is the number $e\geqslant 1$ uniquely determined by the conditions $T^{e-1}(V_n)\neq 0$ and $T^{e} (V_n)= 0$.

 \begin{defn}
	Let $ \mathcal{V}   $ be a $T$-torsion $n$-flag.
	A \textbf{child} of   $\mathcal{V} $ is
	a $T$-torsion $(n+1)$-flag    $ {\mathcal{V}'} $
	whose first $n$-subspaces coincide with those of $ \mathcal{V}   $, i.e.,
	${\mathcal{V}'} =\flag{ V_1\subset V_2\subset \cdots\subset V_{n }\subset V_{n+1}}  $.   We also say that   $ \mathcal{V}  $ is   \textbf{the parent}   of $ {\mathcal{V}'}  $.
\end{defn}

 Let us introduce an equivalence relation in the set of $n$-flags in $ {S^m}$:
\begin{defn}\label{isomorphism}
	Two $T$-torsion $n$-flags $\mathcal{V} =\flag{V_1\subset V_2\subset \cdots\subset V_n}$ and $\mathcal{W}=\flag{ W_1\subset W_2\subset \cdots\subset W_n}$    in $S^m$ are said to be \textbf{isomorphic} if there exists an $\Fq $-linear isomorphism
	$\iota: V_n\rightarrow W_n$
such that
	\begin{enumerate}
		\item $\iota(V_i)=W_i$ for $i=1,\ldots, n$; and
		\item $\iota\circ T=T\circ \iota$.
	\end{enumerate}
 A   \textbf{$T$-torsion ($n$-)flag class} in $S^m$ is an isomorphic class of such $T$-torsion $n$-flags.

  \end{defn}

In the sequel, a $T$-torsion flag class  is denoted by   $ [\mathcal{V} ]$  if it is represented by the $T$-torsion  flag $\mathcal{V} $.
 We also     define the  $T$-\textbf{order} of the torsion flag class $[\mathcal{V} ]$ to be that of
$ \mathcal{V}$ as it does not depend on the choice of $\mathcal{V}  $. Similarly, the \textbf{exponent} of $[\mathcal{V} ]$ is defined to be that of $\mathcal{V} $.

\begin{example}\label{Example:NODE1}
	Any  $T$-torsion $1$-flag $\mathcal{V} $ is merely   a   $1$-dimensional $\Fq $-subspace $V_1\subset \Ker T$. It is clear that any two $T$-torsion $1$-flags are isomorphic. Hence there is a unique $T$-torsion $1$-flag class which we denote  by $ \NODE_{1}$.
\end{example}

Given an $n$-flag $\mathcal{V} =\flag{V_1\subset V_2\subset\cdots\subset V_n}$, one is able to
find a basis $  (\nu_1,\ldots,\nu_n)$ of $V_n$  such that $V_i=\mathrm{Span}_{\Fq }\{\nu_1,\ldots,\nu_i\}$. In such a situation we call $(\nu_1,\ldots,\nu_n)$   an \textbf{ordered basis} of $\mathcal{V} $. So the flag $\mathcal{V} $   can also be realized by the array of vectors $(\nu_1,\ldots,\nu_n)$, and we simply write $$\mathcal{V} =(\nu_1<\cdots<\nu_n)$$ for this relation.

\begin{defn}
	A $T$-torsion $n$-flag class $[\mathcal{V} ]$ is called  \textbf{$T$-vanishing} if it is represented by a $T$-torsion flag $\mathcal{V} =(\nu_1<\cdots<\nu_n )$ satisfying  $T\nu_n=0$.
\end{defn}


Let $\mathcal{V} =\flag{V_1\subset V_2\subset\cdots\subset V_n}=(\nu_1<\cdots<\nu_n)$ be a $T$-torsion  flag.
 The $\Fq $-linear endomorphism  $T:~V_n\to V_n$ is represented  by a nilpotent matrix
$N\in \mathbf{U}^o_n$ with respect to $(\nu_1,\ldots,\nu_n)$,  i.e.,
\begin{equation}\label{Eq:FromTtoN} T(\nu_1,\ldots,\nu_n)=(\nu_1\cdots,\nu_n)N.\end{equation}
Of course, $N$ depends on the choice of $(\nu_1,\ldots,\nu_n)$. However, the conjugacy class of $N$ (up to upper-triangular similarity) is independent:
\begin{lem}For any two ordered-bases $(\nu_1,\ldots,\nu_n)$ and $(\tilde{\nu}_1,\ldots,\tilde{\nu}_n)$ of a $T$-torsion  flag $\mathcal{V} $,  the corresponding matrices $N$ and $\tilde{N }$ of $T$ are  $\mathbf{U}^*_n$- similar.
\end{lem}
This lemma tells us that the $T$-action on a $T$-torsion flag $\mathcal{V} $   corresponds to an upper-triangular matrix  $N\in \mathbf{U}^o_n$ whose conjugacy class is uniquely determined by $\mathcal{V}$.
 Passing to $T$-torsion  flag classes, we have a criterion.

\begin{prop}\label{Prop:Nnm-Sn}
	Let $\mathbf{U}^o_{n,m}$ be the set of  upper-triangular nilpotent $(n\times n)$ matrices with rank
	$ \geqslant n-m$. 	There exists a one-to-one correspondence between the set of $T$-torsion $n$-flag classes   and the set    of conjugacy classes of $\mathbf{U}^o_{n,m}$:
	$$
	 [\mathcal{V} ]~\ \mapsto~\  [N], ~\mbox{ where $\mathcal{V} $ and $N$ are related by   \eqref{Eq:FromTtoN}.}
	$$
	Under this correspondence, the $T$-order  of $ [\mathcal{V}]  $ equals $(n- \mathrm{rank}N) $.
\end{prop}
\begin{proof}It suffices to prove that any matrix   $X\in\mathbf{U}^o_{n,m}$ can be realized as a $T$-torsion $n$-flag(of $ {S^m} $) to which the standard $T$-action   is represented by $X$ (up to conjugacy). We show its construction below ---  Consider $V=\Fq ^{n}$ and the $\Fq $-linear endomorphism $T^X:~V\to V$ determined by $X$, i.e.
	$$T^X(b_1,\ldots, b_n)=(b_1,\ldots, b_n )X$$
	where $ (b_1,\ldots, b_n)$ is the standard basis of $V$. Because $X$ has rank $r\geqslant n-m$,  the kernel of $T^X$ is of dimension $l=n-r$ ($\leqslant m$). Let us take a basis $(c^{(1)}_1,\ldots, c^{(l)}_1)$ of $\Ker(T^X) $.
		As $X$ is nilpotent, one can extend $(c^{(1)}_1,\ldots, c^{(l)}_1)$ to a cyclic basis   of $V$, say
	$$ (c^{(1)}_1,\ldots, c^{(1)}_{k_1},c^{(2)}_1,\ldots, c^{(2)}_{k_2},\ldots, c^{(l)}_1,\ldots,c^{(l)}_{k_l})$$
	such that $T^X(c^{(i)}_1)=0$ and  $T^X(c^{(i)}_j)=c^{(i)}_{j-1}$. In other words, with respect to the cyclic basis, $T^X$ takes its canonical Jordan form $J$. Suppose further that the two bases   are related by
	$$(b_1,\ldots, b_n )=(c^{(1)}_1,\ldots, c^{(1)}_{k_1},c^{(2)}_1,\ldots, c^{(2)}_{k_2},\ldots, c^{(l)}_1,\ldots,c^{(l)}_{k_l})G,
	$$
	for some $G\in \mathrm{GL}(n;\Fq )$.
	Then it follows that $J=G X G^{-1}$.
	
	As we have $l\leqslant m$, we can take vectors in  ${S^m} $:
	$$\tilde{c}^{(1)}_1=(\frac{1}{T},0,\ldots,0), \tilde{c}^{(1)}_2=(\frac{1}{T^2},0,\ldots,0),\ldots,
	\tilde{c}^{(1)}_{k_1}=(\frac{1}{{T}^{k_1}},0,\ldots,0);$$
	$$\tilde{c}^{(2)}_1=(0,\frac{1}{T},0,\ldots,0), \tilde{c}^{(2)}_2=(0,\frac{1}{T^2},0,\ldots,0),\ldots,
	\tilde{c}^{(2)}_{k_2}=(0,\frac{1}{{T}^{k_2}},0,\ldots,0);$$
	$$\cdots \cdots$$
	$$\tilde{c}^{(l)}_1=(0,\ldots,0,\frac{1}{T},0,\ldots,0), \tilde{c}^{(l)}_2=(0,\ldots,0,\frac{1}{T^2},0,\ldots,0),\ldots,
	\tilde{c}^{(l)}_{k_l}=(0,\ldots,0,\frac{1}{{T}^{k_l}},0,\ldots,0).$$
	With respect to this basis, the $T$-action is represented by the matrix $J$.
	 Now, one finds vectors $\tilde{b}_i  $ in $S^{\phi}\cong {S^m} $ such that
	$$  (\tilde{b}_1,\ldots, \tilde{b}_n )=(\tilde{c}^{(1)}_1,\ldots, \tilde{c}^{(1)}_{k_1},\tilde{c}^{(2)}_1,\ldots, \tilde{c}^{(2)}_{k_2},\ldots, \tilde{c}^{(l)}_1,\ldots,\tilde{c}^{(l)}_{k_l})G.
	$$
	It is clear that with respect to $(\tilde{b}_1,\ldots, \tilde{b}_n )$, the $T$-action is represented by the matrix $X$.  In conclusion, the ordered basis $(\tilde{b}_1<\cdots< \tilde{b}_n )$ gives the desired   $T$-torsion $n$-flag.
\end{proof}

    \begin{example}\label{Example:findVfromN} Recall the matrix $R_{1,1,2,1}$ given as in Equation \eqref{Eqt:R1121example} which is of rank $2$. If we set $n=4$ and $m=3$, then $2\geqslant n-m$ is true and  there exists a $T$-torsion $4$-flag class $[\mathcal{V}]$ which corresponds to the conjugacy class $[R_{1,1,2,1}]$. In fact, we can take $\mathcal{V}=(\nu_1<\nu_2<\nu_3<\nu_4 )$ where the base vectors $\nu_i\in $ $S^3=( {S^1})^{\oplus 3}$ are given by
   	$$~\begin{cases}
   	\nu_1=(\frac{1}{T},0,0),\\ \nu_2=(\frac{1}{T^2},0,0) ,\\
   	\nu_3= (0,\frac{1}{T},0),\\\nu_4= (\frac{1}{T^3},\frac{1}{T^2},0)
   	\end{cases}$$
   	because we have
   	$$~\begin{cases}
   	T\nu_1=0,\\ T\nu_2=\nu_1 ,\\
   	T\nu_3= 0,\\T\nu_4= \nu_2+\nu_3.
   	\end{cases}$$
   	If we consider  $R_{1,2,3,4}$ which is a trivial $4\times 4$ matrix, i.e.,  of rank $0$, then  for $n=4$, $m=3$, we have $0< n-m$. Therefore, there is no $T$-torsion $4$-flag class  corresponding   to $[R_{1,2,3,4}]$.
   \end{example}

 \bigskip\subsection{The  tree of $T$-torsion flag classes}\label{subSec:treeflag}

We are in a position to introduce the core concept in this paper, the  \textbf{$T$-torsion  tree}. In essence, it is formed by $T$-torsion flag classes in  $S^m$   of all levels $n=1,2,\cdots$.

Let us fix the notation first. A $T$-torsion $n$-flag class  is denoted by
 $
\NODE_{i_1,\cdots, i_n}
 $
where  $(i_1,\cdots, i_n)$ is an  index. The indices are designed such that the parent of $\NODE_{i_1,\cdots, i_n}$ is   $\NODE_{i_1,\cdots, i_{n-1}}$.  

\begin{defn}
	The \textbf{$T$-torsion tree} of ${S^m}$, denoted by $\Ttorsiontree^m $,  is the rooted  tree   that consists of the following data:
	\begin{itemize}
		\item the level $1$ node   $  \NODE_{1}   $ (see Example \ref{Example:NODE1}) as the root of $\Ttorsiontree^m$;
		\item all $T$-torsion  flag classes   $  \NODE_{i_1,\cdots, i_n} $ in $S^m$ as level $n$ nodes, for    $n=2,3,\cdots$;
		\item all edges connecting parent $T$-torsion  flag classes $  \NODE_{i_1,\cdots, i_n} $ to their children $  \NODE_{i_1,\cdots, i_n,i_{n+1}} $.
	\end{itemize}
\end{defn}

According to Proposition \ref{Prop:Nnm-Sn}, each node $\NODE_{i_1,\cdots, i_n}$ of $\Ttorsiontree^m$ corresponds to a node  $  [R_{i_1,\cdots, i_n}] $ of the   tree $\matrixtree$ (of nilpotent upper-triangular matrices) whose rank  $\geqslant  n-m $, and vice versa. So \textbf{the $T$-torsion tree $\Ttorsiontree^m $ can be obtained by removing those  nodes of the   tree $\matrixtree$ whose level is $n$ and rank is less than $(n-m)$}.

We have drawn the $T$-torsion tree up to its fourth level in Figure \ref{fig:Ttree}. For $m\geqslant 4$,  the 24 nodes are connected as shown therein. If $m= 3$, there are 23 nodes because the last node $\NODE_{1,2,3,4}$ should be removed as $\Rconjugacyclass{1,2,3,4}$ has rank $0$, which is less than $(4-m=1)$ (cf. Proposition \ref{Prop:Nnm-Sn} and Example \ref{Example:findVfromN}). For the same reason, if $m=2$, there are only 16 nodes left (after removing nodes $\NODE_{1,2,3}$, $\NODE_{1,1,2,4}$, $\NODE_{1,2,1,3}$, $\NODE_{1,2,2,3}$, $\NODE_{1,2,3,1}$, $\NODE_{1,2,3,2}$, $\NODE_{1,2,3,3}$, and $\NODE_{1,2,3,4}$).

\begin{figure}[htbp]\caption{\textbf{The $T$-torsion tree $\Ttorsiontree^m$ up to its fourth level} (the number above each edge denotes the covering degree of the    node on the right end of the edge; see Definition \ref{Defn:coveringdegree}). For $m\geqslant 4$, the 24 nodes are shown. If $m=3$, the last one $\NODE_{1,2,3,4}$ should be removed. If $m=2$, 8 nodes should be removed ($\NODE_{1,2,3}$, $\NODE_{1,1,2,4}$, $\NODE_{1,2,1,3}$, $\NODE_{1,2,2,3}$, $\NODE_{1,2,3,1}$, $\NODE_{1,2,3,2}$, $\NODE_{1,2,3,3}$, and $\NODE_{1,2,3,4}$).}\label{fig:Ttree} 
\begin{tikzpicture}[grow=right, sloped]
\node[bag]{$\NODE_1$ }
child{
	node[bag]{$\NODE_{1,2}$ }
	child{
		node[bag]{$\NODE_{1,2,3}$ }
		child{
			node[bag]{$\NODE_{1,2,3,4}$ }
			edge from parent
			node[above] {$\scriptscriptstyle{ \Nqfrac{m-3} }$}
			node[below]{}
		}
		child{
			node[bag]{$\NODE_{1,2,3,3}$ }
			edge from parent
			node[above] {$ q^{m-1}$}
			node[below]{}
		}
		child{
			node[bag]{$\NODE_{1,2,3,2}$ }
			edge from parent
			node[above] {$ q^{m-2}$}
			node[below]{}
		}
		child{
			node[bag]{$\NODE_{1,2,3,1}$ }
			edge from parent
			node[above] {$ q^{m-3}$}
			node[below]{}
		}
		edge from parent
		node[above] {$\scriptscriptstyle{ \Nqfrac{m-2}} $}
		node[below]{}
	}
	child{
		node[bag]{$\NODE_{1,2,2}$ }
		child{
			node[bag]{$\NODE_{1,2,2,3}$ }
			edge from parent
			node[above] { $\scriptscriptstyle{ \Nqfrac{m-2}} $ }
			node[below]{}
		}
		child{
			node[bag]{$\NODE_{1,2,2,2}$ }
			edge from parent
			node[above] {$ q^{m-1}$}
			node[below]{}
		}
		child{
			node[bag]{$\NODE_{1,2,2,1}$ }
			edge from parent
			node[above] {$ q^{m-2}$}
			node[below]{}
		}
		edge from parent
		node[above] {$ q^{m-1}$}
		node[below]{}
	}
	child{
		node[bag]{$\NODE_{1,2,1}$ }
		child{
			node[bag]{$\NODE_{1,2,1,3}$ }
			edge from parent
			node[above] {  $\scriptscriptstyle{ \Nqfrac{m-2}} $ }
			node[below]{}
		}
		child{
			node[bag]{$\NODE_{1,2,1,2}$ }
			edge from parent
			node[above] {$ q^{m-1}$}
			node[below]{}
		}
		child{
			node[bag]{$\NODE_{1,2,1,1}$ }
			edge from parent
			node[above] {$ q^{m-2}$}
			node[below]{}
		}
		edge from parent
		node[above] {$ q^{m-2}$}
		node[below]{}
	}
	edge from parent
	node[above] {$\scriptscriptstyle{ \Nqfrac{m-1}}  $}
	node[below]{}
}
child{
	node[bag]{$\NODE_{1,1}$  }
	child{
		node[bag]{$\NODE_{1,1,2}$ }
		child{
			node[bag]{$\NODE_{1,1,2,4}$ }
			edge from parent
			node[above] {  $\scriptscriptstyle{ \Nqfrac{m-2}}  $}
			node[below]{}
		}
		child{
			node[bag]{$\NODE_{1,1,2,3}$ }
			edge from parent
			node[above] {$ q^{m-2}$}
			node[below]{}
		}
		child{
			node[bag]{$\NODE_{1,1,2,2}$ }
			edge from parent
			node[above] {$ q^{m-2}$}
			node[below]{}
		}
		child{
			node[bag]{$\NODE_{1,1,2,1}$ }
			edge from parent
			node[above] {$ q^{m-2}(q-1)$}
			node[below]{}
		}
		edge from parent
		node[above] {$\scriptscriptstyle{ \Nqfrac{m-1}}  $}
		node[below]{}
	}
	child{
		node[bag]{$\NODE_{1,1,1}$ }
		child{
			node[bag]{$\NODE_{1,1,1,2}$ }
			edge from parent
			node[above] {$\scriptscriptstyle{ \Nqfrac{m-1}}  $}
			node[below]{}
		}
		child{
			node[bag]{$\NODE_{1,1,1,1}$ }
			edge from parent
			node[above] {$ q^{m-1}$}
			node[below]{}
		}
		edge from parent
		node[above] {$ q^{m-1}$}
		node[below]{}
	}
	edge from parent
	node[above] {$ q^{m-1}$}
	node[below]{}
};
\end{tikzpicture}
\end{figure}
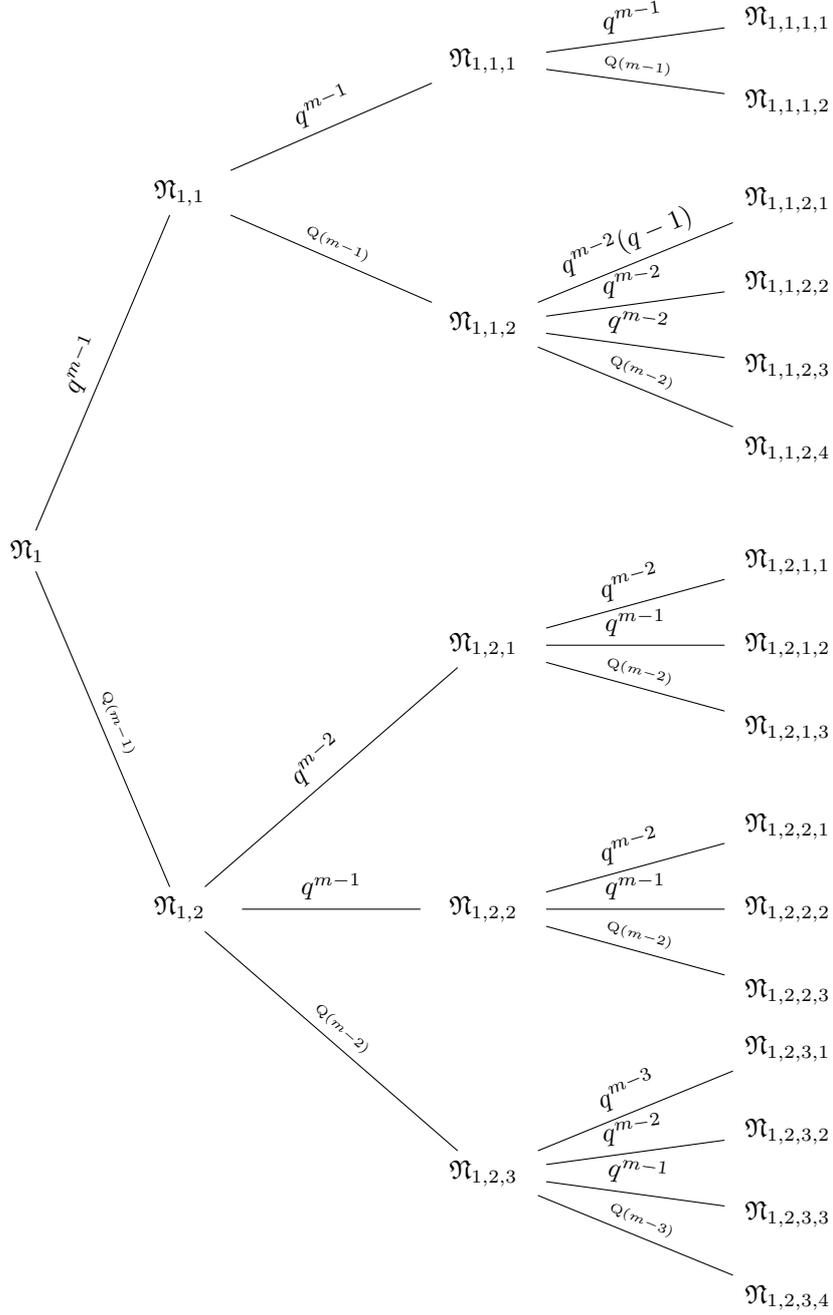 

\bigskip\subsection{Matrix and covering degrees of $T$-torsion flag classes}\label{subSec:matrixcoveringdegrees}

Recall that we have defined  matrix degrees of conjugacy classes, i.e., nodes of the tree of nilpotent upper-triangular matrices (cf. Definition \ref{Defn:matrixdegree}). We wish to find the analogous description of such degrees in terms of $T$-torsion flag classes, i.e., nodes of the $T$-torsion tree $\Ttorsiontree^m$.

 Let  $ [ \mathcal{V}]$ be a $T$-torsion $n$-flag class   represented by the $T$-torsion $n$-flag  $\mathcal{V} =\flag{ V_1\subset V_2\subset \cdots\subset V_{n } }$ in $S^m$.
 We also fix a child class $[\mathcal{V}']$ of $[\mathcal{V}]$. Consider the set of vectors $\nu_{n+1}$ that can generate $[\mathcal{V}']$ from $\mathcal{V}$, i.e.
 \begin{eqnarray*}
 	&&M(\mathcal{V} ;   [\mathcal{V}'] )\\
 	&:=&\{\nu_{n+1} \in {S^m} \setminus V_n \mbox{ s.t. }T\nu_{n+1}\in V_n \mbox{ and the }  (n+1)\mbox{-$T$-torsion flag } \\
 	&&\quad\flag{ V_1\subset V_2\subset \cdots\subset V_{n }\subset V_{n+1}}  \mbox{ represents }   [\mathcal{V}'], \mbox{ where }V_{n+1}=V_n\oplus \Fq \nu_{n+1} \}.
 \end{eqnarray*}
 Two elements $\nu_{n+1}$ and $\tilde{\nu}_{n+1}\in M(\mathcal{V} ;   [\mathcal{V}'] ) $ are said to be \textbf{equivalent} if there exists some nonzero constant $k\in \Fq^* $ such that
 $$T(\nu_{n+1}-k\tilde{\nu}_{n+1})\in TV_n,\mbox{ or }~~ \nu_{n+1}-k\tilde{\nu}_{n+1}\in T^{-1}(TV_n) .$$
 We denote by $ {M}  (\mathcal{V} ;  [\mathcal{V}'] )/_\sim $ the corresponding quotient set modulo equivalences.


 \begin{defn}\label{Def:matrixdegreeflaglcass} With notations above, we define the \textbf{matrix degree} of the $T$-torsion flag class $    [\mathcal{V}'] $ by
 	$$\matrixdeg_{      [\mathcal{V}'] }:= { \#  {M}( \mathcal{V} ;   [\mathcal{V}']) }/_\sim  .$$
 \end{defn}
 We remark that this definition does not depend on the choice of the $T$-torsion  flag $\mathcal{V} $ representing the parent $T$-torsion   flag class   of $[\mathcal{V}']$.
The following facts are easily seen.
 \begin{prop}\label{Prop:twomatrixdegreescoincides} If a $T$-torsion $(n+1)$-flag class $ [\mathcal{V}'] $   corresponds to the conjugacy class $[N']$   (according to Proposition \ref{Prop:Nnm-Sn}), then we have $$\matrixdeg_{      [\mathcal{V}'] }=\matrixdeg_{[N']}.$$
 \end{prop}

 \begin{prop}\label{Prop:MXYdegenerate}With the same settings as earlier, the following statements are equivalent:
 	\begin{itemize}
 		\item[(1)] The $T$-torsion $(n+1)$-flag class $      [\mathcal{V}'] $ is  $T$-vanishing;
 			\item[(2)] We have   $M(\mathcal{V};[\mathcal{V'}]) =  T^{-1}(TV_n) \backslash V_n$;
 			\item[(3)] There exists some  $\nu_{n+1}\in  M(\mathcal{V} ;  [\mathcal{V}'] )$  such that $T \nu_{n+1}=0$;  		
 			 		\item[(4)] The conjugacy class $[N']$ corresponding to $ [\mathcal{V}'] $ is  trivially extended (see Definition \ref{Defn:conjugacyclassdirectson}). 	\end{itemize}
 \end{prop}

 Therefore, according to Propositions \ref{Prop:twomatrixdegreescoincides} and \ref{Prop:conjugacyclassdirectchild}, the matrix degree of a  $T$-vanishing torsion  flag class is $1$.


We have another type of degree.
\begin{defn}\label{Defn:coveringdegree}
	For a $T$-torsion flag class $     [\mathcal{V}']    $ in $S^m$, we define its \textbf{covering degree} by $$\coveringdeg_{      [\mathcal{V}'] }:=\frac{ \#  [\mathcal{V}']  }{ \#  [\mathcal{V}] } ,$$
where  $[\mathcal{V}]$ is the parent $T$-torsion flag class of $[\mathcal{V}']$.
	\end{defn}
	Equivalently, one can fix a representative $\mathcal{V}  = \flag{V_1\subset V_2\subset \cdots \subset V_n}$ of    $ [\mathcal{V}]  $, and then the covering degree $\coveringdeg_{      [\mathcal{V}'] }$ is given by the number of    $(n+1)${-dimensional subspaces} $V_{n+1}\subset {S^m}$  such that $V_n\subsetneq V_{n+1}$ and the $T$-torsion $(n+1)$-flag $\flag{ V_1\subset V_2\subset \cdots\subset V_{n }\subset V_{n+1}}$ represents  $[\mathcal{V}']$.

	For $\nu_{n+1}$, $\tilde{\nu}_{n+1} \in M(\mathcal{V};[\mathcal{V'}])$, the direct sum $V_{n+1}=V_n \oplus \F_q\nu_{n+1}$ coincides with $\tilde{V}_{n+1}=V_n \oplus \F_q\tilde{v}_{n+1}$ if and only if there exists $k \in \F_q^*$ such that $\nu_{n+1}-k\tilde{\nu}_{n+1} \in V_n$. Therefore, $\coveringdeg_{[\mathcal{V}']}$ is  the number of orbits by the obvious $\F_q^* \times V_n$-action on $M(\mathcal{V};[\mathcal{V'}])$: 	
 	\begin{equation}\label{Eqt:cardinalityofcoveringdegree}
 		\coveringdeg_{ [\mathcal{V}'] } =
		\# (M(\mathcal{V};[\mathcal{V'}])/(\F_{q}^*\times V_{n}\mbox{-action})).
		\end{equation}

\begin{prop}\label{prop3.13}  Let  $    [\mathcal{V}]    $ be a $T$-torsion $n$-flag class  and suppose that it is of $T$-order $(n-r)$.
	For any child  class $     [\mathcal{V}']    $ of  $     [\mathcal{V}]    $, we have
the following statements:	\begin{enumerate}
		\item If $[\mathcal{V}']$ is  $T$-vanishing, then we have
		$$\coveringdeg_{      [\mathcal{V}'] }=\Nqfrac{m-n+r}=\frac{q^{m-n+r}-1}{q-1};
		$$
		\item If $[\mathcal{V}']$ is not $T$-vanishing, then we have
		$$\coveringdeg_{      [\mathcal{V}'] }=   q^{m-n+r}\cdot \matrixdeg_{       [\mathcal{V}'] }.
		$$
		
	\end{enumerate}
\end{prop}
\begin{proof}Fix a representative $\mathcal{V}  = \flag{V_1\subset V_2\subset \cdots \subset V_n}$ of    $ [\mathcal{V}]  $.
	
(1) Since $[\mathcal{V}']$ is $T$-vanishing and according to  Proposition \ref{Prop:MXYdegenerate}, one can write $$M(\mathcal{V};[\mathcal{V'}]) =  T^{-1}(TV_n) \backslash V_n.$$
The dimension of $T^{-1}(TV_n)$ is clearly $(m+r)$.  Now we can compute $\coveringdeg_{[\mathcal{V}']}$ using Equation \eqref{Eqt:cardinalityofcoveringdegree}:
\begin{eqnarray*}
	\# (M(\mathcal{V};[\mathcal{V'}])/(\F_{q}^*\times V_{n}\mbox{-action}))
	= \frac{q^{m+r-n}-1}{q-1}.
\end{eqnarray*}
(2) Since $[\mathcal{V}']$ is not $T$-vanishing, one can define an $\F_{q}^*\times T^{-1}(TV_{n})$-action on $M(\mathcal{V};[\mathcal{V'}])$ in an obvious manner. Then, we can use Equation \eqref{Eqt:cardinalityofcoveringdegree} again to obtain:
\begin{eqnarray*}
\coveringdeg_{[\mathcal{V}']}	&=& 	\# (M(\mathcal{V};[\mathcal{V'}])/(\F_{q}^*\times V_{n}\mbox{-action}))\\
&=& 	\# (M(\mathcal{V};[\mathcal{V'}])/(\F_{q}^*\times T^{-1}(TV_{n})\mbox{-action}))
\cdot \# (T^{-1}(TV_{n})/V_n)\\
&=& \matrixdeg_{[\mathcal{V}']} \cdot q^{m+r-n}.
	\end{eqnarray*}

 \end{proof}

\begin{prop}
	 Let  $    [\mathcal{V}]    $ be a $T$-torsion $n$-flag class. Assume that it has totally $i$ child $T$-torsion  classes $      [\mathcal{V}^{(j)}] $, $j=1,\cdots,i$. Then we have
	 \begin{equation*}
	 	\sum_{j=1}^i \coveringdeg_{     [\mathcal{V}^{(j)}]  }=\Nqfrac{m}=\frac{q^m-1}{q-1}.
	 \end{equation*}
\end{prop}
\begin{proof}
	By definition of the covering degree, we have
	\begin{eqnarray*}
		\sum_{j=1}^i \coveringdeg_{     [\mathcal{V}^{(j)}]  }&=&\#\left(\bigcup_{j=1}^{i}M(\mathcal{V},[\mathcal{V}^{(j)}])/(\F_{q}^*\times V_{n}\mbox{-action})\right)\\
		&=& \#\bigl((T^{-1}V_n \setminus V_n )/(\F_{q}^*\times V_{n}\mbox{-action})\bigr)=\frac{q^{m}-1}{q-1}.
	\end{eqnarray*}

  \end{proof}

\subsection{The table of nodes up to level $4$}\label{Sec:TheTable}~
 Throughout the paper,  letters  $r$, $s$, $t$, $x$, $y$, and $z$ denote arbitrary   constants  in $\Fq $ with $r\neq 0$, $s\neq 0$ and $t\neq 0$.
In the table below, we collect some information of the 24 nodes of the  $T$-torsion tree $\Ttorsiontree^m$   for $m\geqslant 4$. If $m=3$, the last line is removed. If $m=2$, then eight lines should be removed (see explanations in Figure \ref{fig:Ttree}). 


 \begin{supertabular}{|p{1cm}|p{2.4cm}|p{2.4cm}| p{7cm} |  }
 	\hline
 	Node 	&  $ \Trank $ ($T$-order)& $\Exp $ (exponent) & Formulas of  $T$ on base vectors ($T\nu_1=0$)    \\\hline
 	$\NODE_{1    }$ & 1&1  &        \\\hline
 	$\NODE_{1,1   }$ & 1&2  &  $T\nu_2=r \nu_1$     \\\hline
 	$\NODE_{1,2 }$&2&1&
 	$  T\nu_2=0$    \\\hline
 	
 	$\NODE_{1,1,1 }$ & 1&3  &  $T\nu_2=r \nu_1$,    $T\nu_3=x\nu_1+s\nu_2$    \\\hline
 	$\NODE_{1,1,2 }$ &2&2&  $T\nu_2=r \nu_1$,
 	$T\nu_3=x\nu_1$    \\\hline
 	$\NODE_{1,2,1 }$ &2&2&
 	$ T\nu_2=0 $,
 	$T\nu_3=r\nu_1 $   \\\hline
 	$\NODE_{1,2,2 }$&2&2&
 	$ T\nu_2=0$, $
 	T\nu_3=x\nu_1+s\nu_2 $    \\\hline
 	$\NODE_{1,2,3 }$&3&1&
 	$  T\nu_2=0$,
 	$T\nu_3=0 $   \\\hline
 	
 	$\NODE_{1,1,1,1}$ & 1&4  &  $T\nu_2=r \nu_1$,    $T\nu_3=x\nu_1+s\nu_2, \qquad\qquad\qquad $     $\qquad\qquad\qquad T\nu_4=y\nu_1+z\nu_2+t\nu_3$   \\\hline
 	$\NODE_{1,1,1,2}$ &2&3
 	&  $T\nu_2=r \nu_1$, $T\nu_3=x\nu_1+s\nu_2$, $T\nu_4=y\nu_1+z\nu_2$   \\\hline
 	$\NODE_{1,1,2,1}$ &2&3&  $T\nu_2=r \nu_1$,
 	$T\nu_3=x\nu_1$, $T\nu_4=y\nu_1+z\nu_2+t\nu_3$\ \   ($rz+tx\neq 0$)   \\\hline
 	$\NODE_{1,1,2,2}$ &2&3&   $T\nu_2=r \nu_1$,
 	$T\nu_3=x\nu_1$, $T\nu_4=y\nu_1+s\nu_2$   \\\hline
 	$\NODE_{1,1,2,3}$&2&2&
 	$ T\nu_2=r \nu_1$,
 	$T\nu_3=x\nu_1 $, $T\nu_4=y\nu_1-\frac{tx}{r}\nu_2+t\nu_3$   \\\hline
 	$\NODE_{1,1,2,4}$&3&2&
 	$ T\nu_2=r \nu_1$,
 	$T\nu_3=x\nu_1$, $T\nu_4=y\nu_1$ \\\hline
 	$\NODE_{1,2,1,1}$ &2&2&
 	$ T\nu_2=0 $,
 	$T\nu_3=r\nu_1 $, $T\nu_4=x\nu_1+s\nu_2$  \\\hline
 	$\NODE_{1,2,1,2}$ &2&3&
 	$  T\nu_2=0$,
 	$T\nu_3=r\nu_1 $, $T\nu_4=x\nu_1+y\nu_2+s\nu_3$  \\\hline
 	$\NODE_{1,2,1,3}$ &3&2&
 	$ T\nu_2=0$, $T\nu_3=r\nu_1 $, $T\nu_4=x\nu_1$  \\\hline
 	$\NODE_{1,2,2,1}$&2&2&
 	$ T\nu_2=0$, $
 	T\nu_3=x\nu_1+r\nu_2 $, $T\nu_4=s\nu_1+z\nu_2$ $~(xz\neq sr)$ \\\hline
 	$\NODE_{1,2,2,2}$&2&3&
 	$ T\nu_2=0$, $
 	T\nu_3=x\nu_1+r\nu_2, \qquad\qquad\qquad $   $\qquad\qquad\qquad T\nu_4=y\nu_1+z\nu_2+s\nu_3$   \\\hline
 	$\NODE_{1,2,2,3}$&3&2&
 	$  T\nu_2=0$, $
 	T\nu_3=x\nu_1+r\nu_2 $, $T\nu_4=\frac{xz}{r}\nu_1+z\nu_2$ \\\hline
 	$\NODE_{1,2,3,1}$&3&2&
 	$  T\nu_2=0$,
 	$T\nu_3=0 $, $T\nu_4=r\nu_1$    \\\hline
 	$\NODE_{1,2,3,2}$ &3&2& $  T\nu_2=0$,
 	$T\nu_3=0 $, $T\nu_4=x\nu_1+r\nu_2$  \\\hline
 	$\NODE_{1,2,3,3}$&3&2&$
 	T\nu_2=0$,
 	$T\nu_3=0 $, $T\nu_4=x\nu_1+y\nu_2+r\nu_3$ \\\hline
 	$\NODE_{1,2,3,4}$&4&1&$
 	T\nu_2=0$,
 	$T\nu_3=0 $, $T\nu_4=0$   \\\hline
 	
 \end{supertabular}

 \section{The standard correspondence between  arrays and flags}\label{Sec:Standardcorrespondence}

\subsection{Global notations}\label{subSec:31}
 First of all, we  make some   conventions and fix the  notations that will be used throughout the rest of this paper:
 \begin{enumerate}
 	\item As earlier,  $A=\Fq [T]$ denotes the polynomial ring over $\Fq $;
 	\item   $ L $ is a field containing $ \Fq  $; \item $ \iota : A \to L $ is an $ \mathbb{F}_{q} $-algebra homomorphism;
   	\item  Denote by $ L\set{\tau} $ the non-commutative $ L $-algebra which is generated by the $q$-Frobenius endomorphism $ \tau $ such that  $ \tau \cdot a = a^{q } \tau  $ for all $ a \in L$. We refer to the $L$-algebra $ L\set{\tau} $   as a \textbf{twisted polynomial ring} (also known as an Ore ring \cite{MR1503119});
 	\item Let $ \Lbar  $  be the   algebraic closure of $ L $.
 \end{enumerate}
\begin{defn}
	 A field $L$ subject to conditions (2) and (3) will be referred to as an \textbf{$A$-field}.
\end{defn}
 The   action of twisted polynomials on $\Lbar  $ is standard ---  Given   $ f = \sum_{i=0}^{d} g_i \tau^i \in L\set{\tau}  $, its action   $f: \Lbar  \to \Lbar  $ reads
 \[
       \mu \mapsto f(\mu):=\sum_{i=0}^{d} g_i \mu^{q^i}  .
 \]
 The kernel of   $f$  is defined and denoted by	$$\Ker (f):=\{\mu\in \Lbar  | f(\mu)=0\}, $$ which is a finite dimensional  $\Fq $-linear subspace of $\Lbar $.

 \subsection{Flags and the  standard correspondence}\label{Sec:flagsnotations}
 

 Let $n$ be a positive integer.
 We   study  \textbf{$n$-flags}   in $\Lbar$   that are of the form \begin{equation*}
 	\mathcal{V}  =\flag{ V_1\subset V_2\subset \cdots\subset V_n}
 \end{equation*}
 where each $V_i\subset \Lbar$  is a subspace of dimension $i$ (over $\Fq$).
   In contrast to   the $T$-torsion flags   in Section \ref{Sec:basic}, we   do not impose any additional constraints on what is here.

  
  To construct  a specific $n$-flag, we utilize an array  $(u_1, \ldots, u_n) \in (\Lbar^*)^{\times n} $ as \textit{parameters}, and let the $\Fq$-linear subspace $ V_i \subset \Lbar $   be defined by
 $$
 V_i:=\Ker ( \lambdau{u_i} \lambdau{u_{i-1}} \cdots\lambdau{u_1})  $$
   where  we set \begin{equation}\label{Eqt:lambdaui}
   \lambdau{u_i} :=  \tau - u_i\, \quad ( \in \Lbar\set{\tau} ). \end{equation}
   It is clear that
 $V_1\subset V_2\subset \cdots \subset V_n\subset \Lbar$ and each $V_i$ is of dimension $i$.
 Next, we show an important fact.
 \begin{prop}\label{Prop:Fromflagtoui}Using the above notations, the map $\Theta$ defined below is a bijection:
  	$$\Theta:~(\Lbar^*)^{\times n}\to \set{n\mbox{-flags in }\Lbar},\quad (u_1,\cdots, u_n)\mapsto \flag{ V_1\subset V_2\subset \cdots\subset V_n}.
 	$$
 \end{prop}
\begin{proof}
	  To establish this fact, it suffices to find the inverse map \begin{equation}
 \label{Eqt:Lambdafirstappear}
 \Lambda:~\set{n\mbox{-flags in }\Lbar}\to(\Lbar^*)^{\times n}\end{equation} of $\Theta$. The construction is  as follows.
 \begin{itemize}
 	\item[$\bullet$] For any $\mathcal{V}=(V_1\subset V_2\subset \cdots\subset V_n)\in \set{n\mbox{-flags in }\Lbar}$,   choose $\nu_i \in  V_i \setminus V_{i-1}$ ($i=1,\cdots, n$) so that $(\nu_1<\ldots<\nu_n)$ becomes an ordered  basis of  $\mathcal{V} $;
 	\item[$\bullet$] For  all  	$i=1,\ldots,n$, define $l_i$ and $u_i   \in \Lbar^* $ by:
 	\begin{enumerate}
 		\item Set
 		$l_1:=\nu_1$ and $u_1:={l}_1^{q-1}={\nu}_1^{q-1}$;
 		\item Define $    l_2:=\lambdau{u_1}(\nu_2)$ where $\lambdau{u_1}=\tau-u_1$, and define $u_2:={l}_2^{q-1}$;
 		\item Inductively, for $3\leqslant i \leqslant n$,  define $   {l}_i:=\lambdau{u_{i-1}}\cdots \lambdau{u_2}\lambdau{u_1}(\nu_i)$  where $\lambdau{u_{j}}:=\tau-u_{j}$ ($1\leqslant j\leqslant i-1$),  and set $u_i:={l}_i^{q-1}$.
 		
 	\end{enumerate}
 \end{itemize}

 An important feature of this procedure is that if we choose a different ordered bases $(\nu'_1,\ldots,\nu'_n)$ of $\mathcal{V}$, following the above recipe, one gets the same data $( u_1,\cdots,u_n )$. So we can define $$\Lambda(\mathcal{V}):=(u_1,\cdots,u_n) .$$

 We then prove that $\Theta\circ\Lambda(\mathcal{V})=\mathcal{V}$, i.e.,
 $
 V_i = \Ker ( \lambdau{u_i} \lambdau{u_{i-1}} \cdots\lambdau{u_1})
 $. Indeed, we can do  induction on $i$.
 First, $V_1=\Ker(\lambdau{u_1})$ is apparent. Assume that for some $i\geqslant 2$, $V_{i-1}= $ $\Ker ( \lambdau{u_{i-1}} \lambdau{i-2} \cdots\lambdau{u_1})$ is true. To see  $V_i=V_{i-1}\oplus \F_{q}\nu_i \subset  \Ker ( \lambdau{u_i} \lambdau{u_{i-1}} \cdots\lambdau{u_1})$, we only need to show  
 $\nu_i \in \Ker ( \lambdau{u_i} \lambdau{u_{i-1}} \cdots\lambdau{u_1})$. In fact, we have
 $$(\lambdau{u_i} \lambdau{u_{i-1}} \cdots\lambdau{u_1})(\nu_i)=  \lambdau{u_i}  (l_i)= l_i^{q}-u_il_i=0.$$
 Since  $   \Ker ( \lambdau{u_i} \lambdau{u_{i-1}} \cdots\lambdau{u_1}) $ is of dimension $i$, we must have  $V_{i}=\Ker ( \lambdau{u_i} \lambdau{u_{i-1}} \cdots\lambdau{u_1})$.


 We can also  verify the reversed relation  $\Lambda\circ\Theta(u_1,\cdots,u_n)=(u_1,\cdots,u_n)$ which is straightforward. So the map $\Lambda$ we just defined is   the desired inverse of $\Theta$. 
\end{proof}

\begin{defn}\label{corresponding}
	In the sequel, we refer to the one-to-one correspondence
	$$\Lambda: ~\set{n\mbox{-flags in }\Lbar}\rightleftarrows (\Lbar^*)^{\times n}:~\Theta$$
  
as the \textbf{standard correspondence}. It could also refer to any intermediate steps  that connect the   elements
 	\begin{equation}
\label{Eqt:standardconstruction}
	\begin{tikzcd}[column sep=tiny]
	\mathcal{V}=\left(V_1 \subset V_2 \subset \cdots \subset V_n\right)=\left(\nu_1<\ldots<\nu_n\right)\quad \ar[rr, mapsto] \ar[rrrr, "\Lambda",  bend left=30] &&\quad\left(l_1, \cdots, l_n\right)\quad \arrow[rr, mapsto]&& \quad\left(u_1, \cdots, u_n\right). \ar[llll,  "\Theta",  bend left=30]\\
	\end{tikzcd}
 	\end{equation}
\end{defn}

  \begin{defn}\label{Def:definedoverL} 	
 	A flag    $\mathcal{V}=(V_1\subset V_2\subset \cdots\subset V_n)$ in $\Lbar$   is said to be   \textbf{defined  over $L$ }if it is stable under the Galois action, i.e.    $\xi(V_i)=V_i$ holds for all $\xi\in\Aut(\bar L/L)$ and $i=1,\cdots, n$.
 	
 \end{defn}

 \begin{lem}\label{Lem:uiallinLstar}
 	A flag  $\mathcal{V}=(V_1\subset V_2\subset \cdots\subset V_n)$   in $\Lbar$ is  defined  over $L$ if and only if all entries in $(u_1, \ldots, u_n)=\Lambda(\mathcal{V})$ are in  $L^{*} $.	
 \end{lem}
 We skip the proof  as it is an easy induction. Combining Proposition \ref{Prop:Fromflagtoui} and Lemma \ref{Lem:uiallinLstar}, we get the following fact.
 \begin{cor}\label{Cor:FromflagoverLtoui}The standard correspondence induces a pair of mutually inverse bijections 	$$\Lambda: ~\set{n\mbox{-flags in }\Lbar \mbox{ defined over }L}\rightleftarrows (L^*)^{\times n}:~\Theta.$$
 	
 	\end{cor}

In the sequel,  we continue to use the global symbols $A$($=\Fq [T]$), $L$, $L\set{\tau}$ and so on, and the standard correspondence that we set up in this section.

 \section{Drinfeld modules and Drinfeld modular curves}\label{Sec:Drinfeldmodulescurves}
\bigskip

We first give a brief introduction to the notion of Drinfeld module, which was  introduced by Drinfeld in his celebrated work  \cite{Drinfeld1974}. For more in-depth studies, please see \cites{Drinfeld1974,Gekeler1986,Goss1996,TTAFF,MR4321964}. 

\subsection{Drinfeld modules}
 A $\textbf{Drinfeld module}$ over $L$  is  an $ \mathbb{F}_{q } $-algebra homomorphism
 \[\phi :  A \to L\set{\tau},\qquad a\mapsto \phi_a\,,\]
 satisfying the conditions
 \begin{enumerate}
 	\item  There exists an element $ a \in A $ such that $ \phi_ a \not = \iota (a )$; 
 	\item  The following diagram of $ \mathbb{F}_{q } $-algebra homomorphisms
 	\begin{equation*}
 	\begin{tikzcd}
 	A\arrow[rd,"\iota"] \arrow[r,"\phi"] & L\set{\tau}\arrow[d,"\partial_0"]\\
 	& L
 	\end{tikzcd}
 	\end{equation*}
 	is commutative, where $\partial_0$   is the standard point derivation:
 	\[
 	\partial_0:\quad L\set{\tau} \to L , \qquad f=\sum_{i=0}^{m} g_i \tau^i \mapsto g_0.
 	\]
 \end{enumerate}
 Drinfeld modules over $\Lbar$ are defined in the same fashion. For a Drinfeld module $\phi$ as above, the kernel of $\iota$, which is a prime ideal in $A$, is called   the \textbf{characteristic} of $\phi$.
 As $ A$ is the polynomial ring $\Fq [T] $,   a Drinfeld module $ \phi $ is uniquely determined by a twisted polynomial $\phi_T$ over $L$. We suppose that
 \begin{equation*}
 \phi_T =   g_m \tau^m +\cdots + g_2 \tau^2 + g_1 \tau + g_0\,,
 \end{equation*}
 where $ g_ m$  is not $0 $ for some integer $m>0$ and $g_0=\iota(T)$. The number $ m  $ is called the \textbf{rank} of   $ \phi $. 

 For every  $a\in A$, the kernel of $\phi_a $ is an  $ A $-submodule of $\Lbar $. In a special situation described below, the $A$-module structure of $\Ker (\phi_a)$ is explicit.

 \begin{lem}[{\cite{Gekeler1986}*{Proposition I.1.6}}]\label{Lem:moduleiso}
 	If  $a\in A$ is coprime to the characteristic of $\phi$, then
 	\[ \Ker ( \phi_{a} ) \cong  (A/(aA))^{\oplus m} \]
 	as $A$-modules.
 \end{lem}




 \begin{defn}
 	Two Drinfeld modules $\phi$ and $\psi$   over $L$  are said to be \textbf{isomorphic} over $\Lbar $, if   there exists an element $\lambda\in \Lbar ^*$ such that for all $a\in A$, the equation
 	\begin{equation}\label{Eq:defnisomorphic}
 		\lambda\phi_a=\psi_a\lambda
 	\end{equation}
 (which amounts to the condition  $\lambda\phi_T=\psi_T\lambda$)	holds in $\Lbar \set{\tau}$.
 \end{defn}



 \bigskip\subsection{Torsion flags of a Drinfeld module}\label{subSec:torsionflagdm}
 Now, we introduce the key concept  of this paper.
 \begin{defn}\label{Defn:phiTtorsionnflag}Let $\phi: A \to L\set{\tau}$   be a  Drinfeld module. A  \textbf{$\phi_T$-torsion $n$-flag} is an $n$-flag $\mathcal{V}=(V_1\subset V_2\subset \cdots\subset V_n)$   in $\Lbar$ such that
 	\begin{enumerate}
 		\item $\phi_T (V_1)=0$;
 		\item $\phi_T (V_i)\subseteq V_{i-1}$ for all $i=2,\cdots,n$.
 	\end{enumerate}
 \end{defn}

Suppose that the  Drinfeld module $\phi$ under consideration  is of rank $m$ and its characteristic is coprime to $T$ (in $A=\Fq[T]$). Let
 \begin{equation*}
 S^{\phi}:=  \varinjlim_n \Ker(\phi_{T^n})\subset \Lbar \end{equation*}
 be the $\phi_T$-divisible group   of $\phi$, which is an  $A $-submodule of $\Lbar$. 
 
 Note that    $\phi_T$-torsion $n$-flags in $\Lbar$ are sitting in $S^\phi$ indeed. 
  According to Lemma \ref{Lem:moduleiso},   we can find an isomorphism of   $A$-modules
 $$I^\phi:\quad S^{\phi} \stackrel{\sim }{\longrightarrow}  {S^m}.  $$  (See Equation \eqref{Eqt:Sm} for the definition of $S^m$.)  Note that the isomorphism $I^\phi$ is not unique. 
Because the two conditions in Definition \ref{Defn:phiTtorsionnflag}  imply that each $V_i$ is contained in $S^\phi$,  $I^\phi\mathcal{V}$ is a $T$-torsion  $n$-flag in  $  S^m$. So, $I^\phi $ establishes a one-to-one correspondence between $\phi_T$-torsion $n$-flags in $S^\phi$ and $T$-torsion $n$-flags in $S^m$.

We can also define  isomorphism of $\phi_T$-torsion flags   in the same fashion as in Definition \ref{isomorphism}. 
\begin{defn}\label{isomorphism2}
	Two $\phi_T$-torsion $n$-flags $\mathcal{V} =\flag{V_1\subset V_2\subset \cdots\subset V_n}$ and $\mathcal{W}=\flag{ W_1\subset W_2\subset \cdots\subset W_n}$ of ${S^{\phi}}$ are said to be \textbf{isomorphic} if there exists an $\Fq $-linear isomorphism
	$\iota: V_n\rightarrow W_n$
	such that
	\begin{enumerate}
		\item $\iota(V_i)=W_i$ for $i=1,\ldots, n$; and
		\item $\iota\circ \phi_{T}=\phi_{T}\circ \iota$.
	\end{enumerate}
	
\end{defn}

A  \textbf{$\phi_T$-torsion ($n$-)flag class} in ${S^\phi}$ is an isomorphic class of such $\phi_T$-torsion $n$-flags. As always, for a $\phi_T$-torsion $n$-flag in $S^\phi$, say  $\mathcal{V}$, the corresponding $\phi_T$-torsion  flag class is denoted by $[\mathcal{V}]$.

Next, let us fix  a $T$-torsion $n$-flag class in the $A$-module ${S^m}$, say $\NODE_{i_1,\cdots, i_n}$, namely,  a level $n$ node of the $T$-torsion tree.
 \begin{defn}   	
 	A $\phi_T$-torsion $n$-flag  $\mathcal{V} $ in $\Lbar$   is said to be \textbf{subordinate to}   $\NODE_{i_1,\cdots, i_n}$ if its image   $I^\phi\mathcal{V}$ in $S^m$ sent by $I^\phi\colon S^{\phi} \stackrel{\sim}{\rightarrow}  {S^m}$ represents  the given $T$-torsion $n$-flag class   $\NODE_{i_1,\cdots, i_n}$, i.e., $[I^\phi\mathcal{V}]=\NODE_{i_1,\cdots, i_n}$.
 	
 \end{defn}

 \begin{notation}\label{Notation:pair} For an $A$-field $L$ and a Drinfeld module $\phi$ over $L$,	denote by $ \mathcal{G}(\Lbar,\phi;\NODE_{i_1,\cdots, i_n} )$ the set of all $\phi_T$-torsion $n$-flags in $\Lbar$  which are    subordinate to    $ \NODE_{i_1,\cdots, i_n}$. 	Denote by $ \mathcal{G}(L,\phi;\NODE_{i_1,\cdots, i_n} )$ the set of all $\phi_T$-torsion $n$-flags in $\Lbar$  which are  defined  over $L$ (see Definition \ref{Lem:uiallinLstar}) and subordinate to    $ \NODE_{i_1,\cdots, i_n}$.

 \end{notation}

As $I^\phi$ is an isomorphism, it is clear that $\mathcal{V}$, an $\phi_T$-torsion $n$-flag in $ S^\phi $ that is subordinate to $\NODE_{i_1,\cdots, i_n}$, exists. So $ \mathcal{G}(\Lbar,\phi;\NODE_{i_1,\cdots, i_n} )$ is never empty; but note that $ \mathcal{G}(L,\phi;\NODE_{i_1,\cdots, i_n} )$ could be empty.

 \bigskip\subsection{$(m,j)$-type normalized Drinfeld modular curves}\label{Sec:Drfeldparameters}
 We make some important conventions in subsequent analysis:
 \begin{itemize}
 	\item We assume that $m\geqslant 3$ is an odd number; 
 	\item We only consider \textbf{$(m,j)$-type normalized
 	Drinfeld modules} over $L$, i.e., those of the form
 	\begin{equation} \label{Eq:normalizedtype}
 		\phi_{T} =  -\tau^m + g \tau ^j + 1,\qquad \mbox{ for some } g\in L.
 	\end{equation}
 	Here  $m>j $, and $ m$ and $ j$ are mutually coprime positive integers.  	
 \end{itemize}
\begin{notation}\label{Nota:DLmj}
	We denote by $  \mathcal{D}_L^{(m,j)} $ the set of     $(m,j)$-type  normalized Drinfeld modules.
\end{notation}

  Meanwhile, we have $\Fq(G)$, the function field of the modular curve that   parameterize    $(m,j)$-type normalized Drinfeld modules as in Equation \eqref{Eq:normalizedtype}, following the idea of \cite{Elkies2001}.  \textit{This  function field can also be notated as} ${\dot{\mathcal{F}}}^{(m,j)}_{(0)}$, according to our more general notation rules (see Notation \ref{notation:normalizedfunctionfield}).

 Recall that $\Nqfrac{m}:=\frac{q^m-1}{q-1}$. We call $ J(\phi): = g ^{\Nqfrac{m}} \in L$  the $ J $-\textbf{invariant} of the  Drinfeld module  $ \phi \in  \mathcal{D}_L^{(m,j)}  $ where $g\in L$ is the coefficient of $\tau^j$ in $\phi_T$ (see Equation \eqref{Eq:normalizedtype}).
 A well-known fact is that the isomorphism class of an elliptic curve is completely determined by its  j-invariant. A similar fact  for Drinfeld modules is the following   lemma.
 \begin{lem}[{\cite{Bassa2015}*{Section 4}}]\label{Lem:Isomorphic}
 	For two  $(m,j)$-type normalized Drinfeld modules   $ \phi$ and  $ \tilde{\phi} $   which are represented  by
 	\begin{equation*}
 		\phi_{T} = -\tau^m + g \tau ^j +1  \mbox{~and~}~ \tilde{\phi}_{T}  = -\tau^m + \tilde{g} \tau ^j +1,
 	\end{equation*}  respectively,
 	the following statements are equivalent:
 	\begin{enumerate}
 		\item The Drinfeld modules $\phi$ and $\tilde{\phi}$ are isomorphic over $\Lbar  $;
 		\item There exists some $ \lambda \in  \mathbb{F}_{q^m}^{* } $ such that $ g  =\tilde{g}  \lambda^ {q^j -1 } $;
 		\item The $J$-invariants of $\phi$ and $\tilde{\phi}$ coincide: $ J(\phi) = J(\tilde{\phi}) $.
 	\end{enumerate}
 \end{lem}

 In other words, $L$-points of the affine line  whose   function field is $\Fq(J)$ parameterize    isomorphic classes of   $(m,j)$-type normalized Drinfeld modules.  To be consistent  with the notation  that we will use later, \textit{we denote this function field by} ${\mathcal{F}}^{(m,j)}_{(0)}$.

Next,  we proceed to  the primary subject for investigation in this paper --- certain   Drinfeld modular curves that are connected with a given $T$-torsion flag  class.   Let us   clarify such objects. 
Fix a $T$-torsion $n$-flag class $\NODE_{i_1,\cdots, i_n}$  of the $A$-module ${S^m}$.
What we give below is an   \textit{informal} definition of the notion of normalized Drinfeld modular curves (of $(m,j)$-type) related to the said flag class $\NODE_{i_1,\cdots, i_n}$.

\begin{defn}   \label{Def:normalizedDrinfeldmodularcurve}  An algebraic curve $C$ (over $\Fq$) is called  the $(m,j)$-type normalized Drinfeld modular curve subordinate to $\NODE_{i_1,\cdots, i_n}$ if for any $A$-field $L$, the $L$-points of $C$  parameterize
	pairs $(\phi, \mathcal{V})$ where $\phi\in  \mathcal{D}_L^{(m,j)} $ (see Notation \ref{Nota:DLmj}) and $\mathcal{V} \in \mathcal{G}(L,\phi;\NODE_{i_1,\cdots, i_n} )$ (see Notation \ref{Notation:pair}).   The notation of such a curve $C$ is $\dot{X}^{(m,j)}_{i_1,\cdots,i_{n}  }$.
	
\end{defn}

In the rest of this paper,  {we  assume that the curve}   $\dot{X}^{(m,j)}_{i_1,\cdots,i_{n}  }$ \textbf{in question exists and is geometrically irreducible}. We will verify this fact for the $(m,j)=(3,2)$-case in the next Section \ref{Sec:32flagclassfunctionfield}. (For those curves we reviewed in the introduction,   their existence are clear.)

\begin{notation}\label{notation:normalizedfunctionfield}
	We denote by
	$\dot{\mathcal{F}}^{(m,j)}_{i_1,\cdots,i_{n}}$  the function field of $\dot{X}^{(m,j)}_{i_1,\cdots,i_{n}  }$.
\end{notation}

Inspired by  the key reference  \cite{Bassa2015}, we show a method  to obtain  $\dot{X}^{(m,j)}_{i_1,\cdots,i_{n}  }$. Recall from Section \ref{Sec:flagsnotations} that we have established a one-one correspondence between   flags $\mathcal{V} =\{V_1\subset\cdots \subset V_n\}$  defined  over $L$ and arrays $(u_1,\cdots, u_n)\in (L^*)^{\times n}$ (Corollary \ref{Cor:FromflagoverLtoui}).
So, the basic idea is to      \textbf{use   parameters $(u_1,\cdots, u_n)\in (L^*)^{\times n}$ to  parameterize   pairs}
$ (\phi, \mathcal{V} )$ such that $\mathcal{V}=\Theta(u_1,\cdots, u_n)$.

We wish to find how $\phi$ is related with the parameters  $(u_1,\cdots, u_n)$.
Resume notations and the standard correspondence \eqref{Eqt:standardconstruction} that we used in Section \ref{Sec:flagsnotations}. In particular, we have $V_1=\Fq \nu_1=\Ker(\tau-u_1)$ and $u_1=\nu_1^{q-1}$.
As earlier, we assume that
$	\phi_{T} = - \tau^m + g   \tau ^j + 1$.
Since  it is required $ {\phi_T}(V_1)=0$, or $\nu_1 \in \Ker (\phi_T)$,  we can derive that
\begin{equation}\label{Eqt:gintermsofu1}
g=\frac{ \nu_1 ^{q^m -1 } -1 }{  \nu_1 ^{q^j -1}}=\frac{ u_1^{\Nqfrac{m} } -1 }{  u_1^{\Nqfrac{j}}}.
\end{equation}

So $g$ is determined only by $u_1$.
For this reason, we  parameterize    $\phi\in  \mathcal{D}_L^{(m,j)}$ by $u_1$ as follows:
\begin{equation}\label{Eqt:phiTui00}
\phi_T= \phiTu{u_1}: =  -\tau ^m + \frac{ u_1^{\Nqfrac{m} } -1 }{  u_1^{\Nqfrac{j}}} \tau ^j +1  .
\end{equation}


It remains to investigate under what condition the $n$-flag $\mathcal{V} =\{V_1\subset\cdots \subset V_n\}$ would be a $\phi_T$-torsion $n$-flag {subordinate to} the given flag class $\NODE_{i_1,\cdots, i_n}$. This is usually a very technical problem (depending on the specific data of $\NODE_{i_1,\cdots, i_n}$) and one ends up with  certain constraints on the  parameters $(u_1,\cdots, u_n)$, which intuitively give  the desired modular curve.  In the next Section   \ref{Sec:32flagclassfunctionfield}, we will give a  theoretical approach to     the modular curve subordinate to any node $\NODE_{i_1,\cdots, i_n}$ for the particular  $(m,j)=(3,2)$-case.


\bigskip
\subsection{$(m,j)$-type Drinfeld modular curves}
\label{Sec:mjtypeDrinfeldmodularcurves}
 
\begin{defn}\label{Def:isomorphicDrinfeldpairs}
	Consider  pairs
	$( \phi,\mathcal{V}  )$,
	where $\phi\in  \mathcal{D}_L^{(m,j)} $  and $\mathcal{V} \in  \mathcal{G}(L,\phi;\NODE_{i_1,\cdots, i_n} )$.  	
	Two such pairs  $(\phi,\mathcal{V} )$ and $(\tilde{\phi},\tilde{\mathcal{V}}  )$ are said to be equivalent   if    there exists some  $\chi\in \Lbar ^*$    such that
	\begin{itemize}
		\item[1)] $\chi \phi_T=\tilde{\phi}_T\chi  $ in $\Lbar \set{\tau}$ (i.e., $\phi$ and $\tilde{\phi}$ are isomorphic Drinfeld modules);
		\item[2)] 	$\chi V_i=\tilde{V_i}$ in $\Lbar $ { for } $i=1,\ldots,n  $,
		where $\mathcal{V} =(V_1\subset\cdots \subset V_n)$ and $\tilde{\mathcal{V}}  =(\tilde{V}_1\subset\cdots \subset \tilde{V}_n)$.
		
	\end{itemize}

\end{defn}

\begin{defn}[Drinfeld modular curve,  an informal definition]\label{Def:Drinfeldmodularcurve}
	An algebraic curve $C$ is called the  Drinfeld modular curve subordinate to $\NODE_{i_1,\cdots, i_n}$ if for all $A$-field $L$, the  $L$-points on $C$ parameterize      equivalent classes of pairs
	$ (\phi, \mathcal{V} )$,
	where $\phi\in  \mathcal{D}_L^{(m,j)} $  and $ \mathcal{V} \in  \mathcal{G}(L,\phi;\NODE_{i_1,\cdots, i_n} )$. The notation of such a curve is $X^{(m,j)}_{i_1,\cdots,i_{n}  }$.
\end{defn}

There is an obvious covering map  from $\dot{X}^{(m,j)}_{i_1,\cdots,i_{n}  }$ to $ {X}^{(m,j)}_{i_1,\cdots,i_{n}  }$ , sending $L$-points 	$ (\phi, \mathcal{V} )$ (on $\dot{X}^{(m,j)}_{i_1,\cdots,i_{n}  }$)  to its equivalent class 	$ [(\phi, \mathcal{V} )]$ (on $ {X}^{(m,j)}_{i_1,\cdots,i_{n}  }$). If the normalized modular curve $\dot{X}^{(m,j)}_{i_1,\cdots,i_{n}  }$   exists and is geometrically  irreducible, then so is the    modular curve $ {X}^{(m,j)}_{i_1,\cdots,i_{n}  }$.

\begin{notation}\label{notation:normalizedfunctionfield2}
	We denote by
	$ {\mathcal{F}}^{(m,j)}_{ i_1,\cdots, i_n }$  the function field of    ${X}^{(m,j)}_{ i_1,\cdots,i_{n}   }$.
\end{notation}

Now we examine how the equivalence relation defined by Definition  \ref{Def:isomorphicDrinfeldpairs} is characterized by parameters. Suppose that   an array   $(u_1,\cdots, u_n)\in (L^*)^{\times n}$ corresponds to $(\phi,\mathcal{V})$ where $ \phi\in \mathcal{D}_L^{(m,j)} $ is given by   $\phi_T:=\phiTu{u_1}$  and $\mathcal{V}:=\Theta(u_1,\cdots, u_n)$ is a $\phi_T$-torsion $n$-flag  {subordinate to} $\NODE_{i_1,\cdots, i_n}$. Suppose that $(\tilde{u}_1,\cdots, \tilde{u}_n)\in (L^*)^{\times n}$ corresponds to $(\tilde{\phi},\tilde{\mathcal{V}})$ in the same fashion.
If $\chi \in \Lbar ^*$  gives an isomorphism between the pairs  $(\phiTu{u_1},\mathcal{V} )$ and $(\phiTu{\tilde{u}_1},\tilde{\mathcal{V}}  )$ as described by Definition \ref{Def:isomorphicDrinfeldpairs}, then from the first condition $\chi\phi_T=\tilde{\phi}_T\chi$, one easily derives that $\chi^{q^m-1}=1$, i.e. $\chi\in \F_{q^m}^*\subset \Lbar^*$.

From the second condition, $\chi V_i=\tilde{V}_i$,  an ordered basis    $(\nu_1,\nu_2,\cdots,\nu_n)$ of $\mathcal{V}$ yields an ordered basis $(\tilde{\nu}_1,\tilde{\nu}_2,\cdots,\tilde{\nu}_n)$ of $\tilde{\mathcal{V}}$ where $\tilde{\nu}_i=\chi \nu_i$.  Recall Proposition \ref{Prop:Fromflagtoui} and we recover  $(\tilde{u}_1,\cdots, \tilde{u}_n)$ as follows. First, we have
$\tilde{u}_1=(\tilde{\nu}_1)^{q-1}=\chi^{q-1}\nu_1^{q-1}=\chi^{q-1}u_1$.
Then from $\tilde{l}_2=(\tau-\tilde{u}_1) \tilde{\nu}_2$ we get $\tilde{l}_2=\chi^{q} l _2$ and
$
\tilde{u}_2=(\tilde{l}_2)^{q-1}=\chi^{q(q-1) }l^{q-1}_2=\chi^{q(q-1) } u_2
$.
By an induction    we get
$$
\tilde{u}_3=\chi^{q^2(q-1) } u_3,~\ldots, ~\tilde{u}_n=\chi^{q^{n-1}(q-1) } u_n \,.
$$

This fact is reinterpreted by the following proposition.
\begin{prop}  The group $\F_{q^m}^*/ \F_q^*$ acts on the normalized Drinfeld modular curve $\dot{X}^{(m,j)}_{i_1,\cdots,i_{n}  }$ by
	\begin{equation}\label{Eqt:chiaction}\chi  (u_1,u_2,\cdots, u_n) := (\chi^{q-1}u_1,\chi^{q(q-1) }u_2,\cdots, \chi^{q^{n-1}(q-1) }  u_n),\quad \forall \chi\in \F_{q^m}^*.
	\end{equation}
\end{prop}
The action of $\F_{q^m}^*/ \F_q^*$   is transitive along each fiber of the covering $\dot{X}^{(m,j)}_{i_1,\cdots,i_{n}  }$ to $ {X}^{(m,j)}_{i_1,\cdots,i_{n}  }$.
Hence the field extension $ \dot{\mathcal{F}}^{(m,j)}_{i_1,\cdots,i_{n}} /  {\mathcal{F}}^{(m,j)}_{i_1,\cdots,i_{n}} $ is Galois with group $\F_{q^m}^*/ \F_q^*$, and $ {\mathcal{F}}^{(m,j)}_{i_1,\cdots,i_{n}}$ is the fixed field of group $\F_{q^m}^*/ \F_q^*$.


\begin{cor}
	The degree of ${ \dot{\mathcal{F}} }^{(m,j)}_{i_1,\cdots,i_{n}}/  { \mathcal{F} }^{(m,j)}_{i_1,\cdots,i_{n}}$ equals $ \Nqfrac{m}$.
	
\end{cor}

Next, we wish to express $ {\mathcal{F}}^{(m,j)}_{i_1,\cdots,i_{n}}$  using $ \dot{\mathcal{F}}^{(m,j)}_{i_1,\cdots,i_{n}}$. To this end, consider a series of functions in $ \dot{\mathcal{F}}^{(m,j)}_{i_1,\cdots,i_{n}}$ which are fixed by the $\F_{q^m}^*/ \F_q^*$-action:
\begin{equation}
\label{Eqt:wianduim-jcase}
w_1=u_1^{\Nqfrac{m}},~w_2=\frac{u_2}{u_1^{q}},~w_3=\frac{u_3}{u_2^{q}},~\ldots,~w_n=\frac{u_n}{u_{n-1}^{q}}.
\end{equation}

\begin{prop}\label{Prop:fromutow}
	For a function field $\Fq(u_1,\cdots,u_n)$ which is equipped with an action by the group $\F_{q^m}^*/ \F_q^*$ as in Equation \eqref{Eqt:chiaction}, the associated fixed field   is given by $\Fq(w_1,\cdots, w_n)$ where the generators  $w_1$, $\cdots$, $w_n$  are given as in Equation \eqref{Eqt:wianduim-jcase}.

\end{prop}
\begin{proof}
	If $n=1$, the statement is clear. Since $\Nqfrac{m}$ and $q$ are coprime, the field extension $\Fq(u_1)/\Fq(w_1)$ is a Kummer extension.
	
	Then we prove  by induction on $n$. Suppose that   the statement is true for $n=k$. For the case $n=k+1$, we have
	$$\Fq(u_1,\cdots, u_k,u_{k+1})=\Fq(u_1,\cdots, u_k)(w_{k+1})$$
	where $w_{k+1}=\frac{u_{k+1}}{u_{k}^{q}}$. Suppose that  $y=\sum_{i=0}^{N}\alpha_i w_{k+1}^i$ is an element in $\Fq(u_1,\cdots, u_k,u_{k+1})$ where $\alpha_i\in  \Fq(u_1,\cdots, u_k)$,  and $y$ is fixed by the  $\F_{q^m}^*/ \F_q^*$-action. Then every coefficient $\alpha_i$ is also fixed by $\F_{q^m}^*/ \F_q^*$, and hence we have $\alpha_i\in  \Fq(w_1,\cdots, w_k)$.
	
\end{proof}



In conclusion, to find the function field   $ {\mathcal{F}}^{(m,j)}_{i_1,\cdots,i_{n}}$ of the  Drinfeld modular curve ${X}^{(m,j)}_{i_1,\cdots,i_{n}  }$, we need to follow some inductive steps:
\begin{enumerate}
	\item For $n=1$, $ {\mathcal{F}}^{(m,j)}_{1}$ is trivially the function field $\Fq(w_1)$.
	\item Suppose that $ {\mathcal{F}}^{(m,j)}_{i_1,\cdots,i_{n-1}}$ is already found. One uses the algebraic relations of $(u_1,\cdots,u_n)$ for $ \dot{\mathcal{F}}^{(m,j)}_{i_1,\cdots,i_{n}}$ to derive the minimal polynomial of $w_n$ in ${\mathcal{F}}^{(m,j)}_{i_1,\cdots,i_{n-1}}$. Then we have
	$ {\mathcal{F}}^{(m,j)}_{i_1,\cdots,i_{n}}=  {\mathcal{F}}^{(m,j)}_{i_1,\cdots,i_{n-1}}(w_n)$, the desired function field.
\end{enumerate}

Here is the main result of this section.

\begin{thm}\label{prop:coveringdegree1}For any child node  $\NODE_{i_1,\cdots, i_n,i_{n+1}}$ of   $\NODE_{i_1,\cdots, i_n}$,  both  degrees  of the field extensions  $  {\dot{\mathcal{F}}}^{(m,j)}_{i_1,\cdots,i_{n},i_{n+1}} /{\dot{\mathcal{F}}}^{(m,j)}_{i_1,\cdots,i_{n}}$ and   $ {{\mathcal{F}}}^{(m,j)}_{i_1,\cdots,i_{n},i_{n+1}} / {{\mathcal{F}}}^{(m,j)}_{i_1,\cdots,i_{n}}  $ are   equal to   the covering degree of    $\NODE_{i_1,\cdots, i_n, i_{n+1}}$, i.e.,
	$$[{ \mathcal{F} }^{(m,j)}_{i_1,\cdots,i_{n},i_{n+1}}:  { \mathcal{F} }^{(m,j)}_{i_1,\cdots,i_{n-1},i_{n}}]=[{ \dot{\mathcal{F}} }^{(m,j)}_{i_1,\cdots,i_{n},i_{n+1}}:  { \dot{\mathcal{F}} }^{(m,j)}_{i_1,\cdots,i_{n-1},i_{n}}]=\coveringdeg_{\NODE_{i_1,\cdots, i_n,i_{n+1}} }.$$

\end{thm}
\begin{proof}
	
	By Definition \ref{Def:normalizedDrinfeldmodularcurve}, the extension degree of ${\dot{\mathcal{F}}}^{(m,j)}_{i_1,\cdots, i_n}$ over ${\dot{\mathcal{F}}}^{(m,j)}_{0}$ equals the number of $T$-torsion flags in the class $\NODE_{i_1,\cdots, i_n}$. The rest is clear.
\end{proof}

\subsection{Some particular modular curves}
\begin{example}\label{Ex:1} Consider top nodes of the $T$-torsion tree $\Ttorsiontree^m$ (cf. Figure \ref{fig:Ttree}), and let us denote them by $\NODE_{1}$, $\NODE_{1,1}$, $\NODE_{1,1,1}$, $\NODE_{1,1,1,1}$, and so on. For notational simplicity, denote the topmost node of level $n$  as $\NODE_{1\times n}$. The corresponding matrix is of the form
	\[R_{1\times n} =
	\begin{array}{c@{\hspace{-5pt}}l}
	\left(\begin{array}{l|l}
	0 & I_{n-1}  \\
	\hline 0   & 0  		
	\end{array}\right)
	\begin{array}{l}
	\left. \rule{0mm}{2  mm} \right\} n-1 \\
	\left. \rule{0mm}{2 mm} \right\} 1.
	\end{array}
	\end{array}
	\]
	
	The associated (normalized and reduced) Drinfeld modular curves 	$\dot{X}^{(m,j)}_{1\times n  }$ and ${X}^{(m,j)}_{1\times n  }$
	are  studied in our previous work \cite{MR4203564}.
	These curves are exactly the recursive and good (BBGS) tower invented by  	
	Bassa, Beelen, Garcia, and Stichtenoth \cites{Bassa2015,A.Bassa2014}.
	
\end{example}

\begin{example}\label{Ex:2} Consider the nodes $\NODE_{1,2}$, $\NODE_{1,2,1}$, $\NODE_{1,2,1,1}$, $\NODE_{1,2,1,1,1}$, and so on. Let us denote the level $(n+2)$ one in this sequence by $\NODE_{1,2,1\times  n  }	$ which corresponds to the matrix
	
	\[R_{1,2,1\times n}=
	\begin{array}{c@{\hspace{-5pt}}l}
	\left(\begin{array}{ll|l}
	0 & 0 & I_n \\
	\hline 0 & 0 & 0 \\
	0 & 0 & 0
	\end{array}\right)
	\begin{array}{l}
	\left. \rule{0mm}{2.9 mm} \right\} n \\
	\left. \rule{0mm}{3.4 mm} \right\} 2.
	\end{array}
	\end{array}
	\]

	The associated Drinfeld modular curves 	$\dot{X}^{(3,2)}_{1,2,1\times  n   }$ and ${X}^{(3,2)}_{1,2,1\times n  }$
	are  studied   by  Anbar,  Bassa, and  Beelen \cite{Nurdagul2017}. Indeed,  they  constructed a certain tower of Drinfeld modular curves.
	This tower parameterizes
	Drinfeld modules associated with isogenies $\lambda_i$ ($i=1,2,\cdots$) subject to the property that the $T$-action on
	a  rank $2$ module   $\Ker(\lambda_i\circ \lambda_{i+1})  $ is trivial  (but not necessarily trivial on the kernel of $\lambda_{i+2} \circ \lambda_{i+1} \circ \lambda_i $). The present paper is   very much inspired by their idea.

\end{example}


\begin{example}Similar to the previous example,     consider the sequence of nodes $\NODE_{1}$, $\NODE_{1,2}$, $\cdots$, $\NODE_{1,2,\cdots,\kappa}$, $\NODE_{1,2,\cdots,\kappa,1}$,
	$\NODE_{1,2,\cdots,\kappa,1,1}$, $\cdots$,
	$ \NODE_{1,2,3,4,5,..., \kappa,1\times n}$,   and so on. According to Theorem  \ref{prop:coveringdegree1}, we can find  degrees  of field extensions one by one along the corresponding sequence of Drinfeld modular curves, which are exactly the covering degrees:   $\coveringdeg_{\NODE_{1,2}}$    $=\Nqfrac{ m-1}$,   $\coveringdeg_{\NODE_{1,2,3}}$   $=\Nqfrac{ m-2}$, $\cdots$, and       $\coveringdeg_{\NODE_{1,2,3,...,\kappa}}$   $=\Nqfrac{ m-\kappa+1}$. Moreover, we have  $\coveringdeg_{\NODE_{1,2,3,4,5,..., \kappa,1\times n}}$    $=q^{m-\kappa}$ for all $n\geqslant 1$.
	
\end{example}

 \section{ $(3,2)$-type normalized Drinfeld modular curves}\label{Sec:32flagclassfunctionfield}
 In this section, we   study  $(3,2)$-type normalized Drinfeld modular curves by investigating their function fields.  Let  $\Knaught :=\Fq(G)$ be the special rational function field in the variable $G$.   Also, let the morphism $\iota: A(=\Fq[T])\to \Knaught $ be determined by $T\mapsto 1$. 
 Moreover, we fix a special $(3,2)$-type normalized Drinfeld module $\Basic$ (over $\Knaught $) defined by
 $$\BasicT =  -\tau^3 + G \tau ^2 + 1 .$$  
 
 We   define a series of subfields $$\Knaught \subset {K_1}\subset \cdots \subset {K_n}\subset {K_{n+1}}\subset \cdots \subset \Knaughtbar $$ for all $n \geqslant 1$ by setting  ${K_n}$ to be  \textit{the subfield generated by $\Ker(\Basic_{T^n})$ over $\Knaught  $}. It is easy to see that each ${K_n}/\Knaught $ is a Galois extension.

\subsection{Structure of a $\BasicT$-torsion flag}\label{Section6.1}
 

 Next, we consider $\BasicT$-torsion  $n$-flags  in $\Knaughtbar $. We have established in Section \ref{Sec:basic} the one-to-one correspondence between $\BasicT$-torsion  $n$-flags  in $\Knaughtbar $ and $T$-torsion  $n$-flags  in $S^3$. If we fix a $T$-torsion $n$-flag class $\NODE_{i_1,\cdots, i_n}$ (in $S^3$), namely a level $n$ node of the $T$-torsion tree $\Ttorsiontree^3$, we can find   ${\mathcal{V}}=\left(V_1 \subset V_2 \subset \cdots \subset V_n\right) $, a  $\BasicT$-Torsion $n$-flag in $\Knaughtbar $  {subordinate to} $\NODE_{i_1,\cdots, i_n}$. 
Via the  standard correspondence (Section \ref{Sec:flagsnotations})   
we have the associated data  $$(u_1,\ldots,u_n) = \Lambda({\mathcal{V}})\in (\Knaughtbar ^*)^{\times n}.$$ By the defining equations of $u_i$ given after Equation \eqref{Eqt:Lambdafirstappear}, we have $u_i\in K_i$. 

\begin{defn}
	   Consider the $A$-field $\Knaughtbar$ and let $u\in \Knaughtbar^*$ be given. The following polynomial       parameterized by $u$ is called the \textbf{modular polynomial}:
\begin{eqnarray}\label{Eqt:kappauT}
\kappau{u}{X}&:=&X^{q^2+q+1}+ \frac{1-u^{q^2+q+1}}{u^{q^2+q}}    X^{q+1}-1 \ \in  \Knaughtbar[X].
\end{eqnarray}

\end{defn} We elect to call $\kappausingle{u}$ the  {modular polynomial} for its significant role played in Drinfeld modular curves subordinate to every node of $\Ttorsiontree^3$, to be elaborated by  Theorem \ref{Thm:NodeToFactor}.

  Before studying the specific curve problem, we note that  ${\mathcal{V}}   $ being a $\BasicT$-torsion  flag  in $\Knaughtbar$     already implies certain constraints on the corresponding data $(u_1,\ldots,u_n) = \Lambda({\mathcal{V}})\in (\Knaughtbar ^*)^{\times n}$, as stated in the following theorem.

\begin{thm}\label{Thm:nflagcondition0}Given $n\geqslant 2$ and $(u_1,\ldots,u_n)  \in (\Knaughtbar ^*)^{\times n}$, let 	  $\mathcal{V}=\Theta(u_1,\cdots, u_n)   $ be the associated $n$-flag in $\Knaughtbar$. Then $\mathcal{V}$ is a $\BasicT$-torsion  flag  if and only if
	\begin{equation}\label{Eqt:kappauuiuiplus101}
	G=\frac{ u_1^{q^2 + q +1 } -1 }{  u_1^{q+1}}
	\end{equation}
	and
	\begin{equation}
	\label{Eqt:kappauuiuiplus100}
	\kappau{u_1}{u_{2}}=\kappau{u_2}{u_{3}}=\cdots=\kappau{u_{n-1}}{u_{n}}=0.\end{equation}
\end{thm}

We need some preparations to facilitate the proof. First, let us consider the following $(3,2)$-type Drinfeld modules $\phinoTu{u}$ and $\psinoTu{u}$ (over $\Knaughtbar $) parameterized by $u\in \Knaughtbar^*$:
\begin{equation}\label{Eqt:phiTui0}
\phiTu{u}: =  -\tau ^3 + \frac{ u^{q^2 + q +1 } -1 }{  u^{q+1}} \tau ^2 +1  ;
\end{equation}   \begin{equation}
\label{Eqt:psiTu0}
\corepoly{u } :=-\tau^3+ \frac{ u ^{q^2+q+1}-1   }{u ^{q^2+q}} \tau^2+1 .\end{equation}

 We then have an obvious identity
\begin{equation}
\label{Eqt:kappau=corepoly}
\kappau{u}{l^{q-1}}= -l\corepoly{u }(l ),\quad \forall l\in \Knaughtbar.
\end{equation}

Recall  that we have introduced
$$\lambdau{u} := \tau - u\ \in  \Knaughtbar\set{\tau}  $$(cf. Equation \eqref{Eqt:lambdaui}).
 The following lemma is needed, and can be verified directly.

\begin{lem}[\cite{Bassa2015}]\label{kappa}
	The element  $\lambdau{u }=\tau-u $ is an isogeny from $\phinoTu{u}$ to $\psinoTu{u} $, i.e.
	$$\lambdau{u }\phiTu{u }=\corepoly{u }\lambdau{u }   .$$
\end{lem}


The next proposition can be directly verified as well.

\begin{prop}\label{Prop:lambda2phiTu2=phiTu3lambda2} Let $u$ and $v$ be in $\Knaughtbar^*$. The following statements are equivalent:
	\begin{enumerate}
		\item The equation $\kappau{u}{v}=0$ holds   in $\Knaughtbar$;
		\item The Drinfeld modules $\psinoTu{u}$ and $\phinoTu{v}$ are one and the same;
		\item The element  $\lambdau{u }=\tau-u $ is an isogeny from $\phinoTu{u}$ to $\phinoTu{v}$, i.e.,
		\begin{equation}
		\label{Eqt:lambdaiphiTui}
		\lambdau{u } \phiTu{u}=\phiTu{v}\lambdau{u }\,.
		\end{equation}
		\item The element  $\lambdau{v }=\tau-v $ is an isogeny from $\psinoTu{u}$ to $\psinoTu{v}$, i.e.,\begin{equation}
		\label{Eqt:lambdauvcorepolyu}
		\lambdau{v}\corepoly{u} =  \corepoly{v}\lambdau{v}  .\end{equation}
		
	\end{enumerate}

\end{prop}

We are in a position to prove the previous main theorem.
\begin{proof}[Proof of Theorem \ref{Thm:nflagcondition0}] Suppose that, under the standard correspondence (Definition \ref{corresponding}), we have $$\mathcal{V}=\left(V_1 \subset V_2 \subset \cdots \subset V_n\right)=\left(\nu_1<\ldots<\nu_n\right)  \xleftrightharpoons[\Lambda]{\Theta} (u_1,\ldots,u_n).  $$
	
	Since $u_1= \nu_1^{q-1}$, it is easy to see that $\BasicT(V_1)=0$, i.e. $\BasicT(\nu_1)=0$, if and only if Equation \eqref{Eqt:kappauuiuiplus101} holds.
	 Note that under this condition we have    $\BasicT=\phiTu{u_1}$. 
So, it remains to   show that under  Condition \eqref{Eqt:kappauuiuiplus101},  $ \phiTu{u_1}(V_i)\subset V_{i-1}$ holds for all $i=2,3,\cdots, n$ if and only if \eqref{Eqt:kappauuiuiplus100} holds.
	
	\textit{Sufficiency}: Suppose that \eqref{Eqt:kappauuiuiplus100} holds. By  Lemma \ref{kappa}  and Equation \eqref{Eqt:lambdauvcorepolyu}, we have
	$$\lambdau{u_{i-1}}\cdots \lambdau{u_1}(\phiTu{u_1}(\nu_i))=\corepoly{u_{i-1}}{} \lambdau{u_{i-1}}\cdots \lambdau{u_1}(\nu_i)=\corepoly{u_{i-1}}{}(l_i)\equalbyreason{\mbox{\eqref{Eqt:kappau=corepoly}}}-\frac{1}{l_{i}}\kappau{u_{i-1}}{u_i}=0.$$
	Therefore, we see that $\phiTu{u_1}(\nu_i) \in V_{i-1}=\Ker (   \lambdau{u_{i-1}} \cdots\lambdau{u_1})$ and hence $\mathcal{V}$ is a $T$-torsion $n$-flag in  $S^{\phi}$.
	
	\textit{Necessity}: 	Assume  that  we have $\phiTu{u_1}(V_i)\subset V_{i-1}$ for all $i=2,3,\cdots, n$. Since $\phiTu{u_1}(\nu_2)\in V_1$, we get
	$$\kappau{u_1}{u_2} \equalbyreason{\mbox{\eqref{Eqt:kappau=corepoly}}} -l_2 \corepoly{u_1}(l_2) =-l_2\corepoly{u_1}\lambdau{u_1}(\nu_2) \equalbyreason{\mbox{Lemma \ref{kappa}}}-l_2\lambdau{u_1}\phiTu{u_1}(\nu_2)=0.$$  Suppose that $\kappau{u_{j-1}}{u_{j}}=0$ is true  for all $2\leqslant j \leqslant i$. Then since $\phiTu{u_1}(\nu_{i+1}) \in V_i$, we have
	\begin{eqnarray*}&&\kappau{u_{i}}{u_{i+1}}\equalbyreason{\mbox{\eqref{Eqt:kappau=corepoly}}} -l_{i+1}\corepoly{u_i}(l_{i+1}) \\
		&&\quad\quad=-l_{i+1} \corepoly{u_i}\lambdau{u_i}\ldots\lambdau{u_1}(\nu_{i+1})\\&&\quad\quad
		=
		-l_{i+1} \lambdau{u_i}\ldots\lambdau{u_1}\phiTu{u_1}(\nu_{i+1})\quad
		\mbox{(by  \eqref{Eqt:lambdauvcorepolyu} and  Lemma \ref{kappa})}\\
		&&\quad\quad=0.
	\end{eqnarray*}
	Therefore,   we have proved $\kappau{u_i}{u_{i+1}}=0$. The induction goes forward until $i=n-1$.
\end{proof}

 \subsection{The function field and minimal polynomial of a $\BasicT$-torsion flag} \label{subSec:L0characteristics}
 Let 
 ${\mathcal{V}}  $ be a $\BasicT$-torsion $n$-flag in $\Knaughtbar $   subordinate to the $T$-torsion flag class   $\NODE_{i_1,\cdots, i_n}$ and  
    $(u_1,\ldots,u_n) = \Lambda({\mathcal{V}})$ be as before. We construct a function field associated to ${\mathcal{V}}$  by
 \begin{equation}\label{Eqt:mathcalF} 	 
 \mathcal{F}_{{\mathcal{V}}} := \Knaught  (u_1,\ldots,u_n)  \quad (\subset {K_n}),	 
 \end{equation} 
 which is the core object of study in this section. According to Equation \eqref{Eqt:kappauuiuiplus101}, we can also write 
 \begin{equation}\label{Eqt:mathcalF2} 	 
 \mathcal{F}_{{\mathcal{V}}}  =\Fq(u_1,\ldots,u_n).	 
 \end{equation} 
  We   call       $\mathcal{F}_\mathcal{V}$ the ($(3,2)$-type) \textbf{function field}   of   $\mathcal{V}$     or the associated flag class  $\NODE_{i_1,\cdots, i_n}$. The reason will be explained by the subsequent Proposition \ref{prop:isomorphismFVbetaV}.
 
 In the meantime, we define a map   \begin{equation}
 \label{Eqt:mathcalP}
 \mathcal{P} :\set{\BasicT\mbox{-torsion } n \mbox{ flags in }\Knaughtbar}\rightarrow {K_n}\,,\quad {\mathcal{V}}  \mapsto    u_n\end{equation}   which is clearly well-defined.

 Here and in the sequel, we    denote  the parent $\BasicT$-torsion $(n-1)$-flag of ${\mathcal{V}}$ by ${\mathcal{V}_{n-1}}$ which is made of sub spaces $  (V_1\subset V_2\subset \cdots \subset V_{n-1})$.    Define a set \begin{equation}
 \label{Eqt:MVVdiamond}
 M({\mathcal{V}},{\mathcal{V}_{n-1}}):= \{ \mbox{child }  \BasicT \mbox{-torsion }  n \mbox{-flags of }  {\mathcal{V}_{n-1}} \mbox{ which are isomorphic to } {\mathcal{V}}.\}\end{equation}

We now construct a polynomial $\minimalpoly{\mathcal{V}}\in {K_n}[X]$ using elements in $M({\mathcal{V}},{\mathcal{V}_{n-1}})$:
 \begin{equation}\label{beta}
 \minimalpoly{\mathcal{V}}(X)  := \prod_{\mathcal{V}' \in M({\mathcal{V}},{\mathcal{V}_{n-1}})}(X-\mathcal{P}(\mathcal{V}')) .
 \end{equation}
 We  call       $\minimalpoly{\mathcal{V}}$ the ($(3,2)$-type) \textbf{minimal polynomial}  of   $\mathcal{V}$     or the associated flag class  $\NODE_{i_1,\cdots, i_n}$. The reason will be explained by   Propositions \ref{prop:isomorphismFVbetaV} and    \ref{Prop:main}.  An evident fact is that,  if $\mathcal{V}$ and $\mathcal{W}$ are isomorphic $\BasicT$-torsion flags and share the same parent flag, then their corresponding minimal polynomials are identical: $\minimalpoly{\mathcal{V}}=\minimalpoly{\mathcal{W}}$.
 
 \begin{example}
 	\label{Example:n=1}
 	If $n=1$, then our convention is $\mathcal{V}_0=\set{0}$ and  	
 	$\mathcal{F}_{{\mathcal{V}_{0}}}=\Knaught=\Fq(G)$. It is also clear that we have $$\mathcal{F}_{{\mathcal{V}}}=K_0(u_1)=\Fq(u_1).$$ Moreover, $M(\mathcal{V} ,\mathcal{V}_{0})$ consists of all  $\BasicT$-torsion $1$-flags.  In this case the minimal polynomial of $u_1$ is \begin{equation}\label{Eqt:betaV1}\minimalpoly{\mathcal{V}}(X)= X^{q^2+q+1}-GX^{q+1}-1.\end{equation}
 	In fact, for any $\mathcal{V}' = (\nu'_1) \in M(\mathcal{V},\mathcal{V}_{0})$, we have $\BasicT(\nu'_1) = 0$, i.e., $$(-\tau^{3}+G\tau^2+1)(\nu'_1) = -(\nu'_1)^{q^3}+G(\nu'_1)^{q^2}+\nu'_1 = 0.$$   
 	Since $\mathcal{P}(\mathcal{V}') = (\nu'_1)^{q-1}$ and $\nu'_1 \neq 0$, it is easy to get $$-(\mathcal{P}(\mathcal{V}'))^{q^2+q+1}+G (\mathcal{P}(\mathcal{V}'))^{q+1}+1 = 0 ,$$ i.e., $\mathcal{P}(\mathcal{V}')$ is a root of the polynomial $X^{q^2+q+1}-G X^{q+1}-1$. Moreover, by counting the number of elements in $M(\mathcal{V},\mathcal{V}_{0})$, which is $(q^2+q+1)$, we see that $\minimalpoly{\mathcal{V}}$ must be of the form \eqref{Eqt:betaV1}. 
 \end{example}
 
 Our definition of $\mathcal{F}_{\mathcal{V}}$ and $\minimalpoly{\mathcal{V}}$ depends on the choice of ${\mathcal{V}}\in \set{\BasicT\mbox{-torsion } n \mbox{ flags in }\Knaughtbar}$. Of course, one may choose a different  $\BasicT$-torsion flag which is also subordinate to $\NODE_{i_1,\cdots, i_n}$. The next proposition explains how the resulting data are related.
 \begin{prop} 
 	\label{prop:isomorphismFVbetaV}   	Let  ${\mathcal{V}}$ and  ${\mathcal{W}}$ be two $\BasicT$-torsion $n$-flag classes in $ \Knaughtbar  $.  Let $\mathcal{F}_{\mathcal{V}}:=\Fq  (u_1,\ldots,u_n)$ and $\mathcal{F}_{\mathcal{W}}:=\Fq  (w_1,\ldots,w_n)$ be the corresponding function fields where  $(u_1,\ldots,u_n) = \Lambda({\mathcal{V}})$ and $(w_{1},\ldots,w_{n}) = \Lambda(\mathcal{W} )$ ($\in (\Knaughtbar ^*)^{\times n}$). Let $\minimalpoly{\mathcal{V}}$ and $\minimalpoly{\mathcal{W}}$ ($\in {K_n}[X]$) be the associated minimal polynomials of $\mathcal{V}$ and $\mathcal{W}$, respectively. 
  	If $\mathcal{V}$ and $\mathcal{W}$ 
 	are both subordinate  to $\NODE_{i_1,\cdots, i_n}$, then  there exists an isomorphism $\psi: \mathcal{F}_{\mathcal{V}}\to \mathcal{F}_{\mathcal{W}}$ of function fields  such that \begin{itemize}
 		\item[(1)] $\psi(u_{i}) = w_{i}$ for all $i = 1,\ldots,n$;
 		\item[(2)] $\psi(\minimalpoly{\mathcal{V}})=\minimalpoly{\mathcal{W}}$.
 	\end{itemize}
 \end{prop}
 The  above fact confirms that the definition of $\mathcal{F}_\mathcal{V}$ in Equation \eqref{Eqt:mathcalF} (or \eqref{Eqt:mathcalF2}) is, up to isomorphisms, solely determined by the isomorphism class of $\mathcal{V}$, or the corresponding $T$-torsion $n$-flag class $\NODE_{i_1,\cdots, i_n}$ to which $\mathcal{V}$ is subordinate (so is the minimal polynomial $\minimalpoly{\mathcal{V}}$). 

 Moreover, we have another  proposition which claims that the coefficients of $\minimalpoly{\mathcal{V}}$ are  indeed in $\mathcal{F}_{{\mathcal{V}_{n-1}}}= \Fq  (u_1,\ldots,u_{n-1})$ ($\subset {K_{n-1}}$).
 \begin{prop}\label{Prop:main} With notations as above, $\minimalpoly{\mathcal{V}}$ is the minimal polynomial of $ u_n = \mathcal{P} (\mathcal{V})$ over $\mathcal{F}_{\mathcal{V}_{n-1}}$. Consequently,  
 		  $\minimalpoly{\mathcal{V}}$ has coefficient in $\mathcal{F}_{{\mathcal{V}_{n-1}}} $ and 
		    is irreducible over $\mathcal{F}_{{\mathcal{V}_{n-1}}}$.

 \end{prop}

 \subsection{Proof of the  two propositions} We need some preparatory works  before we give  the  proofs of Propositions \ref{prop:isomorphismFVbetaV} and  \ref{Prop:main}.
 
 \subsubsection{A preparatory theorem} First, since the Galois action of $\Gal({K_n}/\Knaught  )$ on $\Knaught  $ and $\BasicT$ commutes, one can prove that  
 for any $f \in \Gal({K_n}/\Knaught  )$, the restriction map $\rho_n (f):=f|_{\Ker(\Basic_{T^n})}$ is an automorphism of the $A$-module $\Ker(\Basic_{T^n})$.
 In this way we obtain a homomorphism   $\rho_n:~ \Gal({K_n}/\Knaught  )\to \Aut_{A}(\Ker(\Basic_{T^n}))$, and it   is obviously  injective. Indeed, we have the following quite nontrivial   fact.
 
 
 \begin{thm}\label{Lem:generalmjconjecture} All the maps $\rho_n:~ \Gal({K_n}/\Knaught  )\to \Aut_{A}(\Ker(\Basic_{T^n})) $ for $n\geqslant 1$ are isomorphisms of groups.
 \end{thm}
 To prove  it,   we do induction on $n$ and start with  the following lemma. 
 
 \begin{lem}\label{Lem:rho1iso} The map $\rho_1:~ \Gal({K_1}/\Knaught  )\to \Aut_{A}(\Ker(\BasicT)) $ is an isomorphism.	 
 \end{lem}
 \begin{proof} 
 	According to Lemma \ref{Lem:moduleiso}, we have an isomorphism of $A$-modules \begin{equation}\label{Eqt:mathcalE}\mathcal{E} :~  \Ker(\Phi_{T }) \to (A/T A)^{\oplus 3}\cong \Fq^{\oplus 3} .\end{equation}  So we find a basis  $(\alpha_1,\alpha_2,\alpha_3)$ of $\Ker(\Phi_{T })$.
 	Then from $\BasicT(\alpha_1) = 0$ we get \begin{equation}
 	\label{Eqt:ginalpha1}
 	G = \frac{-\alpha_1^{q^3-1}-1}{\alpha_1^{q^2-1}} \end{equation}
 	and thus $\Knaught  (\alpha_1) = \Fq(G)(\alpha_1) = \F_q(\alpha_1)$. By  \cite{TTAFF}*{Corollary 4.1.2}, we have $[\Knaught  (\alpha_1):\Knaught  ] = [\F_q(\alpha_1):\Fq(G)] = \max\{q^3-1,q^2-1\}=q^3-1$. 
 	
 	Next, we consider the field extension $\Knaught  (\alpha_1,\alpha_2)$ over $\Knaught  (\alpha_1)$. Take $z_1 = \alpha_1^{q-1}$ and $z_2 = ((\tau-z_1)(\alpha_2))^{q-1}$ ($\in \Knaught  (\alpha_1,\alpha_2) $). We can rewrite $\BasicT$ using \eqref{Eqt:ginalpha1}:
 	$$\BasicT = \left(\tau^2+\frac{1}{z_1 ^{q+1}}\tau+\frac{1}{z_1 }\right)(\tau-z_1 ) \in \Knaught  (\alpha_1)\{\tau\}.$$

 	Then from 
 	$\BasicT(\alpha_2) = 0$ we derive 
 	\begin{equation}\label{Eqt:z1z2}  z_2  ^{q+1}+\frac{z_2  }{z_1 ^{q+1}}+\frac{1}{z_1 } = 0.\end{equation}
 	By Eisenstein's criterion, $f(X) := X^{q+1}+\frac{1}{z_{1}^{q+1}}X+\frac{1}{z_1 }$ (in $\Knaught  (\alpha_1)[X]$) is irreducible and the minimal polynomial of $z_2  $, and $[\Knaught  (\alpha_1,z_2 ):\Knaught  (\alpha_1)] = q+1$.
 	
 	Now, consider $t_2 = (\tau-z_1 )(\alpha_2)$ ($\in \Knaught  (\alpha_1,z_2)$). It is   routine  to  check that $\Knaught  (\alpha_1,z_2 ,t_2)$ over $\Knaught  (\alpha_1,z_2 )$ is a Kummer extension with degree $(q-1)$ and $\Knaught  (\alpha_1,\alpha_2)$ over $\Knaught  (\alpha_1,z_2 ,t_2)$ is an Artin-Schreier extension with degree $q$. Hence, we get $[\Knaught  (\alpha_1,\alpha_2):\Knaught  (\alpha_1)]=(q+1)(q-1)q = q^3-q$.

 	By Equation \eqref{Eqt:z1z2}, we can rewrite
 	$$\BasicT = (\tau - \frac{1}{z_1 z_2  })(\tau-z_2  )(\tau-z_1 ) \in \Knaught  (\alpha_1,\alpha_2)\{\tau\}.$$ So it follows from $\BasicT(\alpha_3) = 0$ that 
 	$$  z_3  -\frac{1}{z_1 z_2  } = 0 $$
 	where $z_3   = ((\tau-z_1 )(\tau-z_2  )(\alpha_3))^{q-1}$;  
 	and hence $\Knaught  (\alpha_1,\alpha_2,z_3  ) = \Knaught  (\alpha_1,\alpha_2)$.

 	Then for $t_3 = (\tau-z_2  )(\tau-z_1 )(\alpha_3)$ ($\in \Knaught  (\alpha_1,\alpha_2)$),   $\Knaught  (\alpha_1,\alpha_2,t_3)$ over $\Knaught  (\alpha_1,\alpha_2)$ is a Kummer extension with degree $q-1$ and $\Knaught  (\alpha_1,\alpha_2,\alpha_3)$ over $\Knaught  (\alpha_1,\alpha_2,t_3)$ is an Artin-Schreier extension with degree $q^2$. Hence, we get $[\Knaught  (\alpha_1,\alpha_2,\alpha_3):\Knaught  (\alpha_1,\alpha_2)]=(q-1)q^2 = q^3-q^2$. 
 	
 	In summary, for ${K_1} = \Knaught  (\alpha_1,\alpha_2,\alpha_3)$, 
 	we have $[{K_1}:\Knaught  ] = (q^3-1)(q^3-q)(q^3-q^2)$. Since ${K_1}/\Knaught  $ is Galois, we have $\#(\Gal({K_1}/\Knaught  )) = [{K_1}:\Knaught  ] = (q^3-1)(q^3-q)(q^3-q^2)$, i.e., number of elements in $ \GL_3(\Fq)\cong \Aut_{A}(A/T A)^{\oplus 3}\cong \Aut_{A}(\Ker(\BasicT))$. This proves that  that $\rho_1$ is an isomorphism.
 	
 \end{proof}
 We then prove the previous important theorem.
 \begin{proof}
 	[Proof of Theorem \ref{Lem:generalmjconjecture}]  
 	Suppose that the statement is proved for $n=i-1$ and we show it holds as well for $n=i$.
 	
 	Again according to Lemma \ref{Lem:moduleiso}, we have $\Ker(\Basic_{T^i }) \cong (A/T^i A)^{\oplus 3}$ as $A$-modules, and we have the  natural isomorphism $\Aut_{A}(A/T^i A)^{\oplus 3} $  $\cong  \GL_3(A/T^iA)$. So we directly regard $\rho_i$ as a morphism    $\Gal({K_{i}}/\Knaught  ) \to$ $   \GL_3(A/T^iA)$.
 	
 	Consider the following commutative diagram between exact sequences with all arrows naturally defined: 
 	\begin{equation*}
 	\xymatrix{
 		1\ar[r]&\Gal({K_{i}}/{K_{i-1}})\ar[d]^{\Psi}\ar@{^{(}->}[r] &\Gal({K_{i}}/\Knaught )\ar[d]^{\rho_{i}}\ar[r]^{Q}& \Gal({K_{i-1}}/\Knaught )\ar[d]^{\rho_{i-1}}\ar[r]&1 \\
 		1\ar[r]& K \ar@{^{(}->}[r]  & \GL_3(A/T^iA) \ar[r]^{Q'}&\GL_3(A/T^{i-1}A) \ar[r]&1	.
 	}
 	\end{equation*}
 	Here  the maps $Q$, $Q'$ and the group $K$ are defined as follows: \begin{itemize}
 		\item For all $f \in \Gal({K_{i}}/\Knaught  )$  we define $Q(f) := f|_{{K_{i-1}}}$;
 		\item For all $M\in \GL_3(A/T^{i}A)$, we define $Q'(M) := M \mod T^{i-1}$;
 		\item $K$ consists of matrices of the form $I_3+T^{i-1}X$ where $X\in \Fq^{3\times 3}$ is  arbitrary. Clearly, $K$ is an abelian subgroup in $\GL_3(A/T^{i}A)$.
 	\end{itemize}

 	We can find a basis  $(\alpha_1,\alpha_2,\alpha_3)$ of $\Ker(\Basic_{T^i })$, and hence     ${K_{i}} = \Knaught  (\alpha_1,\alpha_2,\alpha_3)$. By setting $\beta_i = \BasicT(\alpha_i)$ ($\in \Ker(\Basic_{T^{i-1} })$), we also have  	${K_{i-1}} = \Knaught  (\beta_1,\beta_2,\beta_3)$.   
 	Clearly, one has
 	$$[{K_{i}}:{K_{i-1}}] = [{K_{i-1}}(\alpha_1):{K_{i-1}}][{K_{i-1}}(\alpha_1,\alpha_2 ):{K_{i-1}}(\alpha_1 )] [{K_{i-1}}(\alpha_1,\alpha_2,\alpha_3):{K_{i-1}}(\alpha_1, \alpha_{2})].$$

 	The three field extensions $ {K_{i-1}}(\alpha_1)$ $/{K_{i-1}}$, $ {K_{i-1}}(\alpha_1,\alpha_2 )$ $/{K_{i-1}}(\alpha_1 )$, and ${K_{i-1}}(\alpha_1,\alpha_2,\alpha_3)$ $/{K_{i-1}}(\alpha_1, \alpha_{2}) $  are  Artin-Schreier extensions with minimal polynomials $\BasicT(X)-\beta_1$, $\BasicT(X)-\beta_2$, and $\BasicT(X)-\beta_3$, respectively. Therefore, the degree of extension of ${K_{i}}$ over ${K_{i-1}}$ is $q^{9}$. Moreover, since ${K_{i}}$ over $\Knaught  $ is Galois, we have $\#(\Gal({K_{i}}/{K_{i-1}})) = [{K_{i}}:{K_{i-1}}] = q^{9}$ which coincides with $ \#K$.       It is   easy to see that $\Psi$ is injective, and hence it must be an isomorphism. Then by the induction assumption ($\rho_{i-1}$ is an isomorphism) and the standard Five-Lemma, we know that $\rho_i  $ is   an isomorphism, as desired. 
 	
 \end{proof} 
 
 
 \subsubsection{Isomorphisms of torsion flags} Second, we state another important lemma.
 
 \begin{lem}\label{extend0}
 	Suppose that  ${\mathcal{V}}= (V_1,\ldots,V_n)$ and $\mathcal{W} = (W_1,\ldots,W_n)$   are two $\BasicT${-torsion } $n$-flags in $\Knaughtbar $   and their parameters   $(u_1,\ldots,u_n) = \Lambda({\mathcal{V}})$ and $(w_{1},\ldots,w_{n}) = \Lambda(\mathcal{W} )$ ($\in (\Knaughtbar ^*)^{\times n}$) are  given. The following statements are equivalent.
 	\begin{itemize}
 		\item[(1)] The $\BasicT$-torsion $n$-flags $\mathcal{V}$ and $\mathcal{W}$ are isomorphic;
 		\item[(2)] There exists an automorphism $\sigma$ of the $A$-module $\Ker(\Basic_{T^n})$ such that $\sigma(V_i)=W_i$ for all $i = 1,\ldots,n$;
 		\item[(3)]  There exists $f \in \operatorname{Gal}({K_{n}}/\Knaught  )$ such that $f(u_{i}) = w_{i}$ for all $i = 1,\ldots,n$.
 	\end{itemize}
 \end{lem}
 \begin{proof}
 	$(1) \Rightarrow (2)$: Again, we use the isomorphism \eqref{Eqt:mathcalE} of $A$-modules $\mathcal{E}$: $\Ker(\Basic_{T^n})$   $\to$  $(A/T^nA)^{\oplus 3}$. Via $\mathcal{E}$, $\mathcal{V}$ and $\mathcal{W}$ are sent to $T$-torsion $n$-flags in $(A/T^nA)^{\oplus 3}$, denoted by $\mathcal{V}'=(V'_1,\ldots,V'_n)$ and $\mathcal{W}'=(W'_1,\ldots,W'_n)$ respectively.  The isomorphism $\tau: V_n\to W_n$ is transferred to  an isomorphism   $\tau': V'_n$ $\to$ $W'_n$ (of $A$-modules). It suffices to prove that there exists an automorphism $\sigma'$ of the $A$-module $(A/T^nA)^{\oplus 3}$ such that $\sigma'|_{V_n'}=\tau'$.

 	We suppose that  $\mathcal{V}' = (\nu_1<\ldots<\nu_n)$ and $\mu_i = \tau'(\nu_i)$, and thus  $\mathcal{W}' = (\mu_1<\ldots<\mu_n)$. 	
 	Let us use a matrix $N\in \mathbf{U}^o_n$ to represent the $T$-action on $\mathcal{V}'$:
 	$$T(\nu_1,\ldots,\nu_n) = (\nu_1,\ldots,\nu_n)N. $$
 	Clearly, the $T$-action on $\mathcal{W}'$ is represented by $N$ as well
 	$$T(\mu_1,\ldots,\mu_n) =   (\mu_1,\ldots,\mu_n)N.$$

 	First, since $N$ is nilpotent, we can find a cyclic basis $(c^{(1)}_1,\ldots, c^{(1)}_{k_1}, \ldots, c^{(l)}_1,\ldots,c^{(l)}_{k_l})$ of $V'_n$ such that   $Tc_{1}^{(i)} = 0$  and $Tc_{j}^{(i)}= c_{j-1}^{(i)}$ ($1\leqslant l\leqslant 3$, $  k_1+\cdots+k_l=n$). Write $J\in \mathbf{U}^o_n$ (the Jordan form of $N$) for this relation:
 	$T(c^{(1)}_1,\ldots, c^{(1)}_{k_1}\ldots, c^{(l)}_1,\ldots,c^{(l)}_{k_l}) = (c^{(1)}_1,\ldots, c^{(1)}_{k_1},\ldots, c^{(l)}_1,\ldots,c^{(l)}_{k_l})J$; and suppose that   $J = Z^{-1}NZ$ for $Z\in \mathrm{GL}(n;\Fq )$. In other words,we have $$ (c^{(1)}_1,\ldots, c^{(1)}_{k_1}, \ldots, c^{(l)}_1,\ldots,c^{(l)}_{k_l}) = (\nu_1,\ldots,\nu_n)Z.$$
 	
 	The second step is to  extend this cyclic basis to an $\F_q$-basis of $(A/T^nA)^{\oplus 3}$:
 	\begin{itemize}
 		\item From $Tc_1^{(i)} = 0$, we have $c_1^{(i)} \in (T^{n-1}A/T^nA)^{\oplus 3}$.  Extend $(c_1^{(1)},\ldots,c_1^{(l)})$ to an $\F_q$-basis   $(c_1^{(1)},\ldots,c_1^{(l)},c_1^{(l+1)},\ldots,c_1^{(3)})$ of $(T^{n-1}A/T^nA)^{\oplus 3}$;
 		\item For any $l+1\leqslant i\leqslant 3$, find  $c^{(i)}_2\in T^{-1}(c_1^{(i)}) $ ($\in (T^{n-2}A/T^nA)^{\oplus 3} $), $c^{(i)}_3\in T^{-1}(c_2^{(i)}) $ ($\in (T^{n-3}A/T^nA)^{\oplus 3} $), $\cdots$, 	and 	$c_n^{(i)}   \in T^{-1}(c_{n-1}^{(i)})$ ($\in ( A/T^nA)^{\oplus 3} $);
 		\item For any $1\leqslant i\leqslant l$, the original vector $c_{k_i}^{(i)}$   sits in $(T^{n-k_i}A/T^nA)^{\oplus 3}$ because $T^{k_i}c_{k_i}^{(i)} = 0$. So we are able to find new vectors $c_{k_i+1}^{(i)}$, $c_{k_i+2}^{(i)}$, $\cdots$, $c_{n}^{(i)}$ such that $T c_{j}^{(i)}=c_{j-1}^{(i)}$ hold for all $2\leqslant j \leqslant n$.
 	\end{itemize}
 	It is easily seen that  
 	$$(c^{(1)}_1,\ldots, c^{(1)}_{n},c^{(2)}_1,\ldots, c^{(2)}_{n},  c^{(3)}_1,\ldots,c^{(3)}_{n})$$
 	forms a cyclic $\F_q$-basis of $(A/T^nA)^{\oplus 3}$ with $Tc_1^{(i)} = 0$ and $Tc_j^{(i)} = c_{j-1}^{(i)}$, for all $i = 1,\ldots,3$, $j = 1,\ldots,n$.

 	In the meantime, $$(d^{(1)}_1,\ldots, d^{(1)}_{k_1}, \ldots,    d^{(l)}_1,\ldots,d^{(l)}_{k_l}) := (\mu_1,\ldots,\mu_n)Z$$ is a cyclic basis of $W_n'$. So, in the same manner, we are able to   extend it to a cyclic $\F_q$-basis $$(d^{(1)}_1,\ldots, d^{(1)}_{n},d^{(2)}_1,\ldots, d^{(2)}_{n},  d^{(3)}_1,\ldots,d^{(3)}_{n})$$ of $(A/T^nA)^{\oplus 3}$.

 	Finally, we let $\sigma'$ be the $\F_q$-linear endomorphism of $(A/T^nA)^{\oplus 3}$ such that $\sigma'(c_j^{(i)}) = d_j^{(i)}$. This $\sigma'$ is the desired map. 
 	
 	$(2)\Rightarrow(3)$: In this proof,  we adopt the following notations   according to the standard correspondence \eqref{Eqt:standardconstruction}:
 	$$
 	\mathcal{V}= (V_1\subset V_2 \subset \cdots \subset V_{n})=(\nu_1<\ldots<\nu_n)\mapsto (l_1,\cdots,l_n)\mapsto (u_1,\cdots,u_n),	
 	$$
 	and
 	$$\mathcal{W} = (W_1\subset W_2 \subset \cdots \subset W_{n})=(\mu_1<\ldots<\mu_n)
 	\mapsto (z_1,\cdots,z_n)\mapsto (w_1,\cdots,w_n).
 	$$	 
 	Since (according to Theorem \ref{Lem:generalmjconjecture}) $\Gal({K_{n}}/\Knaught  )$ is isomorphic to $\Aut_A(\Ker(\Phi_{T^{n}}))$, we can take $f: = \rho_n  ^{-1}(\sigma)$. In other words, $f \in \Gal({K_{n}}/\Knaught  )$ satisfies $f|_{\Ker(\Phi_{T^{n}})} = \sigma$. 
 	We then  prove that $f(u_i) = w_{i}$ for $i = 1,\ldots,n$ by induction on $i$.

 	\begin{enumerate}
 		\item[(a)] 	The $i = 1$ case is simple:    $\sigma(V_1) = W_1$ implies that $\sigma(\nu_1) = r\mu_1$ for some $r \in \F_q^*$, and hence  $\sigma(\nu_1)^{q-1} = (r\mu_1)^{q-1}$, or $f(u_1) = w_1$. 
 		\item[(b)] Suppose that	$f(u_i) = w_{i}$ for all $i = 1,\ldots,n-1$ are   true.  As
 		$$W_n = \sigma(V_n) = \sigma(V_{n-1} + \F_q\nu_n)  = W_{n-1} + \F_q(\sigma(\nu_n)),$$
 		it follows that $\mu_n = \mu+r \sigma(\nu_n)$, where $\mu \in W_{n-1}$ and $r \in \F_q^{*}$. Therefore, one computes
 		\begin{align*}
 		z_n & = \lambda^{(w_{n-1})} \cdots \lambda^{(w_{2})} \lambda^{(w_1)}\left(\mu_n\right) \\
 		& = \lambda^{(w_{n-1})} \cdots \lambda^{(w_{2})} \lambda^{(w_{1})}\left(\mu+r\sigma(\nu_n)\right)\\	
 		& = \lambda^{(w_{n-1})} \cdots \lambda^{(w_{2})} \lambda^{(w_{1})}\left(r\sigma(\nu_n)\right)\qquad (W_{n-1} = \Ker(\lambda^{(w_{n-1})} \cdots \lambda^{(w_{2})} \lambda^{(w_{1})}))\\	
 		& = \lambda^{f(u_{n-1})} \cdots \lambda^{f(u_{2})} \lambda^{f(u_{1})}\left(r\sigma(\nu_n)\right) \qquad \text{(induction hypothesis)}\\
 		& = f(r\lambda^{\left(u_{n-1}\right)} \cdots \lambda^{\left(u_2\right)} \lambda^{\left(u_1\right)}\left(\nu_n\right)) = rf(l_n).
 		\end{align*}
 		Taking $(q-1)$-th power one gets $w_n=f(u_n)$, as required.
 		
 	\end{enumerate}
 	
 	$(3)\Rightarrow (1)$: If $f \in \operatorname{Gal}({K_{n}}/\Knaught  )$ is given which satisfies $f(u_i) = w_{i}$,     then $\BasicT \circ f = f \circ \BasicT$ (since $f(G)=G$). In the following, we show by induction on $i$ that   $f(V_i)=W_i$, for all $i = 1,\ldots,n$. This $f$ gives the isomorphism of $\mathcal{V}$ and $\mathcal{W}$.
 	\begin{enumerate}
 		\item[(a)] For the $i = 1$ case, as $u_1 = \nu_1^{q-1}$ and $w_1 = \mu_1^{q-1}$,  so $f(u_1) = w_1$ implies  $(f(\nu_1))^{q-1} = \mu_1^{q-1}$. Hence, we have $f(\nu_1) = r\mu_1$, for some $r\in \F_q^*$, and $f(V_1) =   W_1$.
 		\item[(b)] 	Assume that  $f(V_i)=W_i$ is proved for all $i = 1,\ldots,n-1$. Recall that we have  $l_n=\lambda^{\left(u_{n-1}\right)} \cdots \lambda^{\left(u_2\right)} \lambda^{\left(u_1\right)}\left(\nu_n\right)$. So  by  $f(u_{i}) = w_{i}$ one gets
 		$$f(l_n) = f(\lambda^{\left(u_{n-1}\right)} \cdots \lambda^{\left(u_2\right)} \lambda^{\left(u_1\right)}\left(\nu_n\right)) = \lambda^{\left(w_{n-1}\right)} \cdots \lambda^{\left(w_{2}\right)} \lambda^{\left(w_1\right)}\left(f(\nu_n)\right).$$
 		From $f(u_n) = w_{n}$, $u_n=l_n^{q-1}$ and $w_{n}=z_n^{q-1}$,  we can write $f(l_n) = rz_n$ for some  $r \in \F_q^*$.
 		Then using 	
 		$z_n =\lambda^{\left(w_{n-1}\right)} \cdots \lambda^{\left(w_{2}\right)} \lambda^{\left(w_{1}\right)}\left(\mu_n\right)$ we see that    $$ f(\nu_n)-r\mu_n \in \Ker(\lambda^{\left(w_{n-1}\right)} \cdots \lambda^{\left(w_{2}\right)} \lambda^{\left(w_1\right)})=W_{n-1} .$$ 
 		This proves the desired	relation	$$W_n = W_{n-1}+ \F_q\mu_n =    f(V_{n-1}) + \F_q(f(\nu_n)) = f(V_n).$$

 	\end{enumerate}
 \end{proof}

 \begin{cor}\label{root}
 	Let   ${\mathcal{V}} $ be a $\BasicT$-torsion $n$-flag   in $\Knaughtbar $. Let  $\mathcal{P}$ and $M(\mathcal{V};{\mathcal{V}_{n-1}})$ be defined as in Equations \eqref{Eqt:mathcalP} and \eqref{Eqt:MVVdiamond}, respectively. Set $u_n = \mathcal{P}(\mathcal{V})$. Then we have $$\{\mathcal{P}(\mathcal{V}'):\mathcal{V}' \in M(\mathcal{V};{\mathcal{V}_{n-1}})\}=  \{f(u_{n}):f \in \Gal({K_{n}}/\mathcal{F}_{{\mathcal{V}_{n-1}}})\}.$$
 \end{cor}
 \begin{proof}
 	On the one hand, let $u'_{n} \in \{\mathcal{P}(\mathcal{V}'):\mathcal{V}' \in M(\mathcal{V};{\mathcal{V}_{n-1}})\}$ be given, and we wish to find an $f \in \operatorname{Gal}({K_{n}}/\Knaught  )$ such that $ f(u_n)=u_n'$. Indeed, we can assume that $\mathcal{V}' \in M(\mathcal{V};{\mathcal{V}_{n-1}})$ is subject to  $u_n' = \mathcal{P}(\mathcal{V}')$. 	
 	Since $\mathcal{V}'$ is isomorphic to $\mathcal{V}$, by Lemma \ref{extend0}, there exists $f \in \operatorname{Gal}({K_{n}}/\Knaught  )$, such that $f(u_{i}) = u_{i}$, for $i = 1,\ldots,n-1$ and $f(u_n)=u'_n$, where $(u_1,\ldots,u_{n-1},u_{n}) = \Lambda(\mathcal{V})$ and $(u_1,\ldots,u_{n-1},u'_{n}) = \Lambda(\mathcal{V}')$ (as $\mathcal{V}$ and $\mathcal{V}'$ have the same parent flag).  Therefore,   $\mathcal{F}_{{\mathcal{V}_{n-1}}}$ is fixed by $f$, i.e., $f \in \Gal({K_{n}}/\mathcal{F}_{{\mathcal{V}_{n-1}}})$. 
 	
 	On the other hand, given any $f \in \Gal({K_{n}}/\mathcal{F}_{{\mathcal{V}_{n-1}}})$ and  let $u_n'=f(u_n)  $, we need to construct  $\mathcal{V}' \in M(\mathcal{V},{\mathcal{V}_{n-1}})$ satisfying $\mathcal{P}(\mathcal{V}')=u_n'$. The method is in fact a routine --- We start from the data    
 	$\mathcal{V} = (\nu_1<\cdots<\nu_n)$ {and } $(u_1,\ldots,u_n) = \Lambda(\mathcal{V}) $, and let $\nu_{n}' = f(\nu_{n})$.   Construct an $n$-flag $\mathcal{V}' = (\nu_1<\ldots<\nu_{n-1}<\nu_{n}')$. 
 	It is easy to get $\Lambda(\mathcal{V}') = (u_1,\ldots,u_{n-1},u_n')$ by definition of $\Lambda$ and $\mathcal{V}'$ is $\BasicT$-torsion. Also, as  $f(u_i) = u_i$, for $i = 1,\ldots,n-1$ and $f(u_n) = u_{n}'$,   by Lemma \ref{extend0}, $\mathcal{V}'$ is isomorphic to $\mathcal{V}$. Therefore, we have  $\mathcal{V}' \in M(\mathcal{V},{\mathcal{V}_{n-1}})$,  as required.
 \end{proof}

 \subsubsection{Proofs of Proposition \ref{prop:isomorphismFVbetaV} and  \ref{Prop:main}} Finally, we finish the desired proofs of our main propositions.   The first Proposition  \ref{prop:isomorphismFVbetaV} is a  direct consequence of Lemma \ref{extend0}.  
 To prove the second Proposition \ref{Prop:main},  we use Corollary \ref{root} and the definition of $\minimalpoly{\mathcal{V}}$ to rewrite 
 $$\minimalpoly{\mathcal{V}}(X) = \prod_{u_n' \in \{f(u_{n}):f \in \Gal({K_{n}}/\mathcal{F}_{{\mathcal{V}_{n-1}}})\}}  (X-u_n').$$
 By standard Galois theory, we have $\minimalpoly{\mathcal{V}} \in \mathcal{F}_{{\mathcal{V}_{n-1}}}[X]$ and it is irreducible over $\mathcal{F}_{{\mathcal{V}_{n-1}}}$. Since $u_n$ is a root of $\minimalpoly{\mathcal{V}}$, the   statement is  evident now.

 
 \subsection{The $g$-evaluation map}
  Now consider a general $A$-field $L$. Let $g\in \Lbar$ be a fixed element. Define the $g$-evaluation map $$\iotag :~\Fq[G]\to \Fq(g) ,   \quad f(G)\mapsto f(g) .$$
 	Let $\mathcal{O}$ be the \textit{integral closure} of $\Fq[G]$ in $\Knaughtbar=\overline{\Fq(G)}$. We can extend $\iotag $   to a homomorphism  (not necessarily unique) from $\mathcal{O}$ to $\overline{\F_q(g)}$ which is also denoted by $\iotag $ and fitting  the following commutative diagram:
 \begin{equation}\label{Eqt:iotagdiagram}
 \xymatrix{
 	\Fq[G] \ar@{^{(}->}[d]^{i} \ar[r]^{\iotag }  & \F_q(g)\ar@{^{(}->}[d]^{i}  \\
 	 	\mathcal{O}\ar[r]^{\iotag } &\overline{\F_q(g)}  .
 	 }
 \end{equation}
 Similar to the Drinfeld module $\BasicT  =  -\tau^3 + G \tau ^2 + 1 $ (over $\Knaught $),   we shall  consider
 \begin{equation}
 \label{Eqt:phignaught}
 \phignaught =-\tau^3 + g  \tau ^2 + 1 ,  
 \end{equation}
  a normalized $(3,2)$-type Drinfeld module over $\Lbar$ determined by $g$. Then clearly,   
  $\iotag$ intertwines $\BasicT$ and $\phignaught$, and hence   $\iotag(\Ker(\BasicT))\subset \Ker(\phignaught)$. Moreover, for any nonzero $x \in \Ker(\BasicT)$, we have
 $ -x^{q^3-1}+Gx^{q^2-1}+1 =0$, which implies that $\iotag(x)\in\Ker\phignaught$ satisfies $-\iotag(x)^{q^3-1}+g\iotag(x)^{q^2-1}+1 = 0 $, and thus $\iotag(x)$ is not zero.   Hence, $\iotag: ~\Ker\BasicT\to \Ker\phignaught$ is an injection. By   Lemma \ref{Lem:moduleiso}, it is indeed an isomorphism of $A$-modules. For the same reason, we see that   for every $i\geqslant 1$,      \begin{equation}
   \label{Eqt:iotagbijective}\iotag:~\Ker(\Basic_{T^i}) \to \Ker(\phignaughtnoT_{T^i}) 
   \end{equation}  
    is   an isomorphism of $A$-modules.

  \subsection{The affine curve of a $\BasicT$-torsion flag and its $L$-points} 
 \label{Section6.5}

 Suppose that we have a $\BasicT$-torsion $n$-flag $\mathcal{V} = (V_1\subset\ldots \subset V_n)$ in $\Knaughtbar $ which is subordinate to $\NODE_{i_1,\cdots, i_n}$. Write $\mathcal{V}_k = (V_1\subset \ldots \subset V_k)$ ($k=1,\cdots, n-1$) for ancestor flags of $\mathcal{V}_n:=\mathcal{V}$. One could treat $\mathcal{V}_0$ as the trivial vector space $\set{0}$.

 Define $\OV{0}=\Fq[G]$ ($\subset \mathcal{F}_{\mathcal{V}_0}=\Knaught=\Fq(G)$). As usual, we denote the integral closure of $\Fq[G]$ in $\Knaughtbar $ by $\mathcal{O}$.  In what follows, we let $\mathcal{O}_{\mathcal{V}_j}$ be the \textit{integral closure} of   $\F_q[G]$ in $\mathcal{F}_{\mathcal{V}_j}$ (i.e. $\mathcal{O}_{\mathcal{V}_j}=\mathcal{O}\cap \mathcal{F}_{\mathcal{V}_j}$). So one has the obvious inclusions $\Ker\Basic_{T^j}\subset \mathcal{O}_{\mathcal{V}_j}\subset \mathcal{F}_{\mathcal{V}_j}\subset K_j$.

  The next proposition refines the  result of Proposition \ref{Prop:main}.  Recall   that $\minimalpoly{\mathcal{V}_j}\in \mathcal{F}_{{\mathcal{V}_{j-1}}}[X]$ is the minimal polynomial of $\mathcal{V}_j$ ($j=1,\cdots,n$).
 
 \begin{prop}\label{Prop:betainOV}   We have   	
 	$\minimalpoly{\mathcal{V}_j} \in \OV{j-1}[X]$ where $\mathcal{V}_{j-1} = (V_1\subset \ldots \subset V_{j-1})$ is the parent  flag of $\mathcal{V}_j$.
 \end{prop}
 \begin{proof}It suffices to prove the $j=n$ case. 
 	For any $\mathcal{V}' = (\nu_1<\ldots<\nu_{n-1}<\nu'_n) \in M({\mathcal{V}},{\mathcal{V}_{n-1}})$, we have $\nu_1,\ldots,\nu_{n-1},\nu'_n \in \Ker(\Basic_{T^n})\subset K_n$. Therefore, the corresponding $u'_n=\mathcal{P}(\mathcal{V}')$   is integral over $\Fq[G]$.  We concluded that for any $\mathcal{V}'  \in M({\mathcal{V}},{\mathcal{V}_{n-1}})$ one has $\mathcal{P}(\mathcal{V}')\in \mathcal{O}$, and hence   $\minimalpoly{\mathcal{V}}\in \mathcal{O}[X]$ by its definition.
 	
 	So, by Proposition \ref{Prop:main}, we have  $\minimalpoly{\mathcal{V}}\in \mathcal{O}[X]\cap \mathcal{F}_{\mathcal{V}_{n-1}}[X]$, which implies that $\minimalpoly{\mathcal{V}}\in \OV{n-1}[X]$. 
 \end{proof}
 

 We now introduce an important affine curve associated with $\mathcal{V}$ such that its   function field is $\mathcal{F}_{\mathcal{V}}$.
   
 \begin{defn}\label{defn:affinecurveCV}With notations as earlier, define $\mathcal{C}_{\mathcal{V}}$ to be an 
 	affine curve  
 	over $\Fq$  
 in variables $(G,X_1,\ldots,X_n)$ with defining equations
 $$
 \minimalpoly{\mathcal{V}_1} (X_1)=0,\quad \minimalpoly{\mathcal{V}_2} (X_2)=0, \quad \cdots,\quad \minimalpoly{\mathcal{V}_n} (X_n)=0.
 $$  
 
\end{defn}
  
 \smallskip
 
Next, consider   an arbitrary $A$-field, say $L$.     We give an inductive characterization of $\Lbar$ (or $L$)-points of $\mathcal{C}_{\mathcal{V}}$: 
 \begin{itemize}
 	\item First, suppose that $g\in \Lbar$ is given. We can defined the $g$-evaluation map $\iotag :$ $\OV{0}(=\Fq[G])\to \Fq(g)$. Recall from Proposition \ref{Prop:betainOV} that we have   $\minimalpoly{\mathcal{V}_1}\in \OV{0}[X]$.  Indeed, according to Equation \eqref{Eqt:betaV1}, we have $\minimalpoly{\mathcal{V}_1}=X^{q^2+q+1}-GX^{q+1}-1$. So, one can do $g$-evaluation to coefficients of $\minimalpoly{\mathcal{V}_1}$, obtaining $\minimalpoly{\mathcal{V}_1}^{(g)}=X^{q^2+q+1}-gX^{q+1}-1\in \Fq(g)[X]$.
 	\item Second, suppose that $x_1\in \Lbar^*$ is subject to $\minimalpoly{\mathcal{V}_1}^{(g)}(x_1)=0$. Then we can define a $(g,x_1)$-evaluation map $\OV{1}\to \overline{\Fq(g)}$ ($\subset \Lbar$). Recall here that $\OV{1}$ is the integral closure of $\OV{0}=\Fq[G]$ in $\mathcal{F}_{\mathcal{V}_1}$. Since $\mathcal{F}_{\mathcal{V}_1}$ is defined to be $\Fq(G,u_1)$, and $\minimalpoly{\mathcal{V}_{1}}$ is the minimal polynomial of $u_1$ in $\mathcal{F}_{\mathcal{V}_0}=\Knaught=\Fq(G)$ (Proposition \ref{Prop:main}), 
 	the relation $\minimalpoly{\mathcal{V}_1}^{(g)}(x_1)=0$   ensures a well-defined map, the $(g,x_1)$-evaluation $\OV{1} \to \overline{\Fq(g)}$ mapping $G\mapsto g$ and $u_1\mapsto x_1$.  From Proposition \ref{Prop:betainOV} again,  we have   $\minimalpoly{\mathcal{V}_2}\in \OV{1}[X]$, and hence  one can do $(g,x_1)$-evaluation to coefficients of $\minimalpoly{\mathcal{V}_2}$, obtaining $\minimalpoly{\mathcal{V}_2}^{(g,x_1)}\in \overline{\Fq(g)}[X]$.
 	\item Suppose that   $x_{2}\in \Lbar^*$ is subject to $\minimalpoly{\mathcal{V}_2}^{(g,x_1)}(x_{2 })=0$. Then one can define the $(g,x_1, x_{2})$-evaluation map $\OV{2}\to \overline{\Fq(g)}$. Thus from $\minimalpoly{\mathcal{V}_{3}}\in \OV{2}[X]$ we obtain $\minimalpoly{\mathcal{V}_{3}}^{(g,x_1, x_{2} )}\in \overline{\Fq(g)}[X]$.
 	
 	\item Inductively, suppose that $\minimalpoly{\mathcal{V}_i}^{(g,x_1,\cdots,x_{i-1})}\in \overline{\Fq(g)}[X]$ is obtained. Then for any $x_{i}\in \Lbar^*$ subject to $\minimalpoly{\mathcal{V}_i}^{(g,x_1,\cdots,x_{i-1})}(x_{i })=0$, one can define the $(g,x_1,\cdots,x_{i-1},x_{i})$-evaluation map $\OV{i}\to \overline{\Fq(g)}$. Thus from $\minimalpoly{\mathcal{V}_{i+1}}\in \OV{i}[X]$ we obtain $\minimalpoly{\mathcal{V}_{i+1}}^{(g,x_1,\cdots,x_{i-1},x_i)}\in \overline{\Fq(g)}[X]$.
 	
 	\item Following these  steps, we started from $ g\in \Lbar$ and selected elements    $x_1$, $\cdots, x_n$ (all in $\Lbar^*$) such that  $\minimalpoly{\mathcal{V}_n}^{(g,x_1,\cdots,x_{n-1})}(x_n)=0$.
 \end{itemize}
 
 In the sequel, by saying  an \textbf{$\Lbar$-point} of $\mathcal{C}_{\mathcal{V}}$, we mean   an array $(g,x_1,\cdots,x_n)\in \Lbar\times (\Lbar^*)^{\times n}$ satisfying the equations
 	\begin{eqnarray*}
 		\minimalpoly{\mathcal{V}_1}^{(g)}(x_1)&=&0;\\
 		\minimalpoly{\mathcal{V}_2}^{(g,x_1) }(x_2)&=&0;\\
 		&\cdots&\\
 		\minimalpoly{\mathcal{V}_n}^{(g,x_1,\cdots,x_{n-1})}(x_n)&=&0.
 	\end{eqnarray*} An \textbf{$L$-point} of $\mathcal{C}_{\mathcal{V}}$ is an array $(g,x_1,\cdots,x_n)\in L\times (L^*)^{\times n}$ satisfying the same requirements.
 
  The set of $\Lbar$ (resp. $L$)-points of $\mathcal{C}_{\mathcal{V}}$ is denoted by $\mathcal{C}_{\mathcal{V}}(\Lbar)$ (resp. $\mathcal{C}_{\mathcal{V}}(L)$). For a fixed element $g $,  we also denote by
\begin{equation}
\label{Eqt:CgV}
\mathcal{C}^{(g)}_{\mathcal{V}}(\Lbar): =\set{(x_1,\cdots,x_n)\in (\Lbar^*)^{\times n} ~|~(g,x_1,\cdots,x_n) \mbox{ is an } \Lbar \mbox{-point of }\mathcal{C}_{\mathcal{V}}} \end{equation}
the set of  $\Lbar$-points of $\mathcal{C}_{\mathcal{V}}$ leading with $g$. We define   $\mathcal{C}^{(g)}_{\mathcal{V}}(L):=\mathcal{C}^{(g)}_{\mathcal{V}}(\Lbar)\cap (L^*)^{\times n}$.
 
 \subsection{From $\BasicT$-torsion flags to  $\Knaughtbar $-points} Particularly, if the $A$-field $L=\Knaughtbar $, then according to Proposition \ref{Prop:main}, the array $(G,u_1,\cdots,u_n)$ is certainly a $\Knaughtbar $-point of $\mathcal{C}_{\mathcal{V}}$. More generally, we could consider the set of all    $\Knaughtbar $-points of $\mathcal{C}_{\mathcal{V}}$ leading with $G$:
 $$
 \mathcal{C}_{\mathcal{V}}^{( G )}(\Knaughtbar )=\set{(w_1,\cdots,w_n)\in (\Knaughtbar^*)^{\times n} ~|~(G,w_1,\cdots,w_n) \mbox{ is a } \Knaughtbar \mbox{-point of }\mathcal{C}_{\mathcal{V}}}.
 $$
 
 Recall   the notation 
 $\mathcal{G}(\overline{\Knaught},\Basic;\NODE_{i_1,\cdots, i_n})$, referring to the set of all $\BasicT$-torsion $n$-flags in $\overline{\Knaught}$ which are subordinate to    $ \NODE_{i_1,\cdots, i_n}$
 (see Notation \ref{Notation:pair}).  
  Also recall the standard correspondence map $\Lambda:  \set{n\mbox{-flags in }\Knaughtbar}\to ~(\Knaughtbar^*)^{\times n}$ presented in Section \ref{Sec:flagsnotations}.

   \begin{prop}\label{Prop:overlineLambda} 
 The map  $\overline{\Lambda}:$ 
  $\mathcal{G}(\overline{\Knaught},\Basic;\NODE_{i_1,\cdots, i_n})$ $\to$ $\mathcal{C}_{\mathcal{V}}^{( G )}(\Knaughtbar )$ defined by  
 $\mathcal{W}\mapsto \Lambda(\mathcal{W})$  
    is a bijection. \end{prop}

\begin{proof} First of all, we show that $\overline{\Lambda}$ is well-defined. That is, if $\mathcal{W}$ is a  $\BasicT$-torsion $n$-flags   subordinate to    $ \NODE_{i_1,\cdots, i_n}$, then for $\Lambda(\mathcal{W})=(w_1,\cdots, w_n)$, $(G,w_1,\cdots,w_n)$   is indeed a   $\Knaughtbar${-point of } $\mathcal{C}_{\mathcal{V}}$. In fact, we have $\minimalpoly{\mathcal{W}_i} (w_i)=0$ for all $i = 1,\ldots,n$. 	
	Then by Proposition \ref{prop:isomorphismFVbetaV}, we have an isomorphism $\psi: \mathcal{F}_{\mathcal{V}}\to \mathcal{F}_{\mathcal{W}}$ of function fields  such that    $\psi(u_{i}) = w_{i}$  and  $\psi(\minimalpoly{\mathcal{V}_i})=\minimalpoly{\mathcal{W}_i}$ for all $i$, i.e. $\minimalpoly{\mathcal{V}_i}^{(G;w_1,\cdots, w_{i-1})}=\minimalpoly{\mathcal{W}_i}$. We thus have  $\minimalpoly{\mathcal{V}_i}^{(G;w_1,\cdots, w_{i-1})}(w_i)=0$, as required.
	
	Since $\Lambda$ is injective, so is $\overline{\Lambda}$. It remains to show that $\overline{\Lambda}$ is surjective. 
	We prove it by induction on $n$. 	
	Starting with the $n=1$ case, we have $\mathcal{G}(\overline{\Knaught},\Basic;\NODE_{ 1})= M(\mathcal{V} ,\mathcal{V}_{0})$ formed by  all  $\BasicT$-torsion $1$-flags. In the meantime, $\mathcal{C}_{\mathcal{V}}^{( G )}(\Knaughtbar )$ is the set of roots of $\minimalpoly{\mathcal{V}}$. As $n=1$, the map $\overline{\Lambda}$ now coincides with $\mathcal{P}$ (defined by Equation \eqref{Eqt:mathcalP}), and by definition of $\minimalpoly{\mathcal{V}}$, $\overline{\Lambda}$ is a bijection.

	Suppose that this lemma holds true for the parent flag  of $\mathcal{V}$ which we denote by $\mathcal{V}_{n-1}$ and is subordinate to $\NODE_{i_1,\cdots, i_{n-1}}$. So, given any  $\Knaughtbar$-point $(G,w_1,\cdots,w_n)$ of   $\mathcal{C}_{\mathcal{V}}$, we are able find $\mathcal{W}_{n-1}\in  \mathcal{G}(\overline{\Knaught},\Basic;\NODE_{i_1,\cdots, i_{n-1}})$ such that $\overline{\Lambda}(\mathcal{W}_{n-1})=\Lambda(\mathcal{W}_{n-1})=(w_1,\cdots,w_{n-1})$. Pick any $\mathcal{W}_{n}\in  \mathcal{G}(\overline{\Knaught},\Basic;\NODE_{i_1,\cdots, i_{n}})$ whose parent flag is $\mathcal{W}_{n-1}$. Since $\mathcal{W}_{n}$ is isomorphic to $\mathcal{V}_{n}=\mathcal{V} $, we have $\minimalpoly{\mathcal{V}_n}^{(G;w_1,\cdots, w_{n-1})}=\minimalpoly{\mathcal{W}_n}$ (by Proposition \ref{prop:isomorphismFVbetaV} again).
	It follows from $\minimalpoly{\mathcal{V}_n}^{(G;w_1,\cdots, w_{n-1})}(w_n)=0$ that $w_n$ is indeed a root of $\minimalpoly{\mathcal{W}_n}$. Thus by the construction of $\minimalpoly{\mathcal{W}_n}$, there exists some $\mathcal{W}'_n\in 
	M({\mathcal{W}_n},{\mathcal{W}_{n-1}})$ (see \eqref{Eqt:MVVdiamond} for this notation) such that $\mathcal{P}(\mathcal{W}'_n)=w_n$, i.e., $\Lambda(\mathcal{W}'_{n})=(w_1,\cdots,w_n)$. This $\mathcal{W}'_n$ certainly belongs to $   \mathcal{G}(\overline{\Knaught},\Basic;\NODE_{i_1,\cdots, i_{n}})$, and is the desired inverse-image of $(w_1,\cdots,w_n)$ by $\overline{\Lambda}$. 
\end{proof}

  \subsection{From $\phignaught$-torsion flags to  $\Lbar $-points}
  Let $L$ be an $A$-field and $g\in \Lbar$.   
  Recall the Drinfeld module $\phignaughtnoT$ defined by Equation \eqref{Eqt:phignaught}. 
 Following  Notation \ref{Notation:pair},  we use 
 $\mathcal{G}(\Lbar,\phignaughtnoT ;\NODE_{i_1,\cdots, i_n})$  to denote the set of all $\phignaught$-torsion $n$-flags in $\Lbar$ which are subordinate to    $ \NODE_{i_1,\cdots, i_n}$. Recall that $\mathcal{C}_{\mathcal{V}}^{( g )}(\Lbar )$ stands for the set of $\Lbar$-points of $\mathcal{C}_{\mathcal{V}}$ leading with $g$ (see Equation \eqref{Eqt:CgV}).
  \begin{prop}\label{Prop:underlineLambda}
 	The map $\underline{\Lambda}:$ $\mathcal{G}(\Lbar,\phignaughtnoT ;\NODE_{i_1,\cdots, i_n})$ $\to$ $\mathcal{C}_{\mathcal{V}}^{( g )}(\Lbar )$ defined by 
 	$\mathcal{W}\mapsto \Lambda(\mathcal{W})$  	is a bijection. Here     $\Lambda$ is the standard correspondence map $   \set{n\mbox{-flags in }\Lbar}\to ~(\Lbar^*)^{\times n}$ (see Section \ref{Sec:flagsnotations}).

 \end{prop} 
\begin{proof}
	Let us establish the following commutative diagram:
	\begin{equation}\label{BigDiagrame}
	\xymatrix{
		\mathcal{G}(\overline{\Knaught},\Basic;\NODE_{i_1,\cdots, i_n}) \ar[d]_{\overline{\Lambda}} \ar[r]^{S_g} &  \mathcal{G}(\Lbar,\phignaughtnoT ;\NODE_{i_1,\cdots, i_n}) \ar[d]_{\underline{\Lambda}} \\
		\mathcal{C}_{\mathcal{V}}^{( G )}(\Knaughtbar )  \ar[r]^{s_g}  & \mathcal{C}_{\mathcal{V}}^{(g)}(\overline{L}),
	}
	\end{equation}
	where  \begin{itemize}
		\item $\overline{\Lambda}$ is the bijection given by Proposition \ref{Prop:overlineLambda};
		\item $S_g$ is induced by the isomorphism $\iotag$ in Equation \eqref{Eqt:iotagbijective};
		\item $s_g$ is induced by the map $\iotag$ in Diagram \eqref{Eqt:iotagdiagram}:  $$(w_1,\cdots,w_n)\mapsto (\iotag(w_1),\cdots,\iotag(w_n)).$$
	\end{itemize}
In this diagram,   {the map $s_g$ is well defined}, primarily due to $$0=\iotag(\minimalpoly{\mathcal{V}_i}^{(G;w_1,\cdots, w_{i-1})}(w_i))=\minimalpoly{\mathcal{V}_i}^{(g;\iotag(w_1),\cdots, \iotag(w_{i-1}))}(\iotag(w_i)). $$ Also, 
  {Diagram \eqref{BigDiagrame} is commutative}, because $S_g$ is clearly a bijection, and it follows from the definition of $\underline{\Lambda}$ that $\underline{\Lambda}=s_g\circ \overline{\Lambda}\circ S_g^{-1}$. The map  $\underline{\Lambda}$ is injective (because $\Lambda$ is injective) implies that $s_g$ must be an injection.
Since the cardinality of $\mathcal{C}_{\mathcal{V}}^{(g)}(\overline{L})$ can not exceed that of  $\mathcal{C}_{\mathcal{V}}^{(G)}( {\Knaughtbar })$, $s_g$ is indeed a  bijection, and so is  $\underline{\Lambda}$.   
\end{proof}

 \subsection{The two main theorems}\label{Sec:L0case32typeDrinfeldmodularcurves}
 We are ready to state our  main theorems. The first one   presents $(3,2)$-type normalized Drinfeld modular curves in terms of the function fields we have used.  
 	Let    $\NODE_{i_1,\cdots, i_n}$ be a $T$-torsion flag class (in $ S^3$). Take  
any $\BasicT$-torsion flag   $\mathcal{V} = (V_1\subset\ldots \subset V_n)$ (in $\Knaughtbar $) which is    {subordinate to}   $\NODE_{i_1,\cdots, i_n}$.

 \begin{thm}\label{Thm:thehardisomorphism}
 	The $(3,2)$-type normalized Drinfeld modular curve $\dot{X}^{(3,2)}_{i_1,\cdots,i_{n}  }$ over $\Fq$ is given by the affine curve $\mathcal{C}_{\mathcal{V}}$ (see Definition \ref{defn:affinecurveCV}). 
 	Consequently,   the function field    of  $\dot{X}^{(3,2)}_{i_1,\cdots,i_{n}  }$		is isomorphic to   $\mathcal{F}_{\mathcal{V}}$, i.e., 	
 	$$\dot{\mathcal{F}}^{(3,2)} _{i_1,\cdots,i_{n}  }\cong \mathcal{F}_{\mathcal{V}}.$$ 
 	
 	\end{thm}
 \begin{proof}
 	 It amounts to establish a one-to-one correspondence between $\mathcal{G}(L,\phignaughtnoT ;\NODE_{i_1,\cdots, i_n})$ (see Notation \ref{Notation:pair}) and $\mathcal{C}_{\mathcal{V}}^{( g )}(L )=\mathcal{C}^{(g)}_{\mathcal{V}}(\Lbar)\cap (L^*)^{\times n}$, for any $A$-field $L$ and $g\in L$. Now Proposition \ref{Prop:underlineLambda} gives a bijection  $\underline{\Lambda}$ from $\mathcal{G}(\Lbar,\phignaughtnoT ;\NODE_{i_1,\cdots, i_n})$ $\to$ $\mathcal{C}_{\mathcal{V}}^{( g )}(\Lbar )$. The conclusion follows  immediately from Corollary \ref{Cor:FromflagoverLtoui}.
 \end{proof}

  We state our second theorem. Assume that   $\NODE_{i_1,\cdots, i_n}$ has totally $k$ child nodes $\NODE_{i_1,\cdots, i_n,1}$, $\NODE_{i_1,\cdots, i_n,2}$, $\cdots$, and $\NODE_{i_1,\cdots, i_n,k}$. Suppose that $ {\mathcal{V}_{n+1}^{(1)}}$, $\cdots$, $ {\mathcal{V}_{n+1}^{(k)}}$ are the child $\BasicT$-torsion flags of $\mathcal{V}_n=\mathcal{V}$ subordinate to $\NODE_{i_1,\cdots, i_n,1}$, $\NODE_{i_1,\cdots, i_n,2}$, $\cdots$, and $\NODE_{i_1,\cdots, i_n,k}$, respectively.

    \begin{thm}\label{Thm:NodeToFactor} [The Factorization  Theorem] 
    		Let   $\minimalpoly{\mathcal{V}_{n+1}^{(i)}}\in \mathcal{F}_{\mathcal{V}}[X]$ be the minimal polynomial of $\mathcal{V}_{n+1}^{(i)}$, for $i=1,\cdots,k$. Then these polynomials are mutually exclusive and   irreducible factors of 
    	the modular polynomial $\kappa^{(u_n)}$ (see Equation \eqref{Eqt:kappauT}), i.e., the following  decomposition holds in $  \mathcal{F}_{\mathcal{V}}[X] $  	
    	:
     	\begin{equation}\label{Eqt:kappadecompose}
     	     	\kappa^{(u_n)} =\minimalpoly{\mathcal{V}_{n+1}^{(1)}}  ~\minimalpoly{\mathcal{V}_{n+1}^{(2)}} ~ \cdots ~\minimalpoly{\mathcal{V}_{n+1}^{(k)}} .
    	    	\end{equation}

    \end{thm}So, this theorem reveals a relation between  nodes of the $T$-torsion  tree $\Ttorsiontree^3$  and    the modular polynomial --- There is a one-to-one correspondence between child nodes of $\NODE_{i_1,\cdots, i_n}$ and  monic irreducible factors of the modular polynomial $\kappausingle{u_n}$.  

    \begin{proof}      	Suppose that $ \mathcal{V}  =\Theta(u_1,u_2,\cdots,u_n)$. 
    	Any child $\BasicT$-torsion flag of $\mathcal{V}$ is given by $\Theta(u_1,u_2,\cdots,u_n,u_{n+1})$ 
    	where $ u_{n+1}$ is a root of $\kappa^{(u_n)}$ (Theorem \ref{Thm:nflagcondition0}). 	Thus, we can write 
    	\begin{eqnarray*} 
    		\kappau{u_n}{X}&=& (X- u_{n+1}^{(1)})(X- u_{n+1}^{(2)})\cdots (X- u_{n+1}^{(q^2+q+1)}).
    	\end{eqnarray*}
    	Here $ u_{n+1}^{(i)}$, $i=1,\cdots, q^2+q+1$, are all roots of $\kappa^{(u_n)}$.
     Then the decomposition \eqref{Eqt:kappadecompose} follows by the definition of each $\minimalpoly{\mathcal{V}_{n+1}^{(i)}}$ (cf. Equation \eqref{beta}). Since  $\kappausingle{u_n}$ has no multiple roots (by its construction in Equation \eqref{Eqt:kappauT}), these $\minimalpoly{\mathcal{V}_{n+1}^{(i)}}$ are mutually  different.


    \end{proof}

     \begin{remark}\label{Rmk:32tomj}
     	More generally,  it is natural to expect  that our results in this section  are also true for general  $(m,j)$-type Drinfeld module (over $\Knaught $)
    $$\BasicT  =  -\tau^m + G \tau ^j + 1 .$$ (As usual, we require $m\geqslant 3$ and $1\leqslant j\leqslant m-1$ are coprime.) Why we only claim the $(m,j)=(3,2)$-case? In fact, our method to prove Lemma \ref{Lem:rho1iso} heavily relies on the specific values of $m=3$ and $j=2$. Therefore, we cannot easily extend our method which is valid in the $(m,j)=(3,2)$-case to general $(m,j)$-cases.   However, if Lemma \ref{Lem:rho1iso} holds true for an $(m,j)$-type Drinfeld module as defined above, then Theorem \ref{Lem:generalmjconjecture} and all other statements can also be proved. For a further related fact regarding Theorem \ref{Lem:generalmjconjecture}, please refer to \cite{MR3425215}*{Theorem 6} (known as the generalized iteration conjecture); see also \cite{MR1816069}*{$\S$ 19}, and the original theorem by Moore \cites{MR1557441,MR1467085}.
\end{remark}


 \section{A list of   $(3,2)$-type normalized Drinfeld modular curves up to level $4$}\label{Sec:Normalized32things}
 
 In this part,  we consider  the 23 nodes $\NODE_{i_1,\cdots, i_n}$ of   the $T$-torsion tree $\Ttorsiontree^3$ with $1\leqslant n\leqslant 4$. By choosing arbitrary $\BasicT$-torsion flags $\mathcal{V}_{i_1,\cdots,i_{n}}$ subordinate to $\NODE_{i_1,\cdots, i_n}$,  we present results of the $(3,2)$-type normalized Drinfeld modular curves  $\dot{X}^{(3,2)}_{i_1,\cdots,i_{n}  }$ (abbreviated to $\dot{X}_{i_1,\cdots,i_{n}  }$) in terms of the associated minimal polynomials $\minimalpoly{\mathcal{V}_{i_1,\cdots,i_{n}}}$ (abbreviated to $\minimalpoly{i_1,\cdots,i_{n}}$) and    function fields $\mathcal{F}^{(3,2)}_{i_1,\cdots,i_{n}  }$ (abbreviated to $\mathcal{F}_{i_1,\cdots,i_{n}  }$).    Let us employ the conventions and symbols from the previous sections.
 Consider the function field $\Knaught :=\Fq(G)$ and its closure $\Knaughtbar $. In the rest of this section, we only consider the $A$-field $L:= \Knaughtbar $.

 We need to introduce more notations. Define two polynomials  $\betausingle{u},    \gammausingle{u}\in L[X]$  parameterized by $u$:
 \begin{eqnarray}\label{Eqt:betauT0}
 \betau{u}{X}&:=&X^{q+1}+\frac{X}{u ^{q+1}} +\frac{1}{u },
 \\\label{Eqt:gammauT0}
 \mbox{and ~}~	\gammau{u}{X}&:=&X( \betau{u }{X} )^{q-1}-u=X( X^{q+1}+\frac{X}{u ^{q+1}} +\frac{1}{u } )^{q-1}-u \,.
 \end{eqnarray}
 
 Meanwhile, we  define two polynomials parameterized by $u$ and $v$:
 \begin{eqnarray*}
 	\etau{u,v}{X}&:=&X-\frac{1}{uv},\\
 	\xiu{u,v}{X}&:=&X\big(\etau{u,v}{X}\big)^{q-1}-u=X(X-\frac{1}{uv})^{q-1}-u\,.
 \end{eqnarray*}

 Now, we specifically state {normalized Drinfeld modular curves  } up to level $n=4$ (omitting the specific steps to obtain them). We start from $n=1$, i.e. the  root node  $\NODE_1$ of   $\Ttorsiontree^3$. According to Theorem \ref{Thm:nflagcondition0} and Example \ref{Example:n=1}, we have $$\dot{\mathcal{F}}_{1}:= \Fq(G)(u_1)=\Fq(u_1),\quad (\mbox{with }~ 	G=\frac{ u_1^{q^2 + q +1 } -1 }{  u_1^{q+1}}
 ). $$   Here the notation $u_1$ stands for an indeterminate, which should not be confused with the original parameter $u_1\in L^*$ of $\BasicT$-torsion $1$-flags.
  
 
 For $n=2,3,4$,	the function fields for normalized Drinfeld modular curves of  $(3,2)$-type up of level $n$ of the $T$-torsion tree $\Ttorsiontree^3$ are as described in terms of their minimal polynomials --- 
 	\begin{itemize}
 		\item[$\bullet$] At $\NODE_{1,1   }$, the minimal polynomial is given by  $\minimalpoly{1,1}=\gammausingle{u_1} $, and the function field  $\dot{\mathcal{F}}_{ 1,1} =\dot{\mathcal{F}}_{1}(u_2)$ is determined by
 		$$u_2( u_2^{q+1}+\frac{u_2}{u_1^{q+1}} +\frac{1}{u_1} )^{q-1}=u_1   \quad (\mbox{i.e.~}~~  \gammau{u_1}{u_2}=0).$$
 		\item[$\bullet$] At $\NODE_{1,2 }$, the minimal polynomial is given by  $\minimalpoly{1,2}=\betausingle{u_1} $, and the function field  $\dot{\mathcal{F}}_{ 1,2} =\dot{\mathcal{F}}_{1}(u_2)$ is determined by
 		$$u_2^{q+1}+\frac{u_2}{u_1^{q+1}} +\frac{1}{u_1}=0   \quad (\mbox{i.e.~}~~  \betau{u_1}{u_2}=0  ).$$
 		\item[$\bullet$] At $\NODE_{1,1,1 }$, the minimal polynomial is given by  $\minimalpoly{1,1,1}=\gammausingle{u_2} $, and the function field  $\dot{\mathcal{F}}_{ 1,1,1} =\dot{\mathcal{F}}_{1,1}(u_3)$ is determined by
 		$$u_3( u_3^{q+1}+\frac{u_3}{u_2^{q+1}} +\frac{1}{u_2} )^{q-1}=u_2 \quad  (\mbox{i.e.~}~~  \gammau{u_2}{u_3}=0  ).$$
 		\item[$\bullet$] At $\NODE_{1,1,2 }$, the minimal polynomial is given by  $\minimalpoly{1,1,2}=\betausingle{u_2} $, and the function field  $\dot{\mathcal{F}}_{ 1,1,2} =\dot{\mathcal{F}}_{1,1}(u_3)$ is determined by $$u_3^{q+1}+\frac{u_3}{u_2^{q+1}} +\frac{1}{u_2}=0 \quad (\mbox{i.e.~}~~
 		\betau{u_2}{u_3}=0 ).$$
 		\item[$\bullet$] At $\NODE_{1,2,1 }$, the minimal polynomial is given by  $\minimalpoly{1,2,1}=\xiusingle{u_1,u_2} $, and the function field  $\dot{\mathcal{F}}_{ 1,2,1} =\dot{\mathcal{F}}_{1,2}(u_3)$ is determined by
 		$$ u_3(u_3 -\frac{1}{u_1 u_2})^{q-1}=u_1  \quad(\mbox{i.e.~}~~   \xiu{u_1,u_{2}}{u_{3}}=0 ). $$
 		\item[$\bullet$] At $\NODE_{1,2,2 }$, the minimal polynomial is given by  $\minimalpoly{1,2,2}=\gammausingle{u_2} $, and the function field  $\dot{\mathcal{F}}_{ 1,2,2} =\dot{\mathcal{F}}_{1,2}(u_3)$ is determined by
 		$$u_3(u_3^{q+1}+\frac{u_3}{u_2^{q+1}}+\frac{1}{u_2})^{q-1}=u_2   \quad (\mbox{i.e.~}~~   \gammau{u_2}{u_3}=0 ). $$
 		\item[$\bullet$] At $\NODE_{1,2,3 }$, the minimal polynomial is given by  $\minimalpoly{1,2,3}=\etausingle{u_1,u_2} $, and the function field  $\dot{\mathcal{F}}_{ 1,2,3} =\dot{\mathcal{F}}_{1,2}(u_3)
 		$ is determined by   $$u_3 =\frac{1}{u_1 u_2}  \quad(\mbox{i.e.~}~~   \etau{u_1,u_{2}}{u_{3}}=0 ), $$ and thus $\dot{\mathcal{F}}_{1,2,3}=\dot{\mathcal{F}}_{1,2}$.
 		\item[$\bullet$] At $\NODE_{1,1,1,1}$, the minimal polynomial is given by  $\minimalpoly{1,1,1,1}=\gammausingle{u_3} $, and the function field  $\dot{\mathcal{F}}_{ 1,1,1,1} =\dot{\mathcal{F}}_{1,1,1}(u_4)$ is determined by $$u_4(u_4^{q+1}+\frac{u_4}{u_3^{q+1}}+\frac{1}{u_3})^{q-1}=u_3   \quad (\mbox{i.e.~}~~   \gammau{u_3}{u_4}=0 ). $$
 		\item[$\bullet$] At $\NODE_{1,1,1,2}$, the minimal polynomial is given by  $\minimalpoly{1,1,1,2}=\betausingle{u_3} $, and the function field  $\dot{\mathcal{F}}_{ 1,1,1,2} =\dot{\mathcal{F}}_{1,1,1}(u_4)$ is determined by
 		$$  u_4^{q+1}+\frac{u_4}{u_3^{q+1}}+\frac{1}{u_3}=0 \quad (\mbox{i.e.~}~~\quad  \betau{u_3}{u_4}=0 ). $$
 		\item[$\bullet$] At $\NODE_{1,1,2,1}$, the minimal polynomial is given by  $\minimalpoly{1,1,2,1}=\gammausingle{u_3} -\iota^{(u_1,u_2,u_3)} $    
 		where
 		$$\iota^{(u_1,u_2,u_3)}(X):= X \left( X - \frac{1}{u_2 u_3}  \right)^{q-1}\bigl(X \left( X - \frac{1}{u_2 u_3}  \right)^{q-1}-A\bigr)^{q-1}-Q,$$
 		with  $A=\frac{u_1^q({U}-1)}{1-U^{q+1} }$ , $Q=\frac{u_1^{q^2}}{(1- U^{q+1})^{q-1}}$, and  $U=u_1u_2u_3$; the function field  $\dot{\mathcal{F}}_{ 1,1,2,1} =\dot{\mathcal{F}}_{1,1,2}(u_4)$ is determined by   $	\gammau{u_3}{u_4}=\iota^{(u_1,u_2,u_3)}(u_4)$.  
 		\item[$\bullet$] At $\NODE_{1,1,2,2}$, the minimal polynomial is given by  $\minimalpoly{1,1,2,2}=\xiusingle{u_2,u_3} $, and the function field  $\dot{\mathcal{F}}_{ 1,1,2,2} =\dot{\mathcal{F}}_{1,1,2}(u_4)$ is determined by
 		$$u_4(u_4-\frac{1}{u_2u_3})^{q-1}=u_2  \quad (\mbox{i.e.~}~~  \xiu{u_2,u_{3}}{u_{4}}=0 ). $$
 		\item[$\bullet$] At $\NODE_{1,1,2,3}$, the minimal polynomial is given by  $$\minimalpoly{1,1,2,3}(X)=X \left( X - \frac{1}{u_2 u_3}  \right)^{q-1}- \frac{u_1^q({u_1u_2u_3}-1)}{1-(u_1u_2u_3)^{q+1} }.$$  The function field  $\dot{\mathcal{F}}_{ 1,1,2,3} =\dot{\mathcal{F}}_{1,1,2}(u_4)$ is determined  a solution $u_4$ to $\minimalpoly{1,1,2,3}(u_4)=0$.
 		\item[$\bullet$] At $\NODE_{1,1,2,4}$, the minimal polynomial is given by  $\minimalpoly{1,1,2,4} =\etausingle{u_2,u_3} $, and the function field  $\dot{\mathcal{F}}_{ 1,1,2,4} =\dot{\mathcal{F}}_{1,1,2}(u_4)$ is determined by   $$u_4=\frac{1}{u_2u_3}  \quad (\mbox{i.e.~}~~  \etau{u_2,u_{3}}{u_{4}}=0 ),$$ and thus $\dot{\mathcal{F}}_{1,1,2,4}=\dot{\mathcal{F}}_{1,1,2}$.
 		\item[$\bullet$] At $\NODE_{1,2,1,1}$, the minimal polynomial is given by  $\minimalpoly{1,2,1,1}=\xiusingle{u_2,u_3} $, and the function field  $\dot{\mathcal{F}}_{ 1,2,1,1} =\dot{\mathcal{F}}_{1,2,1}(u_4)$ is determined by   $$u_4(u_4-\frac{1}{u_2u_3})^{q-1}=u_2 \quad (\mbox{i.e.~}~~  \xiu{u_2,u_{3}}{u_{4}}=0 ). $$
 		\item[$\bullet$] At $\NODE_{1,2,1,2}$, the minimal polynomial is given by  $\minimalpoly{1,2,1,2}=\gammausingle{u_3} $, and the function field  $\dot{\mathcal{F}}_{ 1,2,1,2 } =\dot{\mathcal{F}}_{1,2,1}(u_4)$ is determined by
 		$$ u_4(u_4^{q+1}+\frac{u_4}{u_3^{q+1}}+\frac{1}{u_3})^{q-1}=u_3  \quad
 		(\mbox{i.e.~}~~  \gammau{u_3}{u_4}=0 ). $$
 		\item[$\bullet$] At $\NODE_{1,2,1,3}$, the minimal polynomial is given by  $\minimalpoly{1,2,1,3}=\etausingle{u_2,u_3} $, and the function field  $\dot{\mathcal{F}}_{ 1,2,1,3} =\dot{\mathcal{F}}_{1,2,1}(u_4)$ is determined by  $$u_4=\frac{1}{u_2u_3}  \quad (\mbox{i.e.~}~~  \etau{u_2,u_{3}}{u_{4}}=0 ),$$
 		and thus $\dot{\mathcal{F}}_{ 1,2,1,3}=\dot{\mathcal{F}}_{ 1,2,1}$.
 		
 		\item[$\bullet$] At $\NODE_{1,2,2,1}$, the minimal polynomial is given by  $$\minimalpoly{1,2,2,1}(X)=X\Bigl(X- (u_3-\frac{1}{u_1u_2})^{q-1} \frac{1}{u_1u_2}\Bigr)^{q-1}-u_1 \Bigl(u_1-(\frac{1}{u_1u_2}-u_3)^{q-1} u_3\Bigr)^{q-1} .$$ The function field  $\dot{\mathcal{F}}_{ 1,2,2,1} =\dot{\mathcal{F}}_{1,2,2}(u_4)$ is determined by $u_4$ subject to $\minimalpoly{1,2,2,1}(u_4)=0$.
 		\item[$\bullet$] At $\NODE_{1,2,2,2}$, the minimal polynomial is given by  $\minimalpoly{1,2,2,3}=\gammausingle{u_3} $, and the function field  $\dot{\mathcal{F}}_{ 1,2,2,2} =\dot{\mathcal{F}}_{1,2,2}(u_4)$ is determined by   $$u_4(u_4^{q+1}+\frac{u_4}{u_3^{q+1}}+\frac{1}{u_3})^{q-1}=u_3  \quad
 		(\mbox{i.e.~}~~  \gammau{u_3}{u_4}=0 ).$$
 		\item[$\bullet$] At $\NODE_{1,2,2,3}$, the minimal polynomial is given by  $$\minimalpoly{1,2,2,3}(X)=X-(u_3-\frac{1}{u_1u_2})^{q-1} \frac{1}{u_1u_2} .$$ Hence the function field  $\dot{\mathcal{F}}_{ 1,2,2,3} =\dot{\mathcal{F}}_{1,2,2}(u_4) $ is determined by $u_4=(u_3-\frac{1}{u_1u_2})^{q-1} \frac{1}{u_1u_2}$, and thus  $\dot{\mathcal{F}}_{ 1,2,2,3}=\dot{\mathcal{F}}_{1,2,2}$.
 		\item[$\bullet$] At $\NODE_{1,2,3,1}$, the minimal polynomial is given by  $\minimalpoly{1,2,3,1}=\etausingle{u_2,u_3} $, and the function field  $\dot{\mathcal{F}}_{ 1,2,3,1} =\dot{\mathcal{F}}_{1,2,3}(u_4)=\dot{\mathcal{F}}_{1,2}(u_4)$ is determined by
 		$$u_4=\frac{1}{u_2u_3}  \quad (\mbox{i.e.~}~~  \etau{u_2,u_{3}}{u_{4}}=0 ),$$
 		and thus $ \dot{\mathcal{F}}_{1,2,3,1}=\dot{\mathcal{F}}_{1,2,3}=\dot{\mathcal{F}}_{1,2}$.
 		\item[$\bullet$] At $\NODE_{1,2,3,2}$, the minimal polynomial is given by  $\minimalpoly{1,2,3,2}=\xiusingle{u_2,u_3} $, and the function field  $\dot{\mathcal{F}}_{ 1,2,3,2} =\dot{\mathcal{F}}_{1,2,3}(u_4)=\dot{\mathcal{F}}_{1,2}(u_4)$ is determined by    $$u_4(u_4-\frac{1}{u_2u_3})^{q-1}=u_2  \quad (\mbox{i.e.~}~~  \xiu{u_2,u_{3}}{u_{4}}=0 ).$$
 		\item[$\bullet$] At $\NODE_{1,2,3,3}$, the minimal polynomial is given by  $\minimalpoly{1,2,3,3}=\gammausingle{u_3} $, and the function field  $\dot{\mathcal{F}}_{ 1,2,3,3} =\dot{\mathcal{F}}_{1,2,3}(u_4)=\dot{\mathcal{F}}_{1,2}(u_4)$ is determined by   $$ u_4(u_4^{q+1}+\frac{u_4}{u_3^{q+1}}+\frac{1}{u_3})^{q-1}=u_3  \quad
 		(\mbox{i.e.~}~~  \gammau{u_3}{u_4}=0 ).$$
 	\end{itemize}

 	
 Finally, let us	    examine Theorem \ref{Thm:NodeToFactor} at the following nodes: 
 \begin{enumerate}
 	\item At $\NODE_{1    }$, it is direct to check that the modular polynomial  	 
 	 $\kappausingle{u_1} $ decomposes (in $\dot{\mathcal{F}}_{1}[X]$): \begin{eqnarray*}
 		\kappau{u_1}{X}&=&\minimalpoly{1,1}(X)\minimalpoly{1,2}(X)=\gammau{u_1}{X}\betau{u_1}{X}\\
 		&=&\Bigl(X\big( X^{q+1}+\frac{X}{u_1^{q+1}} +\frac{1}{u_1} \big)^{q-1}-u_1\Bigr) ( X^{q+1}+\frac{X}{u_1^{q+1}} +\frac{1}{u_1} ).
 	\end{eqnarray*}
 	 \item At $\NODE_{1,1    }$, the modular polynomial   decomposes (in $\dot{\mathcal{F}}_{1,1}[X]$): $$\kappausingle{u_2} =\minimalpoly{1,1,1} \minimalpoly{1,1,2} = \gammausingle{u_2} \betausingle{u_2}. $$
 	\item At $\NODE_{1,2}$, we have the decomposition in $\dot{\mathcal{F}}_{1,2}[X]$:\begin{eqnarray*}\kappausingle{u_2}
 		&=&\minimalpoly{1,2,1} \minimalpoly{1,2,2} \minimalpoly{1,2,3} =\xiusingle{u_1,u_{2}} \gammausingle{u_2} \etausingle{u_1,u_{2}}.  
 	\end{eqnarray*}
 
 	\item At   $\NODE_{1,1,2}$ and in $\dot{\mathcal{F}}_{1,1,2}[X]$  we have \begin{eqnarray*}
 	\kappausingle{u_3} &=&    \minimalpoly{1,1,2,1}\minimalpoly{1,1,2,2}\minimalpoly{1,1,2,3}\minimalpoly{1,1,2,4}\\ &=&\gammausingle{u_3} \xiusingle{u_2,u_{3}} \etausingle{u_2,u_{3}} .
 \end{eqnarray*}
 Here, we have used the  decomposition $$\gammausingle{u_3} =\minimalpoly{1,1,2,1}\minimalpoly{1,1,2,3}$$    which  is nontrivial and can be verified by a lengthy computation. 
 
 	\item At $\NODE_{1,1,1}$ and in $\dot{\mathcal{F}}_{1,1,1}[X]$, we have
 	\begin{equation}\nonumber
 	\kappausingle{u_3} =\minimalpoly{1,1,1,1}\minimalpoly{1,1,1,2} = \gammausingle{u_3} \betausingle{u_3}    .\end{equation}
 	
 		\item At $\NODE_{1,2,2}$ and in $\dot{\mathcal{F}}_{1,2,2}[X]$, we have
 	\begin{equation}\nonumber
 	\kappausingle{u_3} =\minimalpoly{1,2,2,1}\minimalpoly{1,2,2,2}\minimalpoly{1,2,2,3}= \betausingle{u_3}  \gammausingle{u_3}  .\end{equation}
 	Here, we can prove the decomposition
 	\begin{equation*} 
 	\betausingle{u_3} = \minimalpoly{1,2,2,1}\minimalpoly{1,2,2,3}
 	\end{equation*}  which is also   nontrivial.

 	\item At $\NODE_{1,2,3}$, we can    decomposes: $$\kappausingle{u_3} =\minimalpoly{1,2,3,1}\minimalpoly{1,2,3,2}\minimalpoly{1,2,3,3}= \etausingle{u_2,u_{3}} \xiusingle{u_2,u_{3}} \gammausingle{u_3} $$
 in $\dot{\mathcal{F}}_{1,2,3}[X]$. A similar decomposition happens at $\NODE_{1,2,1}$.

 \end{enumerate}	
These minimal polynomials and normalized Drinfeld modular curves, discussed abstractly in the previous section, will be further analyzed with more details in our upcoming work. It should be noted that the processes to obtain them are complex and rely heavily on the specific nodes being considered. Moreover, these results   can be extended to construct Drinfeld modular curves through the use of Proposition \ref{Prop:fromutow} and the method discussed subsequently.


 \appendix
 \section{List of matrices}\label{Sec:listofmatrices}
 In this part, we list all conjugacy classes of $4\times 4$ upper-triangular matrices. Each of the matrix $R_{a,b,c,d}$ below is a representatives of the corresponding conjugacy class. The whole class it represents is also shown.
 
 \begin{enumerate}
 	\item[$\bullet$]
 	The first one is
 	$$
 	R_{1,1,1,1}= \left(  \begin{array}{cccc}0 & 1 &0 &0 \\0 & 0 & 1&0 \\0 & 0 & 0&1\\0 &0&0&0 \end{array} \right)  \sim  \left(  \begin{array}{cccc}0 & r &x &y \\0 & 0 & s&z \\0 & 0 & 0&t\\0 &0&0&0 \end{array} \right)  . $$
 	It has rank $3$, and  exponent  $4$.  Here and in the sequel,  letters  $r$, $s$, $t$, $x$, $y$, and $z$ denote arbitrary   constants  in $\Fq $ with $r\neq 0$, $s\neq 0$ and $t\neq 0$.
 	\item[$\bullet$]
 	$$  R_{1,1,1,2}= \left(  \begin{array}{cccc}0 & 1 &0 &0 \\0 & 0 & 1&0 \\0 & 0 & 0&0\\0 &0&0&0 \end{array} \right)  \sim  \left(  \begin{array}{cccc}0 & r &x &y \\0 & 0 & s&z \\0 & 0 & 0&0\\0 &0&0&0 \end{array} \right)  .$$
 	(Rank $=2$, Exponent $=3$) \item[$\bullet$]
 	$$  R_{1,1,2,1}= \left(  \begin{array}{cccc}0 & 1 &0 &0 \\0 & 0 & 0&1 \\0 & 0 & 0&1\\0 &0&0&0 \end{array} \right)  \sim  \left(  \begin{array}{cccc}0 & r &x &y \\0 & 0 & 0&z \\0 & 0 & 0&t\\0 &0&0&0 \end{array} \right)   \quad  (\mbox{ subject to } rz+tx\neq 0).
 	$$
 	(Rank $=2$, Exponent $=3$)
 	\item[$\bullet$]
 	$$
 	R_{1,1,2,2}= \left(  \begin{array}{cccc}0 & 1 &0 &0 \\0 & 0 & 0&1 \\0 & 0 & 0&0\\0 &0&0&0 \end{array} \right)      \sim  \left(  \begin{array}{cccc}0 & r &x &y \\0 & 0 & 0&s \\0 & 0 & 0&0\\0 &0&0&0 \end{array} \right)  .$$
 	(Rank $=2$, Exponent $=3$)
 	\item[$\bullet$]
 	$$  R_{1,1,2,3}= \left(  \begin{array}{cccc}0 & 1 &0 &0 \\0 & 0 & 0&0 \\0 & 0 & 0&1\\0 &0&0&0 \end{array} \right)  \sim  \left(  \begin{array}{cccc}0 & r &x &y \\0 & 0 & 0& -\frac{tx}{r} \\0 & 0 & 0&t\\0 &0&0&0 \end{array} \right)   .$$ 
 	(Rank $=2$, Exponent $=2$)
 	\item[$\bullet$]
 	$$  R_{1,1,2,4}= \left(  \begin{array}{cccc}0 & 1 &0 &0 \\0 & 0 & 0&0 \\0 & 0 & 0&0\\0 &0&0&0 \end{array} \right)  \sim  \left(  \begin{array}{cccc}0 & r &x &y \\0 & 0 & 0&0 \\0 & 0 & 0& 0\\0 &0&0&0 \end{array} \right)  .
 	$$
 	(Rank $= 1$, Exponent $=2$)

 	\item[$\bullet$]
 	$$
 	R_{1,2,1,1}= \left(  \begin{array}{cccc}0 & 0 &1 &0 \\0 & 0 & 0&1 \\0 & 0 & 0&0\\0 &0&0&0 \end{array} \right)  \sim  \left(  \begin{array}{cccc}0 & 0 &r &x \\0 & 0 & 0&s \\0 & 0 & 0&0\\0 &0&0&0 \end{array} \right)  .$$
 	(Rank $=2$, Exponent $=2$)
 	
 	\item[$\bullet$]
 	$$  R_{1,2,1,2}= \left(  \begin{array}{cccc}0 & 0 &1 &0 \\0 & 0 & 0&0 \\0 & 0 & 0&1\\0 &0&0&0 \end{array} \right)  \sim  \left(  \begin{array}{cccc}0 & 0 &r &x \\0 & 0 & 0&y \\0 & 0 & 0&s\\0 &0&0&0 \end{array} \right)   .$$
 	(Rank $=2$, Exponent $=3$)

 	\item[$\bullet$]
 	$$  R_{1,2,1,3}= \left(  \begin{array}{cccc}0 & 0 &1 &0 \\0 & 0 & 0&0 \\0 & 0 & 0&0\\0 &0&0&0 \end{array} \right)   \sim  \left(  \begin{array}{cccc}0 & 0 &r &x \\0 & 0 & 0&0 \\0 & 0 & 0&0\\0 &0&0&0 \end{array} \right) .
 	$$
 	(Rank $=1$, Exponent $=2$)

 	\item[$\bullet$]
 	$$
 	R_{1,2,2,1}= \left(  \begin{array}{cccc}0 & 0 &0 &1 \\0 & 0 & 1&0 \\0 & 0 & 0&0\\0 &0&0&0 \end{array} \right)  \sim  \left(  \begin{array}{cccc}0 & 0 &x &s \\0 & 0 & r&z \\0 & 0 & 0&0\\0 &0&0&0 \end{array} \right)  \quad ~(\mbox{ subject to }xz\neq sr).$$
 	(Rank $=2$, Exponent $=2$)
 	
 	\item[$\bullet$]
 	$$  R_{1,2,2,2}= \left(  \begin{array}{cccc}0 & 0 &0 &0 \\0 & 0 & 1&0 \\0 & 0 & 0&1\\0 &0&0&0 \end{array} \right)  \sim  \left(  \begin{array}{cccc}0 & 0 &x & y \\0 & 0 & r&z \\0 & 0 & 0&s\\0 &0&0&0 \end{array} \right)  .$$
 	(Rank $=2$, Exponent $=3$)
 	\item[$\bullet$]
 	$$  R_{1,2,2,3}= \left(  \begin{array}{cccc}0 & 0 &0 &0 \\0 & 0 & 1&0 \\0 & 0 & 0&0\\0 &0&0&0 \end{array} \right)  \sim  \left(  \begin{array}{cccc}0 & 0 &x &\frac{xz}{r}\\0 & 0 & r&z \\0 & 0 & 0&0\\0 &0&0&0 \end{array} \right)  .
 	$$
 	(Rank $=1$, Exponent $=2$)
 	
 	\item[$\bullet$]
 	$$
 	R_{1,2,3,1}= \left(  \begin{array}{cccc}0 & 0 &0 &1 \\0 & 0 & 0&0 \\0 & 0 & 0&0\\0 &0&0&0 \end{array} \right)  \sim  \left(  \begin{array}{cccc}0 & 0 &0 &r \\0 & 0 & 0&0 \\0 & 0 & 0&0\\0 &0&0&0 \end{array} \right)  .$$
 	(Rank $=1$, Exponent $=2$)

 	\item[$\bullet$]
 	$$  R_{1,2,3,2}= \left(  \begin{array}{cccc}0 & 0 &0 &0 \\0 & 0 & 0&1 \\0 & 0 & 0&0\\0 &0&0&0 \end{array} \right)   \sim  \left(  \begin{array}{cccc}0 & 0 &0 &x \\0 & 0 & 0&r \\0 & 0 & 0&0\\0 &0&0&0 \end{array} \right) .$$
 	(Rank $=1$, Exponent $=2$)
 	
 	\item[$\bullet$]
 	$$  R_{1,2,3,3}= \left(  \begin{array}{cccc}0 & 0 &0 &0 \\0 & 0 & 0&0 \\0 & 0 & 0&1\\0 &0&0&0 \end{array} \right)  \sim  \left(  \begin{array}{cccc}0 & 0 &0 &x \\0 & 0 & 0&y \\0 & 0 & 0&r\\0 &0&0&0 \end{array} \right)  .
 	$$
 	(Rank $=1$, Exponent $=2$)
 	
 	\item[$\bullet$]
 	$$
 	R_{1,2,3,4}= \left(  \begin{array}{cccc}0 & 0 &0 &0 \\0 & 0 & 0&0 \\0 & 0 & 0&0\\0 &0&0&0 \end{array} \right)   . $$(Rank $=0$, Exponent $=1$)
 \end{enumerate}

 \bibliographystyle{90}
 \bibliography{papers}

\end{document}